\documentclass[a4paper,reqno]{amsart}
\usepackage{latexsym,amsmath,amssymb,amsfonts}
\usepackage{a4wide}
\usepackage{setspace}
\usepackage[pdftex,bookmarks]{hyperref} 

\theoremstyle{plain}
\begingroup
\newtheorem{theorem}{Theorem}[section]
\newtheorem{lemma}[theorem]{Lemma}
\newtheorem{proposition}[theorem]{Proposition}
\newtheorem{corollary}[theorem]{Corollary}
\endgroup
\theoremstyle{definition}
\begingroup
\newtheorem{definition}[theorem]{Definition}
\newtheorem{remark}[theorem]{Remark}
\newtheorem{example}[theorem]{Example}
\endgroup
\theoremstyle{remark}

\mathchardef\emptyset="001F

\numberwithin{equation}{section}

\newcommand{\e}{\varepsilon}
\newcommand{\vphi}{\varphi}
\newcommand{\B}{{\mathbf B}}
\newcommand{\Bpsi}{{\mathbf B}^\psi}
\newcommand{\C}{{\mathbb C}}
\newcommand{\R}{{\mathbb R}}
\newcommand{\N}{\mathbb N}

\newcommand{\M}{\mathbb M}

\newcommand{\hn}{{\mathcal H}^{N-1}}

\renewcommand{\div}{{\rm div}}
\newcommand{\tr}{{\rm trace\,}}

\newcommand{\setmeno}{\!\setminus\!}

\newcommand{\vf}{v_\varphi}
\newcommand{\Htilde}{{\widetilde H}^1_\#}
\newcommand{\Vtilde}{\widetilde {\mathcal{V}}}
\newcommand{\jacg}{J_{\Phi_g}}
\newcommand{\jacn}{J_{\Phi_{g_n}}}
\newcommand{\p}{\mathcal{P}}

\title[Stability for elastic films in two and three dimensions]{Stability of equilibrium configurations for elastic films in two and three dimensions}
\author{M.\ Bonacini}
\address[M.~Bonacini]{SISSA, Via Bonomea 265, 34136 Trieste, Italy}
\email{marco.bonacini@sissa.it}

\thanks{Preprint SISSA 38/2013/MATE}

\begin{document}

\begin{abstract}
We establish a local minimality sufficiency criterion, based on the strict positivity of the second variation, in the context of a variational model for the epitaxial growth of elastic films. Our result holds also in the three-dimensional case and for a general class of nonlinear elastic energies. Applications to the study of the local minimality of flat morphologies are also shown.
\end{abstract}

\maketitle


\begin{section}{Introduction}

In the last few years morphological instabilities of interfaces
in systems governed by the competition between volume and surface energies
have been the subject of investigation of several studies.
Such instabilities occur, for instance,
in the mechanism of the epitaxial growth of an elastic film on a relatively thick substrate,
in presence of a mismatch between the lattice structures of the two crystalline solids.
A threshold effect, known as the Asaro-Grinfeld-Tiller (AGT) instability,
characterizes the observed configurations:
after reaching a critical value of the thickness, a flat layer becomes morphologically unstable,
and typically the free surface starts to develop irregularities (see, for instance, \cite{Grin}).

In this paper we continue the rigorous mathematical investigation of this phenomenon started in \cite{BC},
where the existence of minimizing configurations for a two-dimensional variational model
is established in the framework of linearized elasticity.
In \cite{FFLM} a regularity theory for minimizers is developed,
while qualitative properties of equilibrium configurations are studied in \cite{FM} by means of a new local minimality criterion based on the positivity of the second variation of the total energy of the system: in particular, an analytical study of local and global minimality of the flat configuration is carried out in two-dimensions and for the linear elastic case.
We mention also the related papers \cite{Bon}, where anisotropic surface energies are taken into consideration,
and \cite{FFLM2}, which deals with the evolution by surface diffusion of epitaxially strained films.

In the present work we aim at extending the sufficiency minimality criterion introduced in \cite{FM}
to the physically relevant three-dimensional case and to a larger class of nonlinear elastic energies,
which appear in the context of Finite Elasticity.
In addition, as it was done in \cite{Bon}, we will take into account anisotropic surface energies, that is, we will allow the surface term in the total energy to depend on the orientation of the normal to the free surface.

To be more precise, the functional under consideration is defined over pairs $(h,u)$, where $h:\R^{N-1}\to\R$ is a positive, periodic function whose subgraph $\Omega_h$ represents the reference configuration of the film, and $u:\Omega_h\to\R^N$ is a deformation of the reference configuration.
A Dirichlet boundary condition is imposed on the function $u$ at the interface between the film and the flat substrate, forcing the film to be elastically stressed.
The total energy of a pair $(h,u)$ takes the form
$$
F(h,u) = \int_{\Omega_h} W(\nabla u)\,dz + \int_{\Gamma_h} \psi(\nu)\,d\hn,
$$
where $\Gamma_h$ denotes the free surface of $\Omega_h$ (that is, the graph of $h$), $\nu$ is the unit normal to $\Gamma_h$, and $W$ and $\psi$ are the (nonlinear) elastic energy density and the (anisotropic) surface energy density, respectively.
Here the surface tension $\psi$ is assumed to be regular and to satisfy a uniform ellipticity condition (see Section~\ref{sect:settings} for more details).
We say that a pair $(h,u)$ is a \emph{strong local minimizer} for $F$ if $(h,u)$ minimizes the functional among all competitors $(g,v)$ such that $g$ is in a small $L^\infty$-neighborhood of $h$ and satisfies the volume constraint $|\Omega_g|=|\Omega_h|$, and the gradients of the deformations $\nabla u$, $\nabla v$ are close in $L^\infty$.
Necessary conditions for local minimality are the first order conditions
\begin{align} \label{intro}
\left\{
  \begin{array}{ll}
    \div(DW(\nabla u)) = 0 & \hbox{in }\Omega_h, \\
    DW(\nabla u)[\nu] = 0 & \hbox{on }\Gamma_h, \\
    W(\nabla u) + H^\psi = \hbox{const} & \hbox{on }\Gamma_h,
  \end{array}
\right.
\end{align}
where $H^\psi$ denotes the anisotropic mean curvature of $\Gamma_h$.

In the main result of the paper we provide a sufficient condition for a \emph{critical pair} (that is, a pair $(h,u)$ satisfying \eqref{intro}) to locally minimize the total energy:
precisely, we show that \emph{any regular critical configuration with strictly positive second variation
is a strong local minimizer for $F$}, according to the previous definition (Theorem~\ref{cor:locmin}).
We also prove a stronger result in the case of linear elasticity (see Theorem~\ref{teo:linelas}), namely we replace the $L^\infty$-closeness of the deformation gradients appearing in the definition of local minimizer by a uniform bound on the Lipschitz constant of the deformations.

Although the question whether strict stability implies local minimality is very classical for the standard functionals of the Calculus of Variations, its investigation in the context of free-discontinuity problems has been started only in recent years: in particular, in addition to \cite{FM}, we refer to \cite{CMM, BonMor}, which deal with the Mumford-Shah functional, to \cite{AFM} for a nonlocal isoperimetric problem arising in the modeling of microphase separation in diblock copolymers, and to \cite{CJP} for a variational model dealing with cavities in elastic bodies.

Our minimality criterion can be applied to the study of the local minimality of flat morphologies,
when the amount of material deposited is small.
We will also prove the interesting fact, firstly observed in \cite{Bon}, that for crystalline anisotropies, whose Wulff shape contains a flat horizontal facet, the AGT instability is suppressed, that is the flat configuration is always a local minimizer,
no matter how thick the film is.

We also mention that our result could be useful to deal with the three-dimensional version of the elastic film evolution by surface diffusion with curvature regularization, studied in \cite{FFLM2} in the two-dimensional case. In particular, it is a natural question in this context to ask whether the strict positivity of the second variation guarantees the Lyapunov stability with respect to this evolution; we think that our criterion could be instrumental in establishing such a result.

One of the crucial difficulties that arise when treating the three-dimensional case is the lack of a regularity theory for minimizers, which prevents us to extend completely the results of \cite{FM}. This is the reason why the minimality property that we are able to prove is weaker than the one considered in \cite{FM}, as it requires the $L^\infty$-closeness of the deformation gradients (or a bound on the Lipschitz constant of the deformation in the linear elastic case). While this constraint seems to be not too restrictive in the nonlinear case, we expect that in the linearized framework the local minimality should hold without such a condition; however, our strategy to improve the result in this direction needs a regularity theory which is not yet available in three dimensions.

\smallskip
We now describe with some additional details the strategy leading to our main result.
We first introduce the notion of admissible variation of a critical pair $(h,u)$,
by considering the deformed profiles $h_t:= h + t\phi$, for $t\in\R$, where $\phi\in C^\infty(\R^{N-1})$ is any periodic function with zero mean value. One of the difficulties which arise in the nonlinear context is the issue of the existence of a critical point for the elastic energy in the deformed domain $\Omega_{h_t}$ (that is, a deformation satisfying the first two conditions of \eqref{intro} in $\Omega_{h_t}$).
Nevertheless, by the Implicit Function Theorem we show that, if the elastic second variation at $u$ is uniformly positive in $\Omega_h$ (see condition \eqref{c0}), it is possible to find a critical point $u_g$ for the elastic energy in $\Omega_g$ (which in addition locally minimizes the elastic energy), provided that $g$ is sufficiently close to $h$ in the $W^{2,p}$-topology (see Proposition~\ref{IFT} and Proposition~\ref{prop:minEl}). This allows us to consider a one-parameter family of variations $(h_t,u_{h_t})$ and to define the \emph{second variation of the functional} at the critical pair $(h,u)$ along the direction $\phi$ as the second derivative at $t=0$ of the map $t\mapsto F(h_t,u_{h_t})$.

The explicit computation of the second variation, performed in Theorem~\ref{th:var2}, will show that it can be expressed in terms of a nonlocal quadratic form $\partial^2F(h,u)$ defined on the space $\Htilde(\Gamma_h)$ of the periodic functions $\vphi\in H^1(\Gamma_h)$ such that $\int_{\Gamma_h}\vphi\,d\hn=0$. Then the strict stability condition reads as
\begin{equation} \label{intro2}
\partial^2F(h,u)[\vphi]>C\|\vphi\|_{H^1(\Gamma_h)}^2 \qquad\text{for every }\vphi\in\Htilde(\Gamma_h).
\end{equation}

The proof of the sufficiency of \eqref{intro2} for strong local minimality is inspired by the two-steps strategy devised in \cite{FM}. Firstly, we show that condition \eqref{intro2} is sufficient, in dimension $N=2,3$, for a weaker notion of local minimality, namely with respect to competitors $(g,v)$ with $\|g-h\|_{W^{2,p}}$ sufficiently small.
Since the expression of the second variation involves the trace of the gradient of $W(\nabla u)$ on $\Gamma_h$, a crucial point in the proof of this result consists in controlling this term in a proper Sobolev space of negative fractional order.
We overcome this difficulty by proving careful new estimates for the elliptic system associated with the first variation of the elastic energy in Lemma~\ref{lem:algebra}, which provides a highly non-trivial generalization to the three-dimensional and nonlinear cases of the estimates proved in \cite[Lemma~4.1]{FM}.

The second part of the proof consists in showing that, in any dimension, the aforementioned weaker notion of minimality implies the desired strong local minimality. This is obtained by a contradiction argument: assuming the existence of a sequence $(g_n,v_n)$ converging to $(h,u)$ and violating the minimality of $(h,u)$, one replaces $(g_n,v_n)$ by a new pair $(k_n,w_n)$ selected as solution to a suitable penalized minimum problem, whose energy is still below the energy of $(h,u)$.
Due to minimality, the pairs $(k_n,w_n)$ enjoy better regularity properties: since the $L^\infty$-bound on the deformation gradients allows us to regard the elastic energy as a volume perturbation of the surface area, we may appeal to the regularity theory for quasi-minimizers of the area functional to deduce the $C^{1,\alpha}$-convergence of $k_n$ to $h$. In turn, with the aid of the Euler-Lagrange equations for the minimum problem solved by $(k_n,w_n)$ we obtain the $W^{2,p}$-convergence of $k_n$ to $h$, and we reach a contradiction to the local minimality of $(h,u)$ with respect to $W^{2,p}$-perturbations established in the first step of the proof.

\smallskip
The paper is organized as follows.
We introduce the variational model and the basic definitions in Section~\ref{sect:settings}.
As pointed out in the previous discussion, we need to find deformations which locally minimize the elastic energy in th perturbed reference configurations: this is done in Section~\ref{sect:critel}.
The explicit computation of the second variation is carried out in Section~\ref{sect:varII}, where we also prove two different, equivalent formulations of condition \eqref{intro2}.
In Section~\ref{sect:locmin} we start the proof of the main result of the paper, showing that the strict stability of a critical pair implies local minimality in the $W^{2,p}$-sense; in Section~\ref{sect:WimpliesL} we prove that, in any dimension, local $W^{2,p}$-minimizers are strong local minimizers, and we show how the results can be strengthen in the linear elastic case.
Section~\ref{sect:flat} is devoted to the study of the stability of flat morphologies.
In the final Appendix we collect some auxiliary results that are needed in the rest of the paper.

\end{section}


\begin{section}{Setting of the problem} \label{sect:settings}

In this section we introduce the notation used in the paper
and we describe the setting of the variational problem that we consider.

\subsection{General notation}
We denote by $\M^{N}$ the space of $N\times N$ real matrices
and by $\M^{N}_+$ its subset of matrices with positive determinant.
The scalar product in $\M^N$ is defined by
$A:B := {\rm trace}\,(A^TB)$,
where $A^T$ is the transpose of $A$,
and we denote by $|A|$ the associated euclidean norm.
The symbol $I$ stands for the identity matrix, while $Id:\R^N\to\R^N$ denotes the identity map.
We also deal with \emph{fourth order tensors},
which are linear transformations of the space $\M^{N}$ into itself.
We denote the action of such a tensor $C$ on a matrix $M$ by $CM$.

We write every vector $z \in \mathbb{R}^N$, $N\geq2$, as
$z = (x,y)$,
where $x \in \mathbb{R}^{N-1}$ is the orthogonal projection of $z$
on the hyperplane spanned by $\{ e_1, \ldots, e_{N-1} \}$
and $y \in \mathbb{R}$.
Here $e_1,\dots, e_N$ are the vectors of the canonical basis of $\R^N$.
We denote by $\R^N_+:=\{(x,y)\in\R^N:y>0\}$ and $\R^N_-:=\{(x,y)\in\R^N:y<0\}$
the upper and lower half-space, respectively.

Let $Q = (0,1)^{N-1}$ be the unit square in $\R^{N-1}$.
For $p\in[1,+\infty]$ and $k\geq0$, we denote by $W^{k,p}_\#(Q)$
the set of functions $h:\R^{N-1}\to(0,+\infty)$ of class $W^{k,p}_{loc}(\R^{N-1})$
which are one-periodic with respect to all the coordinate directions,
endowed with the norm $\|\cdot\|_{W^{k,p}(Q)}$.
Similarly, $C^k_\#(Q)$ and $C^{k,\alpha}_\#(Q)$, for $\alpha\in(0,1)$,
denote the sets of one-periodic functions $h:\R^{N-1}\to(0,+\infty)$ of class $C^k$ and $C^{k,\alpha}$, respectively.

Given a smooth orientable $(N-1)$-dimensional manifold $\Gamma \subset \R^N$,
we denote by $\nu$ a normal vector field on $\Gamma$.
If $g:\mathcal{U}\to\R^d$ is a smooth vector-valued function defined in a tubular neighborhood $\mathcal{U}$ of $\Gamma$,
we denote by $\nabla_\Gamma g$ its tangential differential (which we identify with a matrix)
and, if $d=N$, by $\div_\Gamma g$ its tangential divergence.
We refer to \cite[Chapter~2, Section~7]{Sim} for the definition of these tangential differential operators
and for some related identities (in particular, we will make use of the divergence formula, which allows to extend to tangential operators the usual integration by parts formula).
For every $x\in\Gamma$ we set
\begin{equation*}
\B(x):= \nabla_{\Gamma}\nu(x) = \nabla\nu(x), \qquad H(x):=\div\, \nu(x) = \div_{\Gamma} \nu(x )=\tr \B(x).
\end{equation*}
The bilinear form associated with $\B(x)$ is symmetric and, when restricted to $T_x \Gamma{\times}T_x \Gamma$,
it coincides with the {\em second fundamental form of $\Gamma$ at $x$},
while the value $H(x)$ coincides with the {\em mean curvature of $\Gamma$ at $x$}.
If $\psi:\R^N\setmeno\{0\}\to(0,+\infty)$ is a smooth, positively 1-homogeneous and convex function,
we define the {\em anisotropic second fundamental form of $\Gamma$}
and the {\em anisotropic mean curvature of $\Gamma$} by
\begin{equation}\label{defHpsi}
\Bpsi :=  \nabla ( \nabla \psi \circ \nu ), \qquad
H^{\psi}:= \text{trace} \, \Bpsi = \div\, ( \nabla \psi  \circ \nu )
\end{equation}
respectively.
Note that, also in this case, we have
$H^{\psi} = \div_{\Gamma} \, ( \nabla \psi  \circ \nu )$ on $\Gamma$.
Finally, if $\Phi:\R^N\to\R^N$ is a smooth orientation-preserving diffeomorphism,
we denote by $J_\Phi:=|(\nabla\Phi)^{-T}[\nu]|\det \nabla\Phi$ the $(N-1)$-dimensional Jacobian of $\Phi$ on $\Gamma$.

\subsection{The variational model}
We now describe the variational model which will be the subject of this work,
bearing in mind the two-dimensional setting introduced in \cite{BC,FM}.
We first introduce the class of \emph{admissible profiles},
given by Lipschitz, strictly positive and periodic functions:
\begin{align*}
AP(Q):= \Bigl\{ h &: \R^{N-1}\to(0,+\infty) \, : \,  h \text{ is Lipschitz continuous, } \\
&h(x+e_i)=h(x) \text{ for every } x\in\R^{N-1} \text{ and } i=1,\ldots,N-1 \Bigr\} \, .
\end{align*}
Given $h \in AP(Q)$, we define the associated reference configuration $\Omega_h$
and its periodic extension $\Omega_h^\#$ to be the sets
\begin{align*}
\Omega_h := \bigl\{ (x,y) \in \R^N : x\in Q, \; 0 < y < h(x) \bigr\},
\quad
\Omega^{\#}_h := \bigl\{ (x,y) \in \R^N : 0 < y < h(x) \bigr\}
\end{align*}
respectively, and the graph $\Gamma_h$ of $h$ and its periodic extension $\Gamma_h^{\#}$,
representing the \emph{free profile},
$$
\Gamma_h := \bigl\{ (x,h(x)) \in \R^N : x\in Q \bigr\},
\quad
\Gamma_h^{\#} := \bigl\{ (x,h(x)) \in \R^N : x\in \R^{N-1} \bigr\}.
$$
We also introduce the following space of admissible elastic variations:
\begin{align*}
\mathcal{V}(\Omega_h) : =
\Bigl\{ w \in W^{1,\infty}(\Omega_h^\#;\R^N) :\,
w(x,0)&=0, \, w(x + e_i, y ) = w(x,y) \\
&\text{ for all } (x,y) \in \Omega_h^\# \text{ and } i=1,\ldots,N-1 \Bigr\}\,,
\end{align*}
and we will denote by $\Vtilde(\Omega_h)$ the completion of $\mathcal{V}(\Omega_h)$ with respect to the norm of $H^1(\Omega_h;\R^N)$.
Since we assume to be in presence of a mismatch strain at the interface $\{y=0\}$,
we prescribe a boundary Dirichlet datum in the form
$$
u_0 (x,y) := (A[x] + q(x),0),
$$
where $A \in \M^{N-1}_+$ and $q:\R^{N-1}\to\R^{N-1}$ is a smooth function, one-periodic with respect to the coordinate directions.
We can finally define the space of \emph{admissible pairs}
\begin{align*}
X=\Bigl\{(h,u)\in AP(Q) \times W^{1,\infty}(\Omega_h^\#;\R^N) \,:\, u - u_0\in\mathcal{V}(\Omega_h),\;
\det\nabla u(z)>0 \text{ for a.e. }z\in\Omega_h \Bigr\} \, .
\end{align*}

In order to introduce the functional on $X$ which represents the total energy of the system,
we define the \emph{elastic energy density} and the \emph{anisotropic surface energy density} to be, respectively:
\begin{itemize}
  \item $W: \M^{N}_+ \to [0, + \infty)$ of class $C^3$,
  \item $\psi: \R^N \to [0, + \infty)$, of class $C^3$ away from the origin, positively 1-homogeneous, such that
      \begin{equation} \label{boundpsi}
      m|z|\leq\psi(z)\leq M|z| \qquad\text{for all }z\in\R^N
      \end{equation}
      for some positive constants $m,M,$ and satisfying the following condition of uniform convexity:  for every $v \in S^{N-1}$
      \begin{equation} \label{pospsi}
      \nabla^2 \psi (v) [w,w] > \bar{c} \, |w|^2 \quad \text{ for all }w \perp v,
      \end{equation}
      for some constant $\bar{c}>0$.
\end{itemize}
Finally, we define the functional on $X$
$$
F (h,u) := \int_{\Omega_h} W (\nabla u)\, dz + \int_{\Gamma_h} \psi (\nu_h) \, d \mathcal{H}^{N-1},
$$
where $\nu_h$ denotes the exterior unit normal vector to $\Omega_h$ on $\Gamma_h$
(we shall omit the subscript $h$ when there is no risk of ambiguity).

\begin{remark}
Although, for the sake of simplicity, we assume that $W$ is defined on the space $\M^N_+$ of the matrices with positive determinant, the results contained in this paper are valid also for a general nonlinear density $W$ of class $C^3$, defined only on an open subset $\mathcal{O}$ of $\M^N$; in this case the space $X$ should be replaced by the following space of admissible pairs:
$$
\Bigl\{(h,u)\in AP(Q) \times W^{1,\infty}(\Omega_h^\#;\R^N) \,:\, u - u_0\in\mathcal{V}(\Omega_h),\;
\nabla u(z)\in\mathcal{O} \text{ for a.e. }z\in\Omega_h \Bigr\} \, .
$$
The physically relevant condition that $W(\xi)\to+\infty$ as $\det\xi\to0^+$, which is customary in Finite Elasticity, is compatible with our assumption.
When $W$ is a quasi-convex function defined on the whole space $\M^N$ and satisfying standard $p$-growth conditions, the definition of the functional $F$ can be extended to a larger class of admissible pairs by a relaxation procedure (see \cite{ChaSol}).
\end{remark}

We will denote the derivatives of $W$ by
$$
W_\xi(\xi) := DW(\xi) = \biggl( \frac{\partial W}{\partial\xi_{ij}}(\xi) \biggr)_{ij}, \qquad
W_{\xi\xi}(\xi) := DW(\xi) = \biggl( \frac{\partial^2 W}{\partial\xi_{ij}\partial\xi_{hk}}(\xi) \biggr)_{ijhk}.
$$
We now give the definitions of critical point for the elastic energy in a given reference configuration $\Omega_h$,
and of critical pair for the functional $F$.

\begin{definition} \label{def:critpoint}
Let $(h,u) \in X$ with $u\in C^1(\overline{\Omega}_h^\#;\R^N)$.
The function $u$ is said to be a \textit{critical point for the elastic energy in $\Omega_h$} if
\begin{equation} \label{VolEul}
\int_{\Omega_h} W_{\xi} (\nabla u): \nabla w \, dz = 0
\qquad\text{for every }w \in \mathcal{V}(\Omega_h).
\end{equation}
Notice that, by periodicity, \eqref{VolEul} is equivalent to
$$
\begin{cases}
\div \left[ W_{\xi} (\nabla u) \right] = 0 & \text{ in }\Omega^{\#}_h, \\
W_{\xi} (\nabla u) [\nu] = 0 & \text{ on } \Gamma_h^{\#}.
\end{cases}
$$
\end{definition}

\begin{definition} \label{def:critpair}
We say that a pair $(h,u)\in X$ is a \textit{(regular) critical pair} for $F$
if $h\in C^2_\#(Q)$, $u\in C^2(\overline{\Omega}_h^\#;\R^N)$ is a critical point for the elastic energy in $\Omega_h$,
and the following condition holds:
\begin{equation} \label{critpair}
W(\nabla u)+ H^\psi = \text{const} \qquad \text{on }\Gamma_h.
\end{equation}
\end{definition}

In the main result of the paper (Theorem~\ref{cor:locmin}) we provide a sufficient condition for a critical pair $(h,u)\in X$
to be a \emph{local minimizer} of the functional $F$ under volume constraint.

The regularity assumptions on a critical pair $(h,u)$ allow us to extend $u$ to a slightly larger domain, preserving the property that the deformation gradient $\nabla u$ has positive determinant.
More precisely, given a critical pair $(h,u)$ we can find an open set $\Omega'$ of the form $\Omega_{h+\eta}$, for some $\eta>0$, with the following property: denoting by $\Omega'_\#$ the periodic extension of $\Omega'$,
we can extend $u$ to a periodic function of class $C^1$ in $\overline{\Omega}'_\#$
in such a way that $\det\nabla u(z)>0$ for every $z\in\overline{\Omega}'$.
This induces us to consider the following class of competitors:
\begin{align} \label{spaziocompet}
X' := \bigl\{ (g,v)\in X \,:\, \Omega_g\subset\Omega',\; v\in W^{1,\infty}(\Omega'_\#;\R^N),\;
\det\nabla v(z)>0 \text{ for a.e. }z\in\Omega' \bigr\}\,.
\end{align}
We then consider the following notion of local minimality.

\begin{definition} \label{def:locmin}
Let $(h,u)\in X$ be a critical pair for $F$.
We say that $(h,u)$ is a \textit{local minimizer} for $F$ if
there exists $\delta>0$ such that
\begin{equation} \label{locmin}
F(h,u)\leq F(g,v)
\end{equation}
for all $(g,v)\in X'$ with $\|g-h\|_{\infty}<\delta$, $|\Omega_g|=|\Omega_h|$, and
$\|\nabla v - \nabla u\|_{L^\infty(\Omega';\M^{N})}<\delta$.
We say that $(h,u)$ is an \textit{isolated local minimizer} if \eqref{locmin} holds with strict inequality when $g \neq h$.
\end{definition}

\begin{remark} \label{rm:diffeo}
The following construction will be used several times throughout the paper.
Given any admissible profile $h\in AP(Q)$,
we associate with every $g \in AP(Q)$ in a sufficiently small $L^\infty$-neighborhood of $h$
a map $\Phi_g: \overline{\Omega}_h^\# \to \overline{\Omega}_g^{\#}$ with the properties:
\begin{itemize}
  \item $\Phi_g(x,0)=(x,0)$ for every $x\in\R^{N-1}$;
  \item $\Phi_g(x,y)=(x,y+g(x)-h(x))$ in a neighborhood of $\Gamma_h^\#$;
  \item $\Phi_g(x+e_i,y)=\Phi_g(x,y)+(e_i,0)$ for $(x,y)\in \overline{\Omega}_h^{\#}$ and $i=1,\ldots,N-1$;
  \item $\Phi_g$ satisfies the following estimate:
    \begin{equation}\label{diffeo0}
    \| \Phi_g - Id \|_{L^\infty (\Omega_h;\R^N)} \leq \| g- h \|_{L^\infty(Q)}.
    \end{equation}
\end{itemize}
We can explicitly construct the diffeomorphism $\Phi_g$ as follows.
Setting $m_0:=\min h>0$, we fix a nonnegative cut-off function
$\rho\in C^\infty_c(-\frac{m_0}{2},\frac{m_0}{2})$ with $\rho\equiv1$ in $(\frac{m_0}{4},\frac{m_0}{4})$.
Then it is easily seen that, if $\|g-h\|_\infty<\frac{m_0}{4}$, the map
$$
\Phi_g(x,y):=\bigl( x , y + \rho(y-h(x))(g(x)-h(x)) \bigr)
$$
satisfies all the previous conditions.
\end{remark}

\begin{remark}
We note here for later use that, as a consequence of the positive 1-homogeneity of the anisotropy $\psi$,
\begin{equation} \label{hesspsi}
\nabla^2 \psi (v) [v] = 0 \quad
\text{ for every }v \in \R^N\setmeno\{0\}.
\end{equation}
Moreover, given a sufficiently regular admissible profile $h$,
we can prove the following explicit formula for the anisotropic mean curvature of $\Gamma_h$
(see \eqref{defHpsi} for the definition):
\begin{equation} \label{Hpsi}
H^\psi(x,h(x)) = \sum_{i=1}^{N-1}\frac{\partial}{\partial x_i} \Bigl( \frac{\partial\psi}{\partial z_i}(-\nabla h(x),1) \Bigr).
\end{equation}
In fact, observe that by \eqref{hesspsi} we have $\nabla^2\psi(-\nabla h,1)[(-\nabla h,1)]=0$,
that is,
$$
\sum_{j=1}^{N-1} \frac{\partial^2\psi}{\partial z_i\partial z_j}(-\nabla h,1)\frac{\partial h}{\partial x_j} = \frac{\partial^2\psi}{\partial z_i\partial z_N}(-\nabla h,1)
$$
for $i=1,\ldots,N$.
Hence, as $\nabla\psi$ is 0-homogeneous, a straightforward computation yields
\begin{align*}
H^\psi(x,h(x)) &= \div_{\Gamma_h}(\nabla\psi\circ\nu)|_{(x,h(x))}
= - \sum_{j,k=1}^{N-1} \frac{\partial^2\psi}{\partial z_j\partial z_k}(-\nabla h,1) \frac{\partial^2 h}{\partial x_k\partial x_j} \\
& + \frac{1}{1+|\nabla h|^2} \sum_{i,k=1}^{N-1} \biggl( \sum_{j=1}^{N-1} \frac{\partial^2\psi}{\partial z_k\partial z_j}(-\nabla h,1)
 \frac{\partial h}{\partial x_j} - \frac{\partial^2\psi}{\partial z_k\partial z_N}(-\nabla h,1) \biggr)
 \frac{\partial^2 h}{\partial x_k\partial x_i}\frac{\partial h}{\partial x_i},
\end{align*}
from which \eqref{Hpsi} follows by using the previous equality.
\end{remark}

\end{section}


\begin{section}{Critical points for the elastic energy}\label{sect:critel}

The purpose of this section is to associate with every $g$ close to $h$ (in some norm, to be specified)
a deformation $u_g$ such that, if $g$ is fixed, the map $v \mapsto F(g,v)$ has a local minimum at $u_g$.
If this is the case, then in order to prove the local minimality of an admissible pair $(h,u)$
it will be sufficient to compare $F(h,u)$
only with the values of $F$ at pairs of the form $(g,u_g)$,
avoiding in some sense the dependence on the second variable.
The Implicit Function Theorem guarantees that this is in fact possible,
under suitable assumptions on the starting pair $(h,u)$.

\begin{definition} \label{def:minEl}
Let $(h,u)\in X$,
and assume that $u$ is a critical point for the elastic energy in $\Omega_h$,
according to Definition~\ref{def:critpoint}.
We say that $u$ is a \textit{strict $\delta$-local minimizer for the elastic energy in $\Omega_h$},
for $\delta > 0$, if
$$
\int_{\Omega_h} W (\nabla u ) \, dz < \int_{\Omega_h} W (\nabla u + \nabla w) \, dz
$$
whenever $w \in \mathcal{V}(\Omega_h)$ and
$0 < \| \nabla w \|_{L^{\infty} ( \Omega_h ; \M^{N})} \leq \delta$.
\end{definition}

We now provide suitable assumptions on a pair $(h,u)$,
with $u$ critical point for the elastic energy in $\Omega_h$,
which guarantee that if $g$ is a small $W^{2,p}$-perturbation of the profile $h$
then we can find a critical point $u_g$ for the elastic energy in $\Omega_g$
which in addition locally minimizes the elastic energy.
In order to do this, we introduce a fourth order symmetric tensor field, associated with a deformation $u$ in a domain $\Omega_h$, setting
$$
C_u (z):= W_{\xi \xi} (\nabla u (z)) \qquad\text{for every }z \in \overline{\Omega}^{\#}_h\,.
$$

\begin{definition} \label{def:stable1}
Let $(h,u)\in X$ .
We say that the elastic second variation is uniformly positive at $u$ in $\Omega_h$
if there exists a positive constant $c_0$ such that
\begin{equation} \label{c0}
\int_{\Omega_h} C_u \nabla w \! : \! \nabla w \, dz  \geq c_0 \| w \|^2_{H^1 (\Omega_h; \R^N)}
\qquad\text{for every } w \in \Vtilde(\Omega_h)\,,
\end{equation}
where we recall that $\Vtilde(\Omega_h)$ denotes the completion of $\mathcal{V}(\Omega_h)$ with respect to the norm of $H^1(\Omega_h;\R^N)$.
\end{definition}

Arguing as in \cite[Theorem 1]{SS87}, it is possible to prove\footnote{In view of the Remark following \cite[Proposition~9.4]{SS87}, our regularity assumptions on $W$ and $(h,u)$ are sufficient to guarantee the validity of the stated result.} the following equivalent formulation of condition \eqref{c0}.

\begin{theorem} \label{teo:simpsonspector}
Let $(h,u)\in X$ be such that $h\in C^2_\#(Q)$ and $u\in C^2(\overline{\Omega}_h^\#;\R^N)$
is a critical point for the elastic energy in $\Omega_h$.
Then \eqref{c0} holds (with some positive constant $c_0$ depending only on the pair $(h,u)$)
if and only if the following three conditions are satisfied:
\begin{itemize}
\smallskip
\item[(H1)] for all $z \in \Omega_h$ the fourth order tensor $C_u (z)$
satisfies the \emph{strong ellipticity condition}, that is
$$
C_u (z) M : M> 0
$$
whenever $M= a\otimes b$ with $a\neq0$, $b\neq0$;
\smallskip

\item[(H2)] for all $z_0 \in\Gamma_h$ the boundary value problem
$$
\begin{cases}
\div \left[ C_u (z_0) \nabla v \right] = 0 &\text{ in }H_{\nu (z_0)}, \\
\left( C_u (z_0) \nabla v \right) [\nu (z_0)] = 0 &\text{ on }\partial H_{\nu (z_0)},
\end{cases}
$$
where
$$
H_{\nu(z_0)} := \{ z \in \mathbb{R}^N : z \cdot \nu (z_0) > 0 \},
$$
satisfies the \emph{complementing condition}, \emph{i.e.}, the only bounded exponential solution to the previous equation is $v\equiv 0$. By \emph{bounded exponential} we mean a solution of the form
$$
v(z)= {\rm Re} \bigl[ f(z\cdot\nu(z_0)) \, e^{i(z\cdot b)} \bigr]
$$
for some $b\in\partial H_{\nu(z_0)}\setmeno\{0\}$ and $f\in C^\infty([0,+\infty), \C^N)$ satisfying $\sup_s|f(s)|<\infty$;
\smallskip

\item[(H3)] the elastic second variation is strictly positive, that is, for every $w \in \Vtilde(\Omega_h) \setmeno\{0\}$
$$
\int_{\Omega_h} C_u \nabla w \! : \! \nabla w \, dz > 0 \,.
$$
\end{itemize}
\end{theorem}

We are now ready to explain
the construction announced at the beginning of this section.

\begin{proposition} \label{IFT}
Let $(h,u) \in X$ be such that $h\in C^2_\#(Q)$, $u\in C^2(\overline{\Omega}_h^\#;\R^N)$
is a critical point for the elastic energy in $\Omega_h$,
and condition \eqref{c0} holds.
Let $p\in(N,+\infty)$.
There exists a neighborhood $\mathcal{U}$ of $h$ in $W^{2,p}_\#(Q)$
and a map $g\in\mathcal{U} \mapsto u_g \in W^{2,p}(\Omega_g ; \mathbb{R}^N)$
such that:
\begin{itemize}
\item[(i)] $u_g$ is a critical point for the elastic energy in $\Omega_g$,
according to Definition~\ref{def:critpoint};
\item[(ii)] $u_h = u$;
\item[(iii)] the map $g \mapsto u_g \circ \Phi_g$ is of class $C^1$
from $W^{2,p}_{\#}(Q)$ to $W^{2,p} (\Omega_h ; \mathbb{R}^N)$.
\end{itemize}
\end{proposition}

\begin{proof}
We start by observing that if $g\in W^{2,p}_{\#}(Q)$ is close to $h$ in the $W^{2,p}$-topology,
the maps $\Phi_g$ introduced in Remark~\ref{rm:diffeo}
are orientation preserving diffeomorphisms of class $W^{2,p}$ satisfying an estimate
\begin{equation}\label{diffeo}
\| \Phi_g - Id \|_{W^{2,p} (\Omega_h;\R^N)} \leq c \, \| g- h \|_{W^{2,p}(Q)}
\end{equation}
for some constant $c > 0$ depending only on $h$.
Moreover, by construction the map $g\mapsto\Phi_g$ is affine, and hence of class $C^\infty$
from a neighborhood of $h$ in $W^{2,p}_\#(Q)$ to $W^{2,p} (\Omega_h;\R^N)$.

Our aim is to associate, with every $g$ in a sufficiently small $W^{2,p}$-neighborhood of $h$, a solution $u_g$ to \eqref{VolEul} with $u_g-u_0\in\mathcal{V}(\Omega_g)$.
A change of variables shows that a function $v$ is a solution to \eqref{VolEul} with $v - u_0\in\mathcal{V}(\Omega_g)$
if and only if the function $\tilde{v}=v\circ\Phi_g - u_0$ belongs to $\mathcal{V}(\Omega_h)$ and solves
\begin{equation} \label{VolEul2}
\int_{\Omega_h} W_\xi \bigl( (\nabla\tilde{v}+\nabla{u_0}) (\nabla\Phi_g)^{-1} \bigr) (\nabla\Phi_g)^{-T}
: \nabla\tilde{w} \, \det\!\nabla\Phi_g \, dz = 0
\qquad \text{for every }\tilde{w}\in \mathcal{V} (\Omega_h).
\end{equation}
Note that an equivalent formulation of \eqref{VolEul2} is
$$
\begin{cases}
\div \left[ Q_{\Phi_g} ( z, \nabla \tilde{v}(z) ) \right] = 0 & \text{ in }\Omega^{\#}_h, \\
Q_{\Phi_g} ( z, \nabla \tilde{v}(z) ) [\nu] = 0 & \text{ on } \Gamma_h^{\#},
\end{cases}
$$
where we set, for $z\in\overline{\Omega}^{\#}_h$ and $M\in\M^{N}$,
\begin{equation} \label{cambiovar}
Q_{\Phi_g} ( z , M ) :=
\det\!\nabla\Phi_g(z) \, W_{\xi}\! \left( (M + \nabla u_0(z)) (\nabla\Phi_g(z))^{-1} \right) (\nabla\Phi_g(z))^{-T}.
\end{equation}
Our strategy will be to get a solution to this boundary value problem by means of the Implicit Function Theorem. To this aim, let us define the open subsets
\begin{align*}
A &:= \bigl\{\Phi\in W^{2,p}(\Omega_h;\R^N) :
\, \det\!\nabla\Phi>0 \text{ in }\Omega_h^{\#}, \, \nabla\Phi(x+e_i,y)=\nabla\Phi(x,y) \\
&\hspace{2cm} \text{ for } (x,y) \in \Omega^{\#}_h \text{ and } i=1,\ldots,N-1 \bigr\} \,,\\
B &:= \{v\in\mathcal{V}(\Omega_h)\cap W^{2,p}(\Omega_h;\R^N): \det(\nabla v + \nabla u_0)>0 \text{ in }\Omega_h\} \,,
\end{align*}
both equipped with the norm $\|\cdot\|_{W^{2,p}(\Omega_h;\R^N)}$
(notice that the pointwise conditions on the determinants in the definition of the spaces $A$ and $B$
make sense thank to the embedding of $W^{2,p}$ in $C^{1,\alpha}$).
Observing that, for $(\Phi,v)\in A\times B$, the map $z\mapsto Q_\Phi(z,\nabla v(z))$
is of class $W^{1,p}$ in $\Omega_h$
(here $Q_\Phi$ is defined as in \eqref{cambiovar} with $\Phi_g$ replaced by $\Phi$),
we introduce the spaces
\begin{align*}
Y_1 &:= \bigl\{ f\in L^p(\Omega_h; \R^N): f(x+e_i,y)=f(x,y)
\text{ for a.e. } (x,y) \in \Omega^{\#}_h \text{ and } i=1,\ldots,N-1 \bigr\},\\
Y_2 &:= \bigl\{ \eta\in W^{1-\frac1p,p}(\Gamma_h;\R^N): \eta(x+e_i,h(x+e_i))=\eta(x,h(x)) \text{ for a.e. } x\in\R^{N-1} \bigr\},
\end{align*}
and the map $G : A \times B \to Y_1\times Y_2$ defined as
$$
G(\Phi,v):= \Bigl(\div \bigl[ Q_{\Phi} ( \cdot, \nabla v(\cdot) ) \bigr], Q_{\Phi} ( \cdot, \nabla v(\cdot) ) [\nu]\Bigr).
$$
It can be checked that $G$ is a map of class $C^1$, and $G(Id, u-u_0)=(0,0)$ (as $u$ solves \eqref{VolEul}). In order to apply the Implicit Function Theorem, we need to verify that the partial derivative $\partial_v G (Id,u-u_0)$ is an invertible bounded linear operator.
Since for every $v \in \mathcal{V}(\Omega_h) \cap W^{2,p}(\Omega_h;\R^N)$
\begin{align*}
\partial_v G(Id,u-u_0)[v] &= \Bigl( \div \left[W_{\xi\xi}(\nabla u)\nabla v\right] , \bigl(W_{\xi\xi}(\nabla u)\nabla v\bigr) [\nu] \Bigr)\\
&= \bigl( \div \left[C_u\nabla v\right], (C_u\nabla v) [\nu] \bigr),
\end{align*}
the invertibility of the operator $\partial_v G(Id,u-u_0)$ corresponds to prove existence and uniqueness in $\mathcal{V}(\Omega_h) \cap W^{2,p}(\Omega_h;\R^N)$ of solutions to the problem
\begin{equation*}
\begin{cases}
\div \left[ C_u\nabla v \right] = f & \text{ in }\Omega_h^\#, \\
(C_u\nabla v) [\nu] = \eta & \text{ on } \Gamma_h^\#,
\end{cases}
\end{equation*}
for any given $(f,\eta)\in Y_1 \times Y_2$.
The proof of this fact relies on the regularity theory for elliptic systems with mixed boundary conditions,
and in particular on the regularity estimates of Agmon, Douglis and Nirenberg (see \cite[Theorem~10.5]{ADN}),
which can be applied thank to the assumption \eqref{c0},
which is equivalent to the three conditions (H1)--(H3) by Theorem~\ref{teo:simpsonspector},
and to the regularity of $h$ and $u$
(we refer also to \cite{Thom} for a clear presentation of the theory in the context of linear elasticity).

We are now in position to apply the Implicit Function Theorem:
there exist a neighborhood $\mathcal{V}$ of $Id$ in $A$, a neighborhood $\mathcal{W}$ of $u-u_0$ in $B$ and a map
$$
\Phi\in\mathcal{V}\longmapsto u_\Phi\in\mathcal{W}
$$
of class $C^1$ such that $u_{Id}=u-u_0$ and $G(\Phi,u_\Phi)=(0,0)$ for all $\Phi\in\mathcal{V}$.
Finally, thank to \eqref{diffeo}, we can determine a neighborhood $\mathcal{U}$ of $h$ in $W^{2,p}_{\#} (Q)$ such that if $g\in\mathcal{U}$ then $\Phi_g\in\mathcal{V}$. Setting $u_g:=(u_{\Phi_g}+u_0)\circ\Phi_g^{-1}$ for any $g\in\mathcal{U}$, we obtain the conclusion of the proposition.
\end{proof}

\begin{remark} \label{rm:IFT}
From the proof of the previous proposition it follows in particular that there exists a compact set $K\subset\M^N_+$ such that
$$
\nabla u_g(z)\in K
\qquad\text{for every }g\in\mathcal{U}\text{ and } z\in\overline{\Omega}_g.
$$
\end{remark}

We conclude this section by showing that the critical points $u_g$ constructed in Proposition~\ref{IFT}
are also local minimizers of the elastic energy, in the sense of Definition~\ref{def:minEl}.

\begin{proposition} \label{prop:minEl}
Let $\mathcal{U}$ be as in Proposition~\ref{IFT}.
There exist $\delta > 0$ and $\varepsilon > 0$ such that, if $g \in \mathcal{U}$ and
$\| g - h \|_{W^{2,p} (Q)} < \e,$
then $u_g$ is a strict $\delta$-local minimizer for the elastic energy in $\Omega_g$,
according to Definition ~\ref{def:minEl}.
\end{proposition}

\begin{proof}
We start by observing that, if $g \in \mathcal{U}$ and
$\| g - h \|_{W^{2,p}(Q)} < \e$,
then from \eqref{c0} and from the smoothness of the map $g \mapsto u_g\circ\Phi_g$
one can easily deduce that
\begin{equation} \label{provv}
\int_{\Omega_g} C_{u_g} \nabla w \! : \! \nabla w \, dz > \frac{c_0}{4} \| w \|^2_{H^1 (\Omega_g; \R^N)}
\end{equation}
for every $w \in \mathcal{V} (\Omega_g)$, provided  $\e>0$ is small enough.

Let now $w \in \mathcal{V}(\Omega_g)$ satisfy
$0 < \| \nabla w \|_{L^{\infty} ( \Omega_g ; \M^{N})} \leq \delta$,
with $\delta > 0$ to be chosen.
We set
$$
f(t):= \int_{\Omega_g} W (\nabla u_g + t \nabla w) \, dz, \quad t \in [0,1].
$$
Notice that, since $u_g$ is a critical point, $f'(0) = 0$.
Hence, there exists $\tau \in (0,1)$ such that
\begin{align}
\int_{\Omega_g} W (\nabla u_g + \nabla w) \, dz
&= f(1) = f (0) +  \frac{f'' (\tau)}{2} \nonumber \\
&= \int_{\Omega_g} W (\nabla u_g) + \frac12 \int_{\Omega_g} C_{u_g} [ \nabla w , \nabla w ] \, dz \nonumber \\
&\hspace{0.5cm}+ \frac12 \int_{\Omega_g} \left( W_{\xi \xi} ( \nabla u_g + \tau \nabla w )
- W_{\xi \xi} ( \nabla u_g) \right) [ \nabla w , \nabla w ] \, dz \nonumber \\
&\geq \int_{\Omega_g} W (\nabla u_g) + \Bigl( \frac{c_0}{8} - \omega (\delta) \Bigr) \| w  \|^2_{H^1 (\Omega_g;\R^N)}, \label{ugmin}
\end{align}
where we used \eqref{provv} and we set
$$
\omega(\delta):=
\max \Bigl\{  \| W_{\xi \xi} ( A + \tau B ) - W_{\xi \xi} ( A )   \|_{\infty} : \,
A \in K , \,
B \in \mathbb{M}^{N}, \,
|B| \leq \delta, \, 0\leq\tau\leq1 \Bigr\},
$$
with $K$ as in Remark~\ref{rm:IFT}.
Note that $\omega (\delta) \to 0$ as $\delta \to 0^+$.
Therefore, choosing $\delta$ so small that $\omega (\delta) < \frac{c_0}{8}$
it follows from \eqref{ugmin} that $u_g$ is a strict $\delta$-local minimizer.
\end{proof}

\end{section}


\begin{section}{The second variation} \label{sect:varII}

The main result of this section is the explicit computation of the second variation of the functional $F$
along volume-preserving deformations.
Here and in the following we assume that $(h,u)\in X$ satisfies the assumptions of Proposition~\ref{IFT}:
$h\in C^2_\#(Q)$, $u\in C^2(\overline{\Omega}_h^\#;\R^N)$ is a critical point for the elastic energy in $\Omega_h$,
and condition \eqref{c0} holds.

Given $\phi \in C^{2}_\# (Q)$ with $\int_Q \phi \, dx = 0$, for $t \in \R$
we set $h_t:= h + t \phi$.
According to Proposition~\ref{IFT}, for $t$ so small that $h_t \in \mathcal{U}$
we may consider a critical point $u_{h_t}$ for the elastic energy in $\Omega_{h_t}$.
To simplify the notation, we set $u_t :=u_{h_t}$.
We define the \textit{second variation of $F$ at $(h,u)$ along the direction $\phi$} to be the value of
$$
\frac{d^2}{dt^2} \left[ F (h_t,u_t)  \right] |_{t=0}.
$$
We remark that the existence of the derivative is guaranteed by the regularity result contained in Proposition~\ref{IFT}
(see the first step of the proof of Theorem~\ref{th:var2}).

Before stating the main results of this section, we introduce some more notation.
For any one-parameter family of functions $\{g_t\}_{t\in\R}$
we denote by $\dot g_t(z)$ the partial derivative with respect to $t$
of the function $(t,z)\mapsto g_t(z)$. We omit the subscript when $t=0$.
In particular we let
$$
\dot u_t:=\frac{\partial u_t}{\partial t},\,\qquad \dot u:=\frac{\partial u_t}{\partial t} \Big|_{t=0}\,.
$$
We introduce also the following subspace of $H^1(\Gamma_h)$:
\begin{align*}
{\widetilde H}^1_\#(\Gamma_h) &:=\Bigl\{\vartheta \in H_{loc}^1(\Gamma^{\#}_h):\,
\vartheta(x + e_i,h(x + e_i))=\vartheta(x,h(x)) \text{ for a.e. }x \in \R^{N-1} \\
&\hspace{1cm} \text{ and for every }i=1,\ldots,N-1 ,\,
\int_{\Gamma_h}\vartheta\,d\mathcal{H}^{N-1}=0\Bigr\}\,,
\end{align*}
and we define $\vphi \in {\widetilde H}^1_\#(\Gamma_h)$ to be
$$
\vphi:=\frac{\phi}{\sqrt{1+ |\nabla h|^2}} \circ~\pi,
$$
where $\pi :\R^N \to \R^{N-1}$ is the orthogonal projection on the hyperplane spanned by $\{e_1,\ldots,e_{N-1}\}$.
Denote also by $\nu_t$ the outer unit normal vector to $\Omega_{h_t}$ on $\Gamma_{h_t}$, and by $H^\psi_t:=\div\left(\nabla\psi\circ\nu_t\right)$ the anisotropic curvature of $\Gamma_{h_t}$.
It will be convenient to consider, as we did before,
a family of diffeomorphisms $\Phi_t:\overline{\Omega}_h\to\overline{\Omega}_{h_t}$ of class $C^{2}$
such that $\Phi_0=Id$ and $\Phi_t(x,y)=(x,y+t\phi(x))$ in a neighborhood of $\Gamma_h$
(see Remark~\ref{rm:diffeo}).

In the following theorem we deduce an explicit expression of the second variation.

\begin{theorem}\label{th:var2}
Let $(h,u)$, $\phi$, $\vphi$ and $(h_t, u_t)$ be as above. Then
the function $\dot{u}$ belongs to $\mathcal{V}(\Omega_h)$ and satisfies the equation
\begin{align}
\int_{\Omega_h} C_u \nabla \dot{u} : \nabla w \, dz
=\int_{\Gamma_h} \div_{\Gamma_h} (\vphi\, W_{\xi} (\nabla u)) \cdot w
\, d\hn \qquad\text{for all } w\in \Vtilde(\Omega_h). \label{upto-solves}
\end{align}
Moreover, the second variation of $F$ at $(h, u)$ along the direction $\phi$ is given by
\begin{align} \label{fvar2}
\frac{d^2}{dt^2}F&(h_t, u_{t} )|_{t=0}
= - \int_{\Omega_h} C_u \nabla \dot{u} : \nabla \dot{u} \, dz
+ \int_{\Gamma_h} ( \nabla^2 \psi \circ \nu ) [ \nabla _{\Gamma_h} \varphi, \nabla _{\Gamma_h} \varphi ] \, d\hn \nonumber\\
&+ \int_{\Gamma_h} \bigl( \partial_\nu (  W \circ \nabla u ) - \tr(\Bpsi\B) \bigr) \, \vphi^2\,  d\hn \\
&- \int_{\Gamma_h}\bigl( W \circ \nabla u +H^{\psi} \bigr) \,
\div_{\Gamma_h} \left[ \left( \frac{(\nabla h, |\nabla h|^2 )} {\sqrt{1 + |\nabla h|^2}} \circ \pi \right) \vphi^2\right] \, d \hn, \nonumber
\end{align}
where $H^\psi$, $\B$ and $\B^\psi$ are the anisotropic mean curvature, the second fundamental form and the anisotropic second fundamental form of $\Gamma_h$, respectively.
\end{theorem}

Before proving the theorem, we collect in the following lemma
some identities that will be used in the computation of the second variation.

\begin{lemma} \label{lem:identita}
The following identities are satisfied on $\Gamma_h$:
\begin{itemize}
\item[(a)] $\partial_{\nu} H^{\psi} = - \tr \left( \Bpsi\B \right) = -\tr \left( \B^2(\nabla^2\psi\circ\nu) \right)$;
\smallskip
\item[(b)] $\dot{\nu} = - \nabla_{\Gamma_h} \varphi$;
\smallskip
\item[(c)] $\dot{H}^{\psi} = \div_{\Gamma_h} \left( (\nabla^2\psi\circ\nu)[\dot{\nu}] \right) = - \div_{\Gamma_h} \left( ( \nabla^2 \psi \circ \nu ) [ \nabla_{\Gamma_h} \varphi ] \right)$.
\end{itemize}
\end{lemma}

\begin{proof}
Recalling that $\nabla\nu[\nu]=0$, we easily deduce that $\nabla \left( \nabla \psi \circ \nu \right) [\nu] = 0$.
By differentiating,
$$
\partial_{\nu} \left( \nabla \left( \nabla \psi \circ \nu \right) \right)
= - \nabla \left( \nabla \psi \circ \nu \right) \nabla \nu =
- \Bpsi\B,
$$
and from this we obtain (a), since
\begin{align*}
\partial_{\nu} H^{\psi}
&= \partial_{\nu} \left[ \div ( \nabla \psi \circ \nu  ) \right]
=  \partial_{\nu} \left[ \tr \left( \nabla \left( \nabla \psi \circ \nu \right) \right)   \right]\\
&= \text{trace} \left[ \partial_{\nu} \left( \nabla \left( \nabla \psi \circ \nu \right) \right)   \right]
= - \text{trace} \left[ \Bpsi\B \right].
\end{align*}

Let us prove (b). Differentiating with respect to $t$ the identity
$$
\nu_t \circ \Phi_t = \frac{(-\nabla h_t,1)}{\sqrt{1+|\nabla h_t|^2}}\circ\pi \qquad \text{on }\Gamma_h,
$$
and evaluating the result at $t=0$, we get that on $\Gamma_h$ holds
\begin{align*}
\dot{\nu}+(\phi\circ\pi)\partial_y\nu
&= \Biggl( \frac{-\nabla\phi}{\sqrt{1+|\nabla h|^2}} + \frac{(\nabla h\cdot\nabla\phi)\nabla h}{(1+|\nabla h|^2)^{\frac32}}, \frac{-\nabla h\cdot \nabla\phi}{(1+|\nabla h|^2)^\frac32}\Biggl)\circ\pi\\
&= \frac{(-\nabla\phi,0)}{\sqrt{1+|\nabla h|^2}}\circ\pi - \left(\frac{\nabla h\cdot\nabla\phi}{1+|\nabla h|^2}\circ\pi\right)\nu\\
&= \biggl(-\frac{1}{\sqrt{1+|\nabla h|^2}}\circ\pi\biggr) \Bigl[\nabla(\phi\circ\pi)-\bigl(\nabla(\phi\circ\pi)\cdot\nu\bigr)\nu\Bigr]\\
&= \biggl(-\frac{1}{\sqrt{1+|\nabla h|^2}}\circ\pi\biggr) \nabla_{\Gamma_h}(\phi\circ\pi).
\end{align*}
Hence, using the identity
$$
\partial_y\nu=\nabla_{\Gamma_h}\biggl(\frac{1}{\sqrt{1+|\nabla h|^2}}\circ\pi\biggr) \qquad\text{on }\Gamma_h,
$$
we finally get
\begin{align*}
\dot{\nu}
&= -\biggl(\frac{1}{\sqrt{1+|\nabla h|^2}}\circ\pi\biggr) \nabla_{\Gamma_h}(\phi\circ\pi) - \nabla_{\Gamma_h}\biggl(\frac{1}{\sqrt{1+|\nabla h|^2}}\circ\pi\biggr)(\phi\circ\pi)\\
&= -\nabla_{\Gamma_h}\biggl(\frac{\phi}{\sqrt{1+|\nabla h|^2}}\circ\pi\biggr) = -\nabla_{\Gamma_h}\varphi,
\end{align*}
that is (b).

Let us prove (c).
Differentiating in the direction $\nu$ the identity
$( \nabla^2 \psi \circ \nu ) [\nu, \dot{\nu}] =0$
(which follows by \eqref{hesspsi}), we obtain
$$
\nu \cdot \partial_{\nu} \left( ( \nabla^2 \psi \circ \nu) [\dot{\nu}] \right) =
 - ( \nabla^2 \psi \circ \nu ) [\dot{\nu}, \partial_{\nu}\nu] =0,
$$
where we recall that $\partial_\nu\nu=0$.
Hence
\begin{align*}
\dot{H}^{\psi}
&= \frac{\partial}{\partial t}H^\psi_t |_{t=0} =
\frac{\partial}{\partial t}
\left[ \div ( \nabla \psi \circ \nu_t ) \right] |_{t=0}
= \div \left( ( \nabla^2 \psi \circ \nu ) [\dot{\nu}] \right) \\
&= \div_{\Gamma_h} \left( ( \nabla^2 \psi \circ \nu ) [\dot{\nu}] \right)
+ \nu \cdot \partial_{\nu} \left( ( \nabla^2 \psi \circ \nu ) [\dot{\nu}] \right) \\
&= \div_{\Gamma_h} \left( ( \nabla^2 \psi \circ \nu ) [\dot{\nu}] \right)
= - \div_{\Gamma_h} \left( ( \nabla^2 \psi \circ \nu ) [ \nabla_{\Gamma_h} \varphi] \right),
\end{align*}
where in the last equality we used (b).
\end{proof}

We are now ready to perform the computation of the second variation of the functional.

\begin{proof}[Proof of Theorem~\ref{th:var2}]
We divide the proof into several steps.

\smallskip
\noindent{\it Step 1.}
We claim that the regularity property stated in Proposition~\ref{IFT}-(iii) guarantees that
the map $(t,z) \mapsto w_t(z) := u_t\circ\Phi_t(z)$ is of class $C^1$ in $(-\e_0,\e_0)\times\overline{\Omega}_h$
for some $\e_0$ small enough.

Indeed, denoting by $w'_{t_0}$ the derivative of the map $t\mapsto w_t$ with respect to the $W^{2,p}$-norm,
evaluated at some $t_0$ (small), we have that
\begin{equation}\label{differenziab}
\textstyle\frac1s \bigl(w_{t_0+s}-w_{t_0}\bigr) \to w'_{t_0}
\qquad\text{in } W^{2,p}(\Omega_h), \text{ as } s\to0.
\end{equation}
In particular, $w_{t_0+s} \to w_{t_0}$ in $C^1(\overline{\Omega}_h)$ as $s\to0$,
showing that the map $(t,z)\mapsto\nabla w_t(z)$ is continuous in $(-\e_0,\e_0)\times\overline{\Omega}_h$.
Moreover, \eqref{differenziab} implies that $w'_{t_0}=\dot{w}_{t_0}$,
and the continuity of $t\mapsto w'_t$ yields $\dot{w}_{t_0+s} \to \dot{w}_{t_0}$ in $C^0(\overline{\Omega}_h)$ as $s\to0$,
showing that the map $(t,z)\mapsto\dot{w}_t(z)$ is continuous in $(-\e_0,\e_0)\times\overline{\Omega}_h$.
The claim follows.

This provides a justification to all the differentiations that will be performed throughout the proof.
Moreover, it is also easily seen that $\dot{u}_t \in \mathcal{V}(\Omega_{h_t})$ for $t\in(-\e_0,\e_0)$.

\smallskip
\noindent{\it Step 2.} We prove \eqref{upto-solves}. Let us recall that $u_t$ satisfies equation \eqref{VolEul}:
\begin{equation} \label{VolEult}
\int_{\Omega_{h_t}} W_{\xi} (\nabla u_t): \nabla w \, dz = 0
\quad \text{ for every }w \in \mathcal{V}(\Omega_{h_t}).
\end{equation}
Fix $w\in \mathcal{V}(\Omega_h)$. Then $w$ may be extended outside $\Omega_h$ in such a way that $w\in\mathcal{V}(\Omega_{h_t})$ for $t$ small. We can differentiate \eqref{VolEult} with respect to $t$ and  evaluate the result at $t=0$  to obtain
\begin{eqnarray}
0&=&\int_{\Omega_h} C_u \nabla \dot{u} : \nabla w \, dz
+ \int_Q \phi (x) \left[ W_{\xi} (\nabla u ) : \nabla w \right] |_{(x,h(x))} \, dx \label{upto2}\\
&=&\int_{\Omega_h} C_u \nabla \dot{u} : \nabla w \, dz
+ \int_{\Gamma_h} \vphi \, W_{\xi} (\nabla u ) : \nabla w \, d \hn. \nonumber
\end{eqnarray}
Recalling that $W_{\xi} (\nabla u) [\nu]=0$ along $\Gamma_h$,
the second integral in the above formula  can be rewritten as
\begin{eqnarray*}
\int_{\Gamma_h} \vphi \, W_{\xi} (\nabla u ) : \nabla w \, d \hn
&=& \int_{\Gamma_h}\vphi\,W_{\xi} (\nabla u ) : \nabla_{\Gamma_h} w\,d\hn \\
&=& -\int_{\Gamma_h} \div_{\Gamma_h} (\vphi\, W_{\xi} ( \nabla u ) ) \cdot w\, d\hn\,.
\end{eqnarray*}
This concludes the proof of \eqref{upto-solves}.

\smallskip
\noindent{\it Step 3.}
We compute the first variation.
By the positive one-homogeneity of $\psi$ we have on $\Gamma_{h_t}$
$$
\psi (\nu_t ) = \psi \left( \frac{( - \nabla h_t  , 1 )}{\sqrt{1 + | \nabla h_t  |^2 }} \circ\pi \right)
= \frac {\psi \left( ( - \nabla h_t  , 1 ) \right)} {\sqrt{1+|\nabla h_t|^2}} \circ\pi.
$$
Hence,
\begin{align*}
\frac{d}{dt}F( h_t, u_{t} )
&=  \frac{d}{dt} \left[ \int_{Q} \int_0^{h_t} W (\nabla u_t) \, dy \, dx
+ \int_Q \psi ( (- \nabla h_t , 1)) \, d x \right] \\
&=  \int_{Q} \phi (x) \left[ W (\nabla u_t )  \right] |_{(x,h_t(x))} \, dx
+  \int_{Q} \int_0^{h_t } W_{\xi} (\nabla u_t) : \nabla \dot{u}_t \, dy \, dx \\
&\hspace{.5cm}- \int_Q \nabla \psi ( (- \nabla h_t , 1)) \cdot (\nabla \phi,0) \, d x .
\end{align*}
Since $\dot{u}_t \in \mathcal{V}(\Omega_{h_t})$ the second integral vanishes by \eqref{VolEult}.
Then, integrating by parts in the last integral and recalling the expression for the anisotropic mean curvature provided by \eqref{Hpsi}, we obtain
\begin{equation} \label{firstvar}
\frac{d}{dt}F( h_t, u_{t} )
=  \int_{Q} \phi (x) \bigl[  W (\nabla u_t ) + H^{\psi}_t \bigr] |_{(x,h_t(x))} \, dx.
\end{equation}

\smallskip
\noindent{\it Step 4.}
We finally pass to the second variation.
Differentiating \eqref{firstvar} with respect to $t$ and evaluating the result at $t=0$ we get
\begin{align*}
\frac{d^2}{dt^2}F(h_t, u_{t} )|_{t=0}
&= \int_Q \phi (x) \left[ W_{\xi} (\nabla u) : \nabla \dot{u} \right] |_{(x,h(x))} \, dx
+ \int_{Q} \phi (x) \dot{H}^{\psi} |_{(x,h (x))} \, dx \\
&\hspace{.5cm}+ \int_Q \phi (x) \left[ \nabla ( W \circ \nabla u + H^{\psi}) \right] |_{(x,h(x))} \cdot (0, \phi (x)) \, dx \\
&= : I_1 + I_2 + I_3.
\end{align*}
Since $\dot{u}\in\mathcal{V}(\Omega_h)$, thanks to \eqref{upto2} the first integral is
$$
I_1 =  - \int_{\Omega_h} C_u \nabla \dot{u} : \nabla \dot{u} \, dz.
$$
For the second integral, changing variables, using identity (c) of Lemma~\ref{lem:identita}
and integrating by parts we get
\begin{align*}
I_2 =  - \int_{\Gamma_h} \varphi \, \div_{\Gamma_h} \left( ( \nabla^2 \psi \circ \nu ) [ \nabla_{\Gamma_h} \varphi] \right)  d \hn
= \int_{\Gamma_h} ( \nabla^2 \psi \circ \nu ) [ \nabla _{\Gamma_h} \varphi, \nabla _{\Gamma_h} \varphi ] \, d\hn.
\end{align*}
To conclude, we observe that along $\Gamma_h$ the vector $(0, \varphi )$
can be decomposed as
$$
(0, \varphi ) = (0, \varphi )_{\Gamma_h} + (0, \varphi )_{\nu},
$$
with $(0, \varphi )_{\Gamma_h}$ tangent to $\Gamma_h$
and $(0, \varphi )_{\nu}$  parallel to $\nu$, \textit{i.e.},
$$
(0, \varphi )_{\Gamma_h} = \varphi \left[ \frac{(\nabla h, |\nabla h|^2)}{1 + |\nabla h|^2} \circ \pi \right] ,
\quad \quad
(0, \varphi )_{\nu} = \varphi \left[ \frac{( - \nabla h, 1)}{1 + |\nabla h|^2} \circ \pi \right].
$$
Hence, recalling the definition of $\varphi$,
changing variables in $I_3$ and integrating by parts:
\begin{align*}
I_3 &= \int_{\Gamma_h} \varphi \, \nabla ( W \circ \nabla u + H^{\psi})
\cdot (0, \varphi ) \left( \sqrt{1 + |\nabla h|^2} \circ \pi \right) \, d \hn \\
&= \int_{\Gamma_h} \varphi^2 \, \nabla_{\Gamma_h} ( W \circ \nabla u + H^{\psi}) \cdot
\left( \frac{(\nabla h, |\nabla h|^2 )} {\sqrt{1 + |\nabla h|^2}} \circ \pi \right) \, d \hn \\
&\hspace{1cm}+ \int_{\Gamma_h} \varphi^2 \, \partial_{\nu} ( W \circ \nabla u + H^{\psi}) \, d \hn \\
&= - \int_{\Gamma_h} ( W \circ \nabla u + H^{\psi}) \, \div_{\Gamma_h} \left[ \left( \frac{(\nabla h, |\nabla h|^2 )}
{\sqrt{1 + |\nabla h|^2}} \circ \pi \right) \vphi^2\right]  \, d \hn \\
&\hspace{1cm}+ \int_{\Gamma_h} \varphi^2 \, \left[ \partial_{\nu} ( W \circ \nabla u) - \tr (\Bpsi\B) \right] \, d \hn,
\end{align*}
where in the last equality we used identity (a) of Lemma~\ref{lem:identita}.
\end{proof}

\begin{remark}\label{rm:var2s}
For a fixed $s\in\R$ sufficiently small, we deduce also from Theorem~\ref{th:var2} that
\begin{align*}
\frac{d^2}{dt^2}F&(h_t, u_{t} )|_{t=s} = \frac{d^2}{dt^2}F(h_{s+t}, u_{s+t} )|_{t=0} \\
&= - \int_{\Omega_{h_s}} C_{u_s} \nabla \dot{u}_s : \nabla \dot{u}_s \, dz
+ \int_{\Gamma_{h_s}} ( \nabla^2 \psi \circ \nu_s ) [\nabla _{\Gamma_{h_s}} \varphi_s, \nabla _{\Gamma_{h_s}} \varphi_s] \, d\hn \\
&\hspace{0.5cm} + \int_{\Gamma_{h_s}} \bigl( \partial_{\nu_s} (  W \circ \nabla u_s ) - \tr (\Bpsi_s\B_s) \bigr) \, \vphi_s^2\,  d\hn \\
&\hspace{0.5cm} - \int_{\Gamma_{h_s}} \bigl( W \circ \nabla u_s +H^{\psi}_s \bigr) \,
\div_{\Gamma_{h_s}} \left[ \left( \frac{(\nabla h_s, |\nabla h_s|^2 )} {\sqrt{1 + |\nabla h_s|^2}} \circ \pi \right) \vphi_s^2\right] \, d \hn,
\end{align*}
where $\vphi_s:=\frac{\phi}{\sqrt{1+|\nabla h_s|^2}}\circ\pi \in\Htilde(\Gamma_{h_s})$,
$\B_s:=\nabla\nu_s$ and $\Bpsi_s:=\nabla(\nabla\psi\circ\nu_s)$.
Moreover, the function $\dot{u}_s$ belongs to $\mathcal{V}(\Omega_{h_s})$ and satisfies the equation
$$
\int_{\Omega_{h_s}} C_{u_s} \nabla \dot{u}_s : \nabla w \, dz
=\int_{\Gamma_{h_s}} \div_{\Gamma_{h_s}} (\vphi_s\, W_{\xi} (\nabla u_s)) \cdot w
\, d\hn \qquad\text{for all $w\in \Vtilde(\Omega_{h_s})$.}
$$
\end{remark}

\smallskip
\subsection{The second order condition}
The expression of the second variation at a critical pair (see Definition~\ref{def:critpair}) simplifies,
as the last integral in \eqref{fvar2} vanishes by the divergence formula.
This observation suggests to associate with every critical pair $(h,u)\in X$
a quadratic form $\partial^2F(h,u):\Htilde(\Gamma_h)\to\R$ defined as
\begin{align} \label{fquad}
\partial^2F(h,u)[\varphi]:= -\int_{\Omega_h}&C_u\nabla v_\varphi : \nabla v_\varphi \,dz
+ \int_{\Gamma_h}(\nabla^2\psi\circ \nu)[\nabla_{\Gamma_h}\varphi,\nabla_{\Gamma_h}\varphi]\,d\hn \nonumber\\
&+ \int_{\Gamma_h} \bigl( \partial_\nu(W\circ\nabla u)- \tr (\Bpsi\B) \bigr) \vphi^2\,d\hn,
\end{align}
where $\vf\in\Vtilde(\Omega_h)$ is the unique solution to
\begin{equation} \label{vf}
\int_{\Omega_h}C_u\nabla\vf : \nabla w = \int_{\Gamma_h}\div_{\Gamma_h}(\vphi W_\xi(\nabla u)) \cdot w \,d\hn
\qquad \text{for every } w\in\Vtilde(\Omega_h).
\end{equation}
It is easily seen that the positivity of the quadratic form \eqref{fquad}
is a necessary condition for local minimality:
this is made precise by the following theorem.

\begin{theorem} \label{teo:nec}
Let $(h,u)\in X$, with $h\in C^2_\#(Q)$ and $u\in C^2(\overline{\Omega}_h^\#;\R^N)$,
be a local minimizer for $F$, according to Definition~\ref{def:locmin},
and assume in addition that $u$ satisfies \eqref{c0}.
Then the quadratic form \eqref{fquad} is positive semidefinite, \textit{i.e.},
$$
\partial^2 F(h,u)[\vphi]\geq 0 \qquad \text{for every }\vphi\in\Htilde(\Gamma_h).
$$
\end{theorem}

\begin{proof}
Given any $\vphi\in\Htilde(\Gamma_h)\cap C^\infty(\Gamma_h^\#)$, we can consider the deformation $h_t=h+t\phi$, where $\phi(x)=(1+|\nabla h(x)|^2)^{\frac12} \, \vphi(x,h(x))$, and, for $t$ small, the corresponding critical points for the elastic energy $u_{h_t}$. It follows from equation \eqref{fvar2} and from the local minimality of $(h,u)$ (which is in particular a critical pair) that
$$
\partial^2F(h,u)[\vphi] = \frac{d^2}{dt^2}F(h_t, u_{h_t} )|_{t=0} \geq 0.
$$
For a general $\vphi$ the result follows by approximation with functions in $\Htilde(\Gamma_h)\cap C^\infty(\Gamma_h^\#)$
(observe that $\partial^2F(h,u)$ is continuous with respect to strong convergence in $H^1$).
\end{proof}

\begin{definition} \label{def:stable2}
Let $(h,u)\in X$ be a critical pair for the functional $F$, according to Definition~\ref{def:critpair}.
We say that $(h,u)$ is \emph{strictly stable} if the elastic second variation is uniformly positive at $u$ in $\Omega_h$
(see Definition~\ref{def:stable1}) and in addition
\begin{equation} \label{defpos}
\partial^2 F(h,u)[\vphi]>0 \qquad \text{for every }\vphi\in\Htilde(\Gamma_h)\setmeno\{0\}.
\end{equation}
\end{definition}

The main result of this paper (Theorem~\ref{cor:locmin})
states that a strictly stable critical pair is a local minimizer for $F$,
according to Definition~\ref{def:locmin}.
This will be proved in Sections~\ref{sect:locmin} and \ref{sect:WimpliesL},
while we now focus on condition \eqref{defpos} providing two equivalent formulations.

Given a critical pair $(h,u)\in X$ satisfying \eqref{c0}, we define the bilinear form on $\Htilde(\Gamma_h)$
\begin{equation} \label{prodottoscalare}
(\vphi,\vartheta)_\sim := \int_{\Gamma_h}(\nabla^2\psi\circ\nu)[\nabla_{\Gamma_h}\vphi, \nabla_{\Gamma_h}\vartheta]\,d\hn
+\int_{\Gamma_h}a\,\vphi\,\vartheta\,d\hn
\end{equation}
for $\vphi,\vartheta\in\Htilde(\Gamma_h)$, where $a:=\partial_\nu(W\circ\nabla u)-\tr(\Bpsi\B)$ on $\Gamma_h$.
Arguing as in \cite[Proposition~4.2]{CMM}, one can show that if
\begin{equation} \label{prodscalpos}
(\vphi,\vphi)_\sim>0 \qquad \text{for every }\vphi\in\Htilde(\Gamma_h)\setmeno\{0\},
\end{equation}
then $(\cdot\,,\cdot)_\sim$ is a scalar product which defines
an equivalent norm on $\Htilde(\Gamma_h)$, denoted by $\|\cdot\|_\sim$.
We omit the proof also of the following result, since it can be deduced by repeating the proof of \cite[Proposition~3.6]{FM}
(see also \cite[Proposition~4.3, Theorem~4.6, Theorem~4.10]{CMM}).

\begin{theorem} \label{teo:condequiv}
The following statement are equivalent.
\begin{itemize}
  \item [(i)] Condition \eqref{defpos} holds.
  \item [(ii)] Condition \eqref{prodscalpos} is satisfied and $T:\Htilde(\Gamma_h)\to\Htilde(\Gamma_h)$, defined by duality as
      \begin{equation} \label{T}
      (T\vphi,\vartheta)_\sim :
      = \int_{\Gamma_h} \div_{\Gamma_h}(\vartheta \, W_\xi(\nabla u)) \cdot \vf \,d\hn
      = \int_{\Omega_h} C_u \nabla \vf : \nabla v_{\vartheta} \,dz
      \end{equation}
      for every $\vphi,\vartheta\in\Htilde(\Gamma_h)$, is a compact, monotone, self-adjoint linear operator such that
      \begin{equation} \label{lambda1}
      \lambda_1<1, \qquad\text{where }\lambda_1:=\max_{\|\vphi\|_\sim=1} (T\vphi,\vphi)_\sim.
      \end{equation}
  \item [(iii)] Condition \eqref{prodscalpos} is satisfied and defining, for $v\in\Vtilde(\Omega_h)$, $\Phi_v$ to be the unique solution in $\Htilde(\Gamma_h)$ to the equation
      \begin{equation*}
      (\Phi_v , \vartheta)_\sim =
      \int_{\Gamma_h} \div_{\Gamma_h} (\vartheta \, W_\xi(\nabla u)) \cdot v \, d\hn
      \qquad \text{for every }\,\vartheta\in\Htilde(\Gamma_h),
      \end{equation*}
      we have
      \begin{equation}\label{mu1}
      \mu_1 := \min \biggl\{ \int_{\Omega_h} C_u\nabla v : \nabla v \,dz \;:\; v\in\Vtilde(\Omega_h), \; \|\Phi_v\|_\sim=1 \biggr\} >1.
      \end{equation}
\end{itemize}
\end{theorem}

\begin{remark} \label{rm:autovalori}
We remark that, by definition of $T$, we have
\begin{equation} \label{fquadT}
\partial^2F(h,u)[\vphi] = \|\vphi\|^2_\sim - (T\vphi,\vphi)_\sim \qquad \text{for every }\,\vphi\in\Htilde(\Gamma_h).
\end{equation}
Observe also that $\lambda_1$ coincides with the greatest $\lambda$ such that the following system
\begin{equation} \label{autovalori}
\begin{cases}
\lambda \int_{\Omega_h} C_u\nabla v : \nabla w = \int_{\Gamma_h} \div_{\Gamma_h} (\vphi \, W_\xi(\nabla u) ) \cdot w\, d\hn
& \text{for every } w \in\Vtilde(\Omega_h), \\
(\vphi,\psi)_\sim = \int_{\Gamma_h} \div_{\Gamma_h} (\psi \, W_\xi(\nabla u) ) \cdot v\, d\hn
& \text{for every } \psi \in\Htilde(\Gamma_h)
\end{cases}
\end{equation}
admits a nontrivial solution $(v,\vphi)\in\Vtilde(\Omega_h)\times\Htilde(\Gamma_h)$: in fact, $\lambda$ is an eigenvalue of $T$ with eigenfunction $\vphi$ if and only if the pair $(\frac{v_\vphi}{\lambda},\vphi)$ is a nontrivial solution to \eqref{autovalori}.
\end{remark}

\begin{corollary} \label{cor:defpos}
If \eqref{defpos} holds, then $\partial^2F(h,u)$ is uniformly positive: that is, there exists a constant $C>0$ such that
$$
\partial^2F(h,u)[\vphi]\geq C\|\vphi\|^2_{H^1(\Gamma_h)} \qquad \text{for every }\,\vphi\in\Htilde(\Gamma_h).
$$
\end{corollary}

\begin{proof}
By \eqref{fquadT}, recalling that $\|\cdot\|_\sim$ is an equivalent norm on $\Htilde(\Gamma_h)$ and that $\lambda_1<1$ we have
$$
\partial^2F(h,u)[\vphi]
= \|\vphi\|_\sim^2 - (T\vphi,\vphi)_\sim
\geq (1-\lambda_1)\|\vphi\|_\sim^2
\geq C \|\vphi\|_{H^1(\Gamma_h)}^2,
$$
which is the conclusion.
\end{proof}

\end{section}


\begin{section}{$W^{2,p}$-local minimality} \label{sect:locmin}

In this section we prove the first part of the main result of the paper,
namely that the strict stability of a critical pair $(h,u)$
is a sufficient condition for local minimality, in the following weaker sense:

\begin{definition} \label{def:locmindeb}
Let $p\in[1,\infty)$.
We say that a critical pair $(h,u)\in X$ is a \textit{$W^{2,p}$-local minimizer} for $F$ if there exists $\delta>0$ such that
\begin{equation} \label{locminw}
F(h,u)\leq F(g,v)
\end{equation}
for all $(g,v)\in X$ with $0<\|g-h\|_{W^{2,p}(Q)}<\delta$, $|\Omega_g|=|\Omega_h|$,
and $\|\nabla v - \nabla u\|_{L^\infty(\Omega_g;\M^{N})}<\delta$.
We say that $(h,u)$ is an \textit{isolated $W^{2,p}$-local minimizer} if the inequality in \eqref{locminw} is strict when $g\neq h$.
\end{definition}

\begin{theorem} \label{teo:locmin}
Let $N=2,3$, and let $p>2N$. If $(h,u)\in X$ is a strictly stable critical pair for $F$,
according to Definition~\ref{def:stable2},
then $(h,u)$ is an isolated $W^{2,p}$-local minimizer for $F$.
\end{theorem}

The proof will be achieved by following, essentially,
the strategy developed in \cite[Proposition~4.5 and Theorem~4.6]{FM} (see also \cite[Theorem~5.1]{CMM}).
As has been observed in \cite{FM}, the main difficulty in proving Theorem~\ref{teo:locmin}
comes from the presence, in the expression of the quadratic form associated with the second variation,
of the trace of the gradient of $W(\nabla u)$ on $\Gamma_h$:
the crucial estimate is provided by Lemma~\ref{lem:algebra},
where it is shown how to control this term in a proper Sobolev space of fractional order,
uniformly with respect to small $W^{2,p}$-variations of the profile $h$
(we refer to Section~\ref{sect:appendix} for the definition and properties of fractional Sobolev spaces).

Let $\mathcal{U}_\delta := \{g\in C^\infty_\#(Q) : \|g-h\|_{W^{2,p}(Q)}<\delta, \, |\Omega_g|=|\Omega_h|\}$,
where $\delta>0$ is so small that $\mathcal{U}_\delta$ is contained in the neighborhood $\mathcal{U}$ of $h$ determined by Proposition~\ref{IFT}: this allows us to consider, for $g\in\mathcal{U}_\delta$, a critical point $u_g$ for the elastic energy in $\Omega_g$.
We denote by $c_0$ a positive constant such that $g\geq 2c_0$ in $Q$ for every $g\in\mathcal{U}_\delta$.

\begin{lemma} \label{lem:algebra}
Under the assumptions of Theorem~\ref{teo:locmin}, we have that
$$
\sup_{g\in\mathcal{U}_\delta}
\| \partial_{\nu_g}(W(\nabla u_g))\circ\Phi_g-\partial_{\nu}(W(\nabla u)) \|_{W_\#^{-\frac1p,p}(\Gamma_h)} \to 0
\quad\text{as }\delta\to0.
$$
\end{lemma}

\begin{proof}
We set, for $g\in\mathcal{U}_\delta$, $v_g := u_g - u\circ\Psi_g$ (where $\Psi_g:=\Phi_g^{-1}$),
and we denote by $v_g^i$ the components of $v_g$.
We remark that, by Proposition~\ref{IFT},
\begin{equation} \label{convw2p}
\sup_{g\in\mathcal{U}_\delta} \|v_g\|_{W^{2,p}(\Omega_g;\R^N)}\to0
\qquad
\text{as }\delta\to0,
\end{equation}
and moreover, since $p>2N$, $u_g\circ\Phi_g\to u$ in $C^{1,\alpha}(\overline{\Omega}_h;\R^N)$ as $\delta\to0$,
for $\alpha=1-\frac{N}{p}$,
uniformly with respect to $g\in\mathcal{U}_\delta$.

\smallskip
\noindent{\it Step 1.}
We start by observing that,
using the equations satisfied by $u$ and $u_g$ and performing a change of variable, we get
\begin{align*}
\int_{\Omega_g} \bigl[ W_\xi(\nabla u_g) - W_\xi(\nabla(u\circ\Psi_g)) \bigr] : \nabla w \,dz
= \int_{\Omega_g} d_g : \nabla w \,dz
\qquad\text{for all }w\in\mathcal{V}(\Omega_g),
\end{align*}
where $d_g:= W_\xi(\nabla(u\circ\Psi_g)(\nabla\Psi_g)^{-1})(\nabla\Psi_g)^{-T}\det\nabla\Psi_g - W_\xi(\nabla(u\circ\Psi_g))$.
Observe in particular that, by using the explicit construction of the diffeomorphism $\Psi_g$ (see Remark~\ref{rm:diffeo})
and the regularity of $u$,
\begin{equation} \label{stima-1lemma}
\sup_{g\in\mathcal{U}_\delta} \|d_g\|_{W^{1,p}(\Omega_g;\M^N)} \to 0, \qquad
\sup_{g\in\mathcal{U}_\delta} \Big\|\frac{\partial d_g}{\partial z_k} \Big\|_{L^{p}(\Gamma_g;\M^N)} \to 0  \qquad
\text{as } \delta\to0.
\end{equation}
Fix now any $\vphi\in W^{\frac1p,\frac{p}{p-1}}_\#(\Gamma_g;\R^N)$,
and consider an extension of $\vphi$ (which we still denote by $\vphi$)
such that $\vphi\in W^{1,\frac{p}{p-1}}_\#(\Omega_g;\R^N)$, $\vphi$ vanishes in $\Omega_{g-c_0}$ and
\begin{equation} \label{stima-2lemma}
\|\vphi\|_{W^{1,\frac{p}{p-1}}(\Omega_g;\R^N)}\leq C\|\vphi\|_{W^{\frac1p,\frac{p}{p-1}}(\Gamma_g;\R^N)}
\end{equation}
for some constant $C>0$ which can be chosen independently of $g\in\mathcal{U}_\delta$
(see Theorem~\ref{teo:sobolev3}).

Differentiating the equation $\div[W_\xi(\nabla u_g)-W_\xi(\nabla(u\circ\Psi_g))]=\div d_g$ with respect to $z_k$,
multiplying by $\vphi$ and integrating by parts on $\Omega_g$ we get
\begin{align*}
\int_{\Gamma_g} &C_{u_g}\nabla\Bigl(\frac{\partial v_g}{\partial z_k}\Bigr) [\nu_g] \cdot\vphi \,d\hn
= \int_{\Gamma_g} (C_{u\circ\Psi_g}-C_{u_g})\nabla\Bigl(\frac{\partial(u\circ\Psi_g)}{\partial z_k}\Bigr) [\nu_g] \cdot\vphi \,d\hn \nonumber\\
&\hspace{.5cm} + \int_{\Omega_g} \Bigl[ C_{u_g}\nabla\Bigl(\frac{\partial u_g}{\partial z_k}\Bigr) - C_{u\circ\Psi_g}\nabla\Bigl(\frac{\partial (u\circ\Psi_g)}{\partial z_k}\Bigr) -\frac{\partial d_g}{\partial z_k} \Bigr] : \nabla\vphi \,dz
 +\int_{\Gamma_g} \frac{\partial d_g}{\partial z_k}[\nu_g] \cdot\vphi \,d\hn \nonumber\\
& \leq C \biggl(
 \| C_{u\circ\Psi_g}-C_{u_g} \|_\infty \| \nabla^2 (u\circ\Psi_g) \|_{L^{p}(\Gamma_g)}
 + \| d_g \|_{W^{1,p}(\Omega_g;\M^N)} + \Big\| \frac{\partial d_g}{\partial z_k} \Big\|_{L^{p}(\Gamma_g;\M^N)} \nonumber\\
&\hspace{1.5cm} + \Big\| C_{u_g}\nabla\Bigl(\frac{\partial u_g}{\partial z_k}\Bigr) - C_{u\circ\Psi_g}\nabla\Bigl(\frac{\partial (u\circ\Psi_g)}{\partial z_k}\Bigr) \Big\|_{L^p(\Omega_g;\M^N)}
 \biggr) \|\vphi\|_{W^{\frac1p,\frac{p}{p-1}}(\Gamma_g;\R^N)},
\end{align*}
where we repeatedly used \eqref{stima-2lemma}
(here the constant $C$ is independent of $g\in\mathcal{U}_\delta$).
Hence, recalling \eqref{convw2p} and \eqref{stima-1lemma},
we deduce that for $k=1,\ldots,N$
\begin{equation} \label{stima0lemma}
\sup_{g\in\mathcal{U}_\delta}
\bigg\| C_{u_g}\nabla\Bigl(\frac{\partial v_g}{\partial z_k}\Bigr) [\nu_g] \bigg\|_{W^{-\frac1p,p}_\#(\Gamma_g;\R^N)}\to 0
\qquad\text{as }\delta\to0.
\end{equation}

\smallskip
\noindent{\it Step 2.}
We now claim that for $k=1,\ldots,N$
\begin{equation} \label{claim1lemma}
\sup_{g\in\mathcal{U}_\delta}
\Big\| \nabla\Bigl(\frac{\partial u}{\partial z_k}\Bigr) - \nabla\Bigl(\frac{\partial u_g}{\partial z_k}\Bigr)\circ\Phi_g \Big\|_{W^{-\frac1p,p}_\#(\Gamma_h;\M^N)} \to0
\qquad\text{as }\delta\to0.
\end{equation}

We first note that, thanks to the uniform convergence of $C_{u_g}\circ\Phi_g$ to $C_u$
and to the strong ellipticity of $C_u$,
also the tensors $C_{u_g}$ are strongly elliptic for every $g\in\mathcal{U}_\delta$,
if $\delta$ is sufficiently small; in particular, there exists a positive constant $m_0$ such that
$$
C_{u_g}(z) \, a\otimes b : a\otimes b \geq m_0 \, |a|^2 \, |b|^2
\qquad\text{for every }a,b\in\R^N,
$$
for every $z\in\overline{\Omega}_g$ and for every $g\in\mathcal{U}_\delta$.
Hence the $N\times N$ matrix $Q_g(z)$, whose entries are defined by
\begin{equation} \label{matrice}
q_{ih}(z):=\sum_{j,k=1}^N C_{ijhk}(z)\nu_g^j(z)\nu_g^k(z),
\qquad i,h=1,\ldots,N
\end{equation}
($C_{ijhk}$ denoting the components of the tensor $C_{u_g}$),
is positive definite, and $\det Q_g(z)$ is uniformly positive
with respect to $z\in\overline{\Omega}_g$ and $g\in\mathcal{U}_\delta$.

Setting, for $i,j,k=1,\ldots,N$,
$$
\sigma_{ijk}:= \frac{\partial^2 v_g^k}{\partial z_i\partial z_j} \,,
$$
by Lemma~\ref{lem:sobolev} our claim reduces to show that
$$
\sup_{g\in\mathcal{U}_\delta} \|\sigma_{ijk}\|_{W^{-\frac1p,p}_\#(\Gamma_g)} \to 0 \qquad\text{as }\delta\to0.
$$

We start from the case $N=2$. Consider the following system of equations at the points of $\Gamma_g$:
\begin{equation} \label{sistema6x6}
\left(
  \begin{array}{c}
    \eta_1 \\
    \eta_2 \\
    \vartheta_{11} \\
    \vartheta_{12} \\
    \vartheta_{21} \\
    \vartheta_{22} \\
  \end{array}
\right)
:=
\left(
  \begin{array}{cccccc}
    0 & 0 & a & b & c & d \\
    0 & 0 & a' & b' & c' & d' \\
    \nu_g^2 & 0 & -\nu_g^1 & 0 & 0 & 0 \\
    0 & \nu_g^2 & 0 & -\nu_g^1 & 0 & 0 \\
    0 & 0 & \nu_g^2 & 0 & -\nu_g^1 & 0 \\
    0 & 0 & 0 & \nu_g^2 & 0 & -\nu_g^1 \\
  \end{array}
\right)
\cdot
\left(
  \begin{array}{c}
    \sigma_{111} \\
    \sigma_{112} \\
    \sigma_{121} \\
    \sigma_{122} \\
    \sigma_{221} \\
    \sigma_{222} \\
  \end{array}
\right),
\end{equation}
where the coefficients in the first two rows of the matrix are defined by
\begin{align*}
a= C_{1111}\nu_g^1 + C_{1211}\nu_g^2, \quad
b= C_{1121}\nu_g^1 + C_{1221}\nu_g^2, \\
c= C_{1112}\nu_g^1 + C_{1212}\nu_g^2, \quad
d= C_{1122}\nu_g^1 + C_{1222}\nu_g^2, \\
a'= C_{2111}\nu_g^1 + C_{2211}\nu_g^2, \quad
b'= C_{2121}\nu_g^1 + C_{2221}\nu_g^2, \\
c'= C_{2112}\nu_g^1 + C_{2212}\nu_g^2, \quad
d'= C_{2122}\nu_g^1 + C_{2222}\nu_g^2
\end{align*}
in such a way that
$$
\eta_1 = \biggl( C_{u_g}\nabla\Bigl(\frac{\partial v_g}{\partial z_2}\Bigr) [\nu_g] \biggr) \cdot e_1, \qquad
\eta_2 = \biggl( C_{u_g}\nabla\Bigl(\frac{\partial v_g}{\partial z_2}\Bigr) [\nu_g] \biggr) \cdot e_2.
$$
Hence by \eqref{stima0lemma} we have
\begin{equation} \label{stima1lemma}
\|\eta_i\|_{W^{-\frac1p,p}_\#(\Gamma_g)} \to0 \qquad\text{as }\delta\to0
\end{equation}
(uniformly with respect to $g\in\mathcal{U}_\delta$).
Moreover, observe that we can write each $\vartheta_{ij}$ as a tangential derivative on $\Gamma_g$:
$$
\vartheta_{ij} = \partial_{\tau_g} \Bigl(\frac{\partial v^j_g}{\partial z_i}\Bigr)
= \nabla \Bigl(\frac{\partial v^j_g}{\partial z_i}\Bigr)\cdot(\nu_g^2,-\nu_g^1),
$$
so that by \cite[Theorem~8.6]{FM} we also have
\begin{equation} \label{stima2lemma}
\| \vartheta_{ij} \|_{W^{-\frac1p,p}_\#(\Gamma_g)}
\leq C \Big\| \nabla\Bigl(\frac{\partial v_g^j}{\partial z_i}\Bigr) \Big\|_{L^p(\Omega_g;\R^2)}
\to0 \qquad\text{as }\delta\to0
\end{equation}
(uniformly with respect to $g\in\mathcal{U}_\delta$).
To conclude, observe that the $6\times6$ matrix in \eqref{sistema6x6}
has coefficients uniformly bounded in $C^{0,\alpha}$ with respect to $g\in\mathcal{U}_\delta$,
for $\alpha=1-\frac{2}{p}>\frac{1}{p}$ (as $p>4$);
if we are able to show that its determinant is uniformly positive,
then we can invert the relations in \eqref{sistema6x6} and express $\sigma_{ijk}$
as linear combinations of the quantities estimated in \eqref{stima1lemma} and \eqref{stima2lemma},
and in turn \eqref{claim1lemma} follows by Lemma~\ref{lem:sobolev}.
Hence we are left with the computation of the determinant of the $6\times6$ matrix $M$ appearing in \eqref{sistema6x6},
which turns out to be equal to
$$
\det M = (\nu_g^2(z))^2 \det Q_g(z),
$$
which is uniformly positive as observed before. This concludes the proof of Step 1 in the case $N=2$.

In the three-dimensional case we follow the same strategy.
We observe that, setting
\begin{equation} \label{matrice2}
\eta_{ik} := \biggl( C_{u_g}\nabla\Bigl(\frac{\partial v_g}{\partial z_k}\Bigr) [\nu_g] \biggr) \cdot e_i
\qquad\text{for }i,k=1,2,3,
\end{equation}
by \eqref{stima0lemma} we have
\begin{equation*}
\|\eta_{ik}\|_{W^{-\frac1p,p}_\#(\Gamma_g)} \to0 \qquad\text{as }\delta\to0
\end{equation*}
(uniformly with respect to $g\in\mathcal{U}_\delta$).
Moreover by Theorem~\ref{teo:sobolev4} we have also a similar estimate for the quantities $\vartheta_{ijk}:=\sigma_{ikj}\nu_g^3-\sigma_{i3j}\nu_g^k$
for $i,j=1,2,3$ and $k=1,2$, namely
$$
\|\vartheta_{ijk}\|_{W^{-\frac1p,p}_\#(\Gamma_g)}
\leq C \Big\|\nabla\Bigl(\frac{\partial v_g^j}{\partial z_i}\Bigr)\Big\|_{L^p(\Omega_g;\R^3)} \to0
$$
as $\delta\to 0$, uniformly with respect to $g\in\mathcal{U}_\delta$.
Hence we can write a linear system similar to \eqref{sistema6x6}
by choosing 18 among the 27 quantities $\vartheta_{ijk}$, $\eta_{ik}$
to be expressed as combinations of the 18 (different) terms $\sigma_{ijk}$:
precisely, we consider $\eta_{ik}$ for $k=3$ and $i=1,2,3$,
and all the $\vartheta_{ijk}$ except for $\vartheta_{211}$, $\vartheta_{221}$, $\vartheta_{231}$.
As before, the (computer assisted) computation of the determinant of the $18\times18$ matrix of the system obtained in this way
shows that this coincides (up to a sign) with $(\nu_g^3(z))^{12}\det Q_g(z)$,
which is uniformly positive (see Section~\ref{sect:appdeterminante} for more details).
Inverting these relations we can then write each term $\sigma_{ijk}$ as a linear combination
of the quantities $\vartheta_{ijk}$, $\eta_{ik}$,
and from the previous estimates the claim follows, again using Lemma~\ref{lem:sobolev}.

\smallskip
\noindent{\it Step 3.}
We claim that there exists a constant $C$, independent of $g\in\mathcal{U}_\delta$,
such that for every $\vphi\in W^{\frac1p,\frac{p}{p-1}}_\#(\Gamma_h)$
\begin{equation} \label{claim2lemma}
\big\| W_\xi(\nabla u_g\circ\Phi_g) \vphi \big\|_{W^{\frac1p,\frac{p}{p-1}}(\Gamma_h;\M^N)}
\leq C \, \| \vphi \|_{W^{\frac1p,\frac{p}{p-1}}(\Gamma_h)}.
\end{equation}
In fact, we use Theorem~\ref{teo:sobolev3} to extend $\vphi$ to a function $\tilde{\vphi}\in W_\#^{1,\frac{p}{p-1}}(\Omega_h)$.
Note that, by the Sobolev Imbedding Theorem, setting $q:=\frac{Np}{Np-N-p}$ we have
$$
\|\tilde{\vphi}\|_{L^q(\Omega_h)}
\leq C \|\tilde{\vphi}\|_{W^{1,\frac{p}{p-1}}(\Omega_h)}
\leq C \|\vphi\|_{W^{\frac1p,\frac{p}{p-1}}(\Gamma_h)}
$$
for some constant $C$ independent of $g$ (the second inequality still follows from Theorem~\ref{teo:sobolev3}).
Hence, using H\"{o}lder inequality, we deduce that
\begin{align*}
\| W_\xi&(\nabla u_g\circ\Phi_g)\vphi \|_{W^{\frac1p,\frac{p}{p-1}}(\Gamma_h;\M^N)}
 \leq C \, \| W_\xi(\nabla u_g\circ\Phi_g) \tilde{\vphi} \|_{W^{1,\frac{p}{p-1}}(\Omega_h;\M^N)} \\
&\leq C \, \| W_\xi(\nabla u_g\circ\Phi_g) \|_{L^\infty(\Omega_h;\M^N)} \|\tilde{\vphi}\|_{L^{\frac{p}{p-1}}(\Omega_h)}
 + C \, \big\| \nabla\bigl( W_\xi(\nabla u_g\circ\Phi_g) \tilde{\vphi} \bigr) \big\|_{L^{\frac{p}{p-1}}(\Omega_h)} \\
&\leq C \, \| W_\xi(\nabla u_g\circ\Phi_g) \|_{L^\infty(\Omega_h;\M^N)} \|\tilde{\vphi}\|_{W^{1,\frac{p}{p-1}}(\Omega_h)}
 + C \, \| \tilde{\vphi} \|_{L^q(\Omega_h)} \big\| \nabla \bigl( W_\xi(\nabla u_g\circ\Phi_g) \bigr) \big\|_{L^N(\Omega_h)} \\
&\leq C \Bigl[ \| W_\xi(\nabla u_g\circ\Phi_g) \|_{L^\infty(\Omega_h;\M^N)}
 + \|C_{u_g}\circ\Phi_g \|_{L^\infty(\Omega_h)} \|\nabla^2u_g\circ\Phi_g\|_{L^p(\Omega_h)} \Bigr]
 \|\vphi\|_{W^{\frac1p,\frac{p}{p-1}}(\Gamma_h)}.
\end{align*}
From this estimate, recalling the equiboundedness of $u_g\circ\Phi_g$ in $W^{2,p}(\Omega_h)$,
we obtain that \eqref{claim2lemma} holds
with a constant $C$ depending also on the $C^2$-norm of $W$ on $K$,
where $K$ is the compact subset of $\M^N_+$ given by Remark~\ref{rm:IFT}.

\smallskip
\noindent{\it Step 4.}
We now conclude the proof of the lemma.
For every $\vphi\in W^{\frac1p,\frac{p}{p-1}}_\#(\Gamma_h)$, and for $k=1,\ldots,N$ we have
\begin{align*}
\int_{\Gamma_h} \biggl[ \frac{\partial}{\partial z_k}&\,W(\nabla u) - \Bigl( \frac{\partial}{\partial z_k}\,W(\nabla u_g) \Bigr) \circ\Phi_g  \biggr] \vphi\,d\hn \\
&= \int_{\Gamma_h} \Bigl( W_\xi(\nabla u) - W_\xi(\nabla u_g)\circ\Phi_g \Bigr)
 : \nabla\Bigl(\frac{\partial u}{\partial z_k}\Bigr) \vphi\,d\hn \\
&\hspace{0.5cm}+ \int_{\Gamma_h} W_\xi(\nabla u_g) \circ\Phi_g
 : \Bigl[ \nabla\Bigl(\frac{\partial u}{\partial z_k}\Bigr) - \nabla\Bigl(\frac{\partial u_g}{\partial z_k}\Bigr)\circ\Phi_g \Bigr] \vphi\,d\hn \\
&\leq C \big\| W_\xi(\nabla u) - W_\xi(\nabla u_g) \circ\Phi_g \big\|_{L^\infty(\Gamma_h;\M^N)} \| \vphi \|_{L^{\frac{p}{p-1}}(\Gamma_h)} \\
&\hspace{0.5cm}+ \big\| W_\xi(\nabla u_g\circ\Phi_g) \vphi \big\|_{W^{\frac1p,\frac{p}{p-1}}(\Gamma_h;\M^N)}
 \Big\| \nabla\Bigl(\frac{\partial u}{\partial z_k}\Bigr) - \nabla\Bigl(\frac{\partial u_g}{\partial z_k}\Bigr)\circ\Phi_g \Big\|_{W^{-\frac1p,p}_\#(\Gamma_h;\M^N)},
\end{align*}
where $C$ is a positive constant depending only on the $C^2$-norm of $u$ and on $\hn(\Gamma_h)$.
Hence, since $\sup_{g\in\mathcal{U}_\delta}\big\| W_\xi(\nabla u) - W_\xi(\nabla u_g) \circ\Phi_g \big\|_{L^\infty(\Gamma_h;\M^N)}\to0$ as $\delta\to0$,
recalling \eqref{claim1lemma} and \eqref{claim2lemma} we obtain that
$$
\sup_{g\in\mathcal{U}_\delta}\big\| \nabla (W(\nabla u)) - \nabla (W(\nabla u_g)) \circ\Phi_g \big\|_{W^{-\frac1p,p}_\#(\Gamma_h;\R^N)} \to 0
\qquad\text{as }\delta\to0,
$$
and the conclusion of the lemma follows from Lemma~\ref{lem:sobolev}.
\end{proof}

We can now prove Theorem~\ref{teo:locmin}
by reproducing the strategy of \cite{FM} with easy modifications.
For the sake of completeness and for the reader's convenience we will work out all the details of the proof.

\begin{proof}[Proof of Theorem~\ref{teo:locmin}]
Let $\delta>0$ to be chosen
and consider any $g\in\mathcal{U}_\delta$.
We will denote by $\B_g$ and $H_g$ the second fundamental form and the mean curvature of $\Gamma_g$ respectively,
and by $\Bpsi_g$, $H_g^\psi$ the ``anisotropic versions'' of the same quantities.
We define the bilinear form on $\Htilde(\Gamma_g)$
$$
(\vphi,\vartheta)_{\sim,g} := \int_{\Gamma_g}(\nabla^2\psi\circ\nu_g)[\nabla_{\Gamma_g}\vphi, \nabla_{\Gamma_g}\vartheta]\,d\hn
+\int_{\Gamma_g}a_g\,\vphi\,\vartheta\,d\hn
$$
where $a_g:= \partial_{\nu_g}(W\circ\nabla u_g)-\tr(\Bpsi_g \B_g)$ on $\Gamma_g$,
and we set $\|\vphi\|_{\sim,g}^2:=(\vphi,\vphi)_{\sim,g}$.
We omit the subscript in all the analogous quantities defined on $\Gamma_h$,
according to the notation introduced in Section~\ref{sect:varII}.
We now split the proof into four steps.

\smallskip
\noindent{\it Step 1.}
We start by observing that for every $\vphi\in \Htilde(\Gamma_g)$
\begin{equation} \label{stimap-1}
\int_{\Gamma_h} \bigl( a(\jacg)^2 - (a_g\circ\Phi_g)\jacg \bigr) (\vphi\circ\Phi_g)^2 \,d\hn
\leq c(\delta)\|\vphi\|^2_{H^1(\Gamma_g)}\,,
\end{equation}
where $c(\delta)\to0$ as $\delta\to0$ (independently of $g\in\mathcal{U}_\delta$).
Indeed, by using Lemma~\ref{lem:algebra} and recalling that $\|\jacg-1\|_{L^\infty(\Gamma_h)}\to0$ as $\delta\to0$, we have
\begin{align*}
\int_{\Gamma_h} \bigl( \partial_\nu(W(\nabla u))(\jacg)^2 &- (\partial_{\nu_g}(W(\nabla u_g))\circ\Phi_g)\jacg \bigr) (\vphi\circ\Phi_g)^2 \,d\hn \\
&\leq c'(\delta) \|(\vphi\circ\Phi_g)^2\|_{W^{\frac1p,\frac{p}{p-1}}_\#(\Gamma_h)}
\leq c'(\delta) \|(\vphi\circ\Phi_g)^2\|_{W^{1,\frac{p}{p-1}}(\Gamma_h)} \\
&\leq c''(\delta) \|\vphi\circ\Phi_g\|^2_{H^1(\Gamma_h)}
\leq c'''(\delta)\|\vphi\|^2_{H^1(\Gamma_g)}\,,
\end{align*}
where the third inequality can be deduced by recalling the imbedding of $H^1(\Gamma_h)$ in $L^q(\Gamma_h)$ for every $q$,
which holds in dimension $N\leq3$. Here $c'(\delta),c''(\delta),c'''(\delta)\to0$ as $\delta\to0$,
independently of $g\in\mathcal{U}_\delta$.
Moreover, it is not hard to see that
$$
\sup_{g\in\mathcal{U}_\delta } \big\|\tr(\Bpsi_g\B_g)\circ\Phi_g - \tr(\Bpsi\B)\big\|_{L^{p/2}(\Gamma_h)} \to 0
\quad\text{as }\delta\to0\,,
$$
from which follows by H\"{o}lder inequality
(again using $\|\jacg-1\|_{L^\infty(\Gamma_h)}\to0$)
\begin{align*}
\int_{\Gamma_h} \bigl( \tr(\Bpsi_g\B_g)\circ\Phi_g &- \tr(\Bpsi\B)\jacg \bigr)\jacg (\vphi\circ\Phi_g)^2 \,d\hn \\
&\leq c'(\delta) \|(\vphi\circ\Phi_g)^2\|_{L^{\frac{p}{p-2}}(\Gamma_h)}
= c'(\delta) \|\vphi\circ\Phi_g\|_{L^{\frac{2p}{p-2}}(\Gamma_h)}^2 \\
&\leq c''(\delta) \|\vphi\circ\Phi_g\|^2_{H^1(\Gamma_h)}
\leq c'''(\delta)\|\vphi\|^2_{H^1(\Gamma_g)}\,,
\end{align*}
where the second inequality is justified, as before, by the Sobolev Embedding Theorem.
By combining the previous estimates, \eqref{stimap-1} follows.

\smallskip
\noindent{\it Step 2.}
We claim that if $\delta$ is sufficiently small then for every $g\in\mathcal{U}_\delta$
\begin{equation} \label{equivnorm}
\|\vphi\|^2_{\sim,g} \geq C_1 \|\vphi\|^2_{H^1(\Gamma_g)} \qquad\text{for every }\vphi\in\Htilde(\Gamma_g)
\end{equation}
for some positive constant $C_1$.
To prove \eqref{equivnorm}, we first note that for every $\vartheta\in\Htilde(\Gamma_h)$ one has,
thanks to \eqref{fquadT} and to Corollary~\ref{cor:defpos},
$$
\|\vartheta\|^2_\sim \geq \partial^2F(h,u)[\vartheta] \geq C\|\vartheta\|^2_{H^1(\Gamma_h)}.
$$
For $\vphi\in\Htilde(\Gamma_g)$ we define
$\tilde{\vphi} := (\vphi\circ\Phi_g)\jacg \in\Htilde(\Gamma_h)$;
then, using the area formula we have
\begin{align*}
\|\vphi\|^2_{H^1(\Gamma_g)}
&=  \int_{\Gamma_g} (\vphi^2+|\nabla_{\Gamma_g}\vphi|^2)\,d\hn
= \int_{\Gamma_h} \bigl( (\vphi\circ\Phi_g)^2  +  |(\nabla_{\Gamma_g}\vphi)\circ\Phi_g|^2 \bigr) \jacg \,d\hn \\
&\leq C' \|\tilde{\vphi}\|^2_{H^1(\Gamma_h)} \leq \frac{C'}{C} \,\|\tilde{\vphi}\|^2_\sim
\end{align*}
for some positive constant $C'$ independent of $g\in\mathcal{U}_\delta$.
Now
\begin{align} \label{stimap0}
\|\tilde{\vphi}\|^2_\sim
&= \int_{\Gamma_h} \bigl( a\,\tilde{\vphi}^2
 + (\nabla^2\psi\circ\nu)[\nabla_{\Gamma_h}\tilde{\vphi}, \nabla_{\Gamma_h}\tilde{\vphi}] \, \bigr) \,d\hn \nonumber \\
&= \| \vphi \|^2_{\sim,g}
 + \int_{\Gamma_h} \bigl( a(\jacg)^2-(a_g\circ\Phi_g)\jacg \bigr) (\vphi\circ\Phi_g)^2 \,d\hn \nonumber \\
&\hspace{1cm} + \int_{\Gamma_h}(\nabla^2\psi\circ\nu)[\nabla_{\Gamma_h}\tilde{\vphi}, \nabla_{\Gamma_h}\tilde{\vphi}] \,d\hn \nonumber \\
&\hspace{1cm} - \int_{\Gamma_h}(\nabla^2\psi\circ\nu_g\circ\Phi_g)[(\nabla_{\Gamma_g}\vphi)\circ\Phi_g,(\nabla_{\Gamma_g}\vphi)\circ\Phi_g] \jacg \,d\hn \nonumber\\
&\leq \|\vphi\|^2_{\sim,g} + c(\delta)\|\vphi\|^2_{H^1(\Gamma_g)},
\end{align}
where $c(\delta)$ tends to 0 as $\delta\to0$.
To deduce the last inequality in the previous estimate we used in particular \eqref{stimap-1}
and the fact that $\|\Phi_g-Id\|_{W^{2,p}(\Gamma_h;\R^N)}\to 0$.
Choosing $\delta$ sufficiently small and combining the previous estimates
the claim follows.

\smallskip
\noindent{\it Step 3.}
By Step 2 we can define a compact linear operator $T_g:\Htilde(\Gamma_g)\to\Htilde(\Gamma_g)$ by duality:
\begin{equation} \label{Tg}
(T_g\vphi,\vartheta)_{\sim,g} = \int_{\Gamma_g} \div_{\Gamma_g}(\vartheta \, W_\xi(\nabla u_g)) \cdot \vf \,d\hn
= \int_{\Omega_g} C_{u_g} \nabla \vf : \nabla v_{\vartheta} \,dz
\end{equation}
for every $\vphi,\vartheta\in\Htilde(\Gamma_g)$,
where for $\zeta\in\Htilde(\Gamma_g)$ we denote by $v_\zeta$ the unique solution in $\Vtilde(\Omega_g)$ to the equation
\begin{equation} \label{vfg}
\int_{\Omega_g} C_{u_g} \nabla v_\zeta : \nabla w \,dz = \int_{\Gamma_g}\div_{\Gamma_g}(\zeta \, W_{\xi}(\nabla u_g)) \cdot w \,d\hn
\qquad\text{for every } w\in\Vtilde(\Omega_g).
\end{equation}
Setting, similarly to \eqref{lambda1},
$$
\lambda_{1,g}:= \max_{\|\vphi\|_{\sim,g}=1} (T_g\vphi,\vphi)_{\sim,g},
$$
we claim that
\begin{equation} \label{claimp0}
\lambda_\infty:=\limsup_{\|g-h\|_{W^{2,p}(Q)}\to0} \lambda_{1,g}\leq\lambda_1.
\end{equation}
Indeed, let $(g_n)_n$ be a sequence in $C^{\infty}_\#(Q)$ converging to $h$ in $W^{2,p}(Q)$, $|\Omega_{g_n}|=|\Omega_h|$,
such that
$$
\lambda_\infty=\lim_{n\to+\infty}\lambda_{1,g_n},
$$
and let $u_n$ be the corresponding critical points for the elastic energy in $\Omega_{g_n}$.
Let $\vphi_n\in\Htilde(\Gamma_{g_n})$,
with $\|\vphi_n\|_{\sim,g_n}=1$, be such that
$$
\lambda_{1,g_n} = (T_{g_n}\vphi_n, \vphi_n)_{\sim,g_n} = \int_{\Omega_{g_n}}C_{u_n}\nabla v_{\vphi_n} : \nabla v_{\vphi_n},
$$
where $v_{\vphi_n}$ is defined as in \eqref{vfg}.
We set $\tilde{\vphi}_n := c_n(\vphi_n\circ\Phi_{g_n})\jacn$,
where $c_n := \|(\vphi_n\circ\Phi_{g_n})\jacn\|^{-1}_{\sim}$,
so that $\tilde{\vphi}_n\in\Htilde(\Gamma_h)$ and $\|\tilde{\vphi}_n\|_\sim=1$.
Setting also $w_n:=v_{\vphi_n}\circ\Phi_{g_n}$,
by a change of variables it follows that for every $w\in\Vtilde(\Omega_{g_n})$
\begin{align*}
\int_{\Omega_{g_n}} C_{u_n} \nabla v_{\vphi_n} : \nabla w \,dz
= \int_{\Omega_h} A_n \nabla w_n : \nabla (w\circ\Phi_{g_n}) \,dz,
\end{align*}
where $A_n$ is the fourth order tensor defined by
$$
A_nM =
\Bigl( W_{\xi\xi}\bigl( \nabla(u_n\circ\Phi_{g_n})(\nabla\Phi_{g_n})^{-1} \bigr)
\bigl( M (\nabla\Phi_{g_n})^{-1} \bigr) \Bigr)
(\nabla\Phi_{g_n})^{-T} \det\nabla\Phi_{g_n}
\qquad\text{for }M\in\M^N.
$$
Hence by \eqref{vfg} we see that $w_n\in\Vtilde(\Omega_h)$ solves the equation
\begin{equation} \label{equazionew}
\int_{\Omega_h}A_n\nabla w_n : \nabla w \,dz
= \int_{\Gamma_h} \bigl( \div_{\Gamma_{g_n}}(\vphi_n W_\xi(\nabla u_n)) \circ \Phi_{g_n} \bigr) \cdot w \; \jacn \,d\hn
\end{equation}
for every $w\in\Vtilde(\Omega_{h})$.
Let us observe also that $A_n\to C_u$ uniformly in $\overline{\Omega}_h$.
We now claim that
\begin{equation} \label{claimp1}
\lim_{n\to\infty} \int_{\Omega_h}C_u \nabla v_{\tilde{\vphi}_n} : \nabla v_{\tilde{\vphi}_n} \,dz
= \lim_{n\to\infty} \int_{\Omega_h}A_n \nabla w_n : \nabla w_n \,dz.
\end{equation}
Notice that this implies \eqref{claimp0}, since
\begin{align*}
\lambda_1 &\geq \lim_{n\to\infty} (T\tilde{\vphi}_n,\tilde{\vphi}_n)_\sim
= \lim_{n\to\infty} \int_{\Omega_h} C_u \nabla v_{\tilde{\vphi}_n} : \nabla v_{\tilde{\vphi}_n} \,dz \\
&= \lim_{n\to\infty} \int_{\Omega_h}A_n \nabla w_n : \nabla w_n \,dz
= \lim_{n\to\infty} \int_{\Omega_{g_n}} C_{u_n} \nabla v_{\vphi_n} : \nabla v_{\vphi_n} \,dz =\lambda_{\infty}.
\end{align*}

In order to prove \eqref{claimp1}, we need to deduce some preliminary estimates.
Using the equation satisfied by $v_{\vphi_n}$ and recalling \eqref{provv} we have
\begin{align*}
\frac{c_0}{4} \|v_{\vphi_n}\|^2_{H^1(\Omega_{g_n};\R^N)}
&\leq \int_{\Omega_{g_n}} C_{u_n} \nabla v_{\vphi_n} : \nabla v_{\vphi_n} \,dz
= \int_{\Gamma_{g_n}} \div_{\Gamma_{g_n}} \bigl( \vphi_n W_\xi(\nabla u_n) \bigr) \cdot v_{\vphi_n} \,d\hn \\
&\leq \| \div_{\Gamma_{g_n}} \bigl( \vphi_n W_\xi(\nabla u_n) \bigr) \|_{H^{-\frac12}_\#(\Gamma_{g_n};\R^N)} \|v_{\vphi_n}\|_{H^{\frac12}(\Gamma_{g_n};\R^N)},
\end{align*}
and since the $H^{-\frac12}$-norm in the previous expression is uniformly bounded by Lemma~\ref{lem:gagliardo}
(recall that $\vphi_n$ are uniformly bounded in $H^1(\Gamma_{g_n})$,
and that $W_\xi(\nabla u_n)$ are uniformly bounded in $C^{0,\alpha}(\overline{\Omega}_{g_n};\M^N)$ with $\alpha=1-\frac{N}{p}>\frac12$),
we deduce that
\begin{equation} \label{stimap1}
\sup_n \| v_{\vphi_n} \|_{H^1(\Omega_{g_n};\R^N)} < \infty.
\end{equation}
Moreover we have also
\begin{equation} \label{stimap2}
\sup_n \| w_n \|_{H^1(\Omega_{h};\R^N)} < \infty, \qquad
\sup_n \| v_{\tilde{\vphi}_n} \|_{H^1(\Omega_{h};\R^N)} < \infty,
\end{equation}
where the first estimate follows from \eqref{stimap1}, using the definition of $w_n$.
Finally, arguing as in the proof of the estimate \eqref{stimap0}
with $\tilde{\vphi}$ replaced by $\frac{\tilde{\vphi}_n}{c_n}$
and $\vphi$ replaced by $\vphi_n$, we obtain $c_n \to 1$.

Now we are ready to prove \eqref{claimp1}, from which the conclusion follows.
Observe that, thanks to the uniform bound \eqref{stimap2} and to the uniform convergence of $A_n$ to $C_u$, we have
$$
\lim_{n\to\infty} \int_{\Omega_h}C_u \nabla w_n : \nabla w_n \,dz
= \lim_{n\to\infty} \int_{\Omega_h}A_n \nabla w_n : \nabla w_n \,dz,
$$
thus claim \eqref{claimp1} will follow from
\begin{equation} \label{claimp2}
\lim_{n\to\infty} \int_{\Omega_h} C_u \nabla(v_{\tilde{\vphi}_n}-w_n) : \nabla(v_{\tilde{\vphi}_n}-w_n) \,dz=0,
\end{equation}
since this implies that $v_{\tilde{\vphi}_n}-w_n$ tends to $0$ strongly in $H^1(\Omega_h;\R^N)$.
Hence we are left with the proof of \eqref{claimp2}.

Observe that, as $v_{\tilde{\vphi}_n} - w_n$ is an admissible test function for
both the equations satisfied by $v_{\tilde{\vphi}_n}$ and $w_n$, we have
\begin{align*}
\int_{\Omega_h} C_u &\nabla(v_{\tilde{\vphi}_n}-w_n) : \nabla(v_{\tilde{\vphi}_n}-w_n) \,dz \\
&= \int_{\Omega_h} C_u \nabla v_{\tilde{\vphi}_n} : \nabla(v_{\tilde{\vphi}_n}-w_n) \,dz
- \int_{\Omega_h} (C_u-A_n) \nabla w_n : \nabla(v_{\tilde{\vphi}_n}-w_n) \,dz \\
&\hspace{1cm}- \int_{\Omega_h} A_n \nabla w_n : \nabla(v_{\tilde{\vphi}_n}-w_n) \,dz \\
&= \int_{\Gamma_h} \div_{\Gamma_h} (\tilde{\vphi}_n W_\xi(\nabla u)) \cdot (v_{\tilde{\vphi}_n}-w_n) \,d\hn
- \int_{\Omega_h} (C_u-A_n) \nabla w_n : \nabla(v_{\tilde{\vphi}_n}-w_n) \,dz \\
&\hspace{1cm}- \int_{\Gamma_h} \bigl( \div_{\Gamma_{g_n}}(\vphi_n W_\xi(\nabla u_n))\circ\Phi_{g_n} \bigr) \cdot (v_{\tilde{\vphi}_n}-w_n) \jacn\,d\hn \\
&=: I_1-I_2-I_3.
\end{align*}
It is clear, from the bounds in \eqref{stimap2}
and from the uniform convergence of $A_n$ to $C_u$, that the second integral $I_2$ tends to $0$.
Since, thanks to \eqref{stimap2}, $v_{\tilde{\vphi}_n}-w_n$ is bounded in $H^{\frac12}(\Gamma_h;\R^N)$,
to prove that also the difference $I_1-I_3$ tends to $0$
it will be sufficient to show that
$$
\big\|
\div_{\Gamma_h} (\tilde{\vphi}_n W_\xi(\nabla u))
-
\div_{\Gamma_{g_n}}(\vphi_n W_\xi(\nabla u_n))\circ\Phi_{g_n}
\big\|_{H^{-\frac12}_\#(\Gamma_h;\R^N)}\to0\,.
$$
In turn, by Lemma~\ref{lem:sobolev} the previous convergence will follow from
\begin{equation} \label{stimap3}
\| \tilde{\vphi}_n h_n \|_{H^{\frac12}_\#(\Gamma_h;\M^N)} \to 0,
\end{equation}
where
$$
h_n:= W_\xi(\nabla u)
- c_n^{-1} (\jacn)^{-1} W_\xi(\nabla u_n) \circ\Phi_{g_n}.
$$
Recalling that $c_n\to1$, we have that $h_n\to0$ in $C^{0,\alpha}(\Gamma_{h};\M^N)$
for $\alpha=1-\frac{N}{p}$;
hence by Lemma~\ref{lem:gagliardo} we obtain \eqref{stimap3}, which concludes the proof of Step 3.

\smallskip
\noindent{\it Step 4.}
We define $h_t:=h+t(g-h)$ for $t\in[0,1]$.
Setting $f(t):=F(h_t,u_{h_t})$, we claim that if $\delta$ is sufficiently small then
\begin{equation} \label{claimp3}
f''(t)> 2C_2 \|\vphi_g\|^2_{H^1(\Gamma_g)} \qquad\text{for every }t\in[0,1]
\end{equation}
for some positive constant $C_2$, where $\vphi_g:=\bigl((g-h)/\sqrt{1+|\nabla g|^2}\bigr)\circ\pi$.
In fact, the quantity $f''(t)$ is nothing but the second variation of $F$ at $(h_t,u_{h_t})$ along the direction $g-h$,
hence by Remark~\ref{rm:var2s}
\begin{align} \label{eqp1}
f''(t)
&= -(T_{h_t}\vphi_t,\vphi_t)_{\sim,h_t} + \|\vphi_t\|^2_{\sim,h_t} \nonumber\\
&- \int_{\Gamma_{h_t}} \bigl( W(\nabla u_{h_t}) + H^\psi_{h_t} \bigr) \: \div_{\Gamma_{h_t}} \Biggl[ \biggl( \frac{(\nabla h_t,|\nabla h_t|^2)}{\sqrt{1+|\nabla h_t|^2}}\circ\pi \biggr) \vphi_t^2 \Biggr]\,d\hn,
\end{align}
where $\vphi_t:=\bigl((g-h)/\sqrt{1+|\nabla h_t|^2}\bigr)\circ\pi\in\Htilde(\Gamma_{h_t})$.
Observe that, as $\lambda_1<1$ by Theorem~\ref{teo:condequiv},
combining Step 2 and Step 3 we have that for $\delta$ sufficiently small
\begin{align} \label{eqp2}
-(T_{h_t}\vphi_t,\vphi_t)_{\sim,h_t} &+ \|\vphi_t\|^2_{\sim,h_t}
\geq (1-\lambda_{1,h_t})\|\vphi_t\|^2_{\sim,h_t}
> \frac{1-\lambda_1}{2}\|\vphi_t\|^2_{\sim,h_t} \nonumber\\
&\geq \frac{C_1(1-\lambda_1)}{2}\|\vphi_t\|^2_{H^1(\Gamma_{h_t})}
\geq \frac{C_1(1-\lambda_1)}{4}\|\vphi_g\|^2_{H^1(\Gamma_{g})},
\end{align}
where in the last inequality we used the fact that, for $\delta$ small enough,
\begin{equation} \label{eqp3}
\frac12 \|\vphi_g\|^2_{H^1(\Gamma_g)} \leq \|\vphi_t\|^2_{H^1(\Gamma_{h_t})} \leq 2\|\vphi_g\|^2_{H^1(\Gamma_g)}.
\end{equation}
In addition, as $(h,u)$ is a critical pair,
there exists a constant $\Lambda$ such that $W(\nabla u) + H^\psi\equiv\Lambda$ on $\Gamma_h$, and moreover
\begin{equation} \label{eqp3bis}
\sup_{g\in\mathcal{U}_\delta}\sup_{t\in[0,1]} \bigl\| W(\nabla u_{h_t}) + H^\psi_{h_t}-\Lambda \bigr\|_{L^p(\Gamma_{h_t})}\to0 \qquad\text{as }\delta\to0.
\end{equation}
From this it follows that if $\delta$ is sufficiently small, by H\"{o}lder inequality
\begin{align}\label{eqp4}
- \int_{\Gamma_{h_t}} \bigl( & W(\nabla u_{h_t}) + H^\psi_{h_t} \bigr) \: \div_{\Gamma_{h_t}}
 \biggl[ \Bigl( {\textstyle\frac{(\nabla h_t,|\nabla h_t|^2)}{\sqrt{1+|\nabla h_t|^2}}\circ\pi} \Bigr) \vphi_t^2 \biggr] \,d\hn \nonumber\\
&= - \int_{\Gamma_{h_t}} \bigl( W(\nabla u_{h_t}) + H^\psi_{h_t} - \Lambda \bigr) \: \div_{\Gamma_{h_t}}
 \biggl[ \Bigl( {\textstyle\frac{(\nabla h_t,|\nabla h_t|^2)}{\sqrt{1+|\nabla h_t|^2}}\circ\pi} \Bigr) \vphi_t^2 \biggr] \,d\hn \nonumber\\
& \geq - \bigl\| W(\nabla u_{h_t}) + H^\psi_{h_t}-\Lambda \bigr\|_{L^p(\Gamma_{h_t})} \biggl\{
 \Bigl\| \div_{\Gamma_{h_t}} \Bigl( {\textstyle\frac{(\nabla h_t,|\nabla h_t|^2)}{\sqrt{1+|\nabla h_t|^2}}\circ\pi} \Bigr) \Bigr\|_{L^p(\Gamma_{h_t})}
 \| \vphi_t \|^2_{L^{\frac{2p}{p-2}}(\Gamma_{h_t})} \nonumber\\
&\hspace{2cm} + 2 \bigl\| \nabla_{\Gamma_{h_t}}\vphi_t \bigr\|_{L^2(\Gamma_{h_t};\R^N)}
 \Bigl\| \vphi_t {\textstyle\frac{(\nabla h_t,|\nabla h_t|^2)}{\sqrt{1+|\nabla h_t|^2}}\circ\pi} \Bigr\|_{L^{\frac{2p}{p-2}}(\Gamma_{h_t};\R^N)} \biggr\} \nonumber\\
&\geq - \frac{C_1(1-\lambda_1)}{8}\|\vphi_g\|^2_{H^1(\Gamma_g)},
\end{align}
where in the last inequality we used also
the boundedness of $h_t$ in $W^{2,p}(Q)$, the Sobolev imbedding theorem, \eqref{eqp3bis} and \eqref{eqp3}.

Collecting \eqref{eqp1}, \eqref{eqp2} and \eqref{eqp4} we conclude
that the claim \eqref{claimp3} holds with $C_2= \frac{C_1(1-\lambda_1)}{16}$.

Finally, thank to the fact that $f'(0)=0$ (as $(h,u)$ is a critical pair), we have
\begin{align} \label{eqp5}
F(h,u) = f(0) = f(1)-\int_0^1(1-t)f''(t)\,dt < F(g,u_g) - C_2 \|\vphi_g\|^2_{H^1(\Gamma_g)}.
\end{align}
This inequality is valid for every $g\in\mathcal{U}_\delta$, for a sufficiently small $\delta$.
Now, by an approximation argument,
if $g\in AP(Q)$ is such that $\|g-h\|_{W^{2,p}(Q)}<\delta$ and $|\Omega_g|=|\Omega_h|$,
we set $\tilde{g}:=h+\rho_\varepsilon*(g-h)$,
where $\rho_\varepsilon$ is a standard mollifier with support in $B_\varepsilon(0)$.
Then $\tilde{g}\in\mathcal{U}_\delta$, and $\varepsilon$ can be chosen so small that
$$
F(\tilde{g},u_{\tilde{g}})\leq F(g,u_g) + \frac{C_2}{2} \| \vphi_{\tilde{g}} \|^2_{H^1(\Gamma_{\tilde{g}})},
$$
hence by \eqref{eqp5}
$$
F(h,u) < F(g,u_g) - \frac{C_2}{2} \|\vphi_{\tilde{g}}\|^2_{H^1(\Gamma_{\tilde{g}})}.
$$
Now the minimality with respect to a generic pair $(g,v)$ follows from Proposition~\ref{prop:minEl}.
\end{proof}

\end{section}


\begin{section}{Strong local minimality} \label{sect:WimpliesL}

In the main result of this section (Theorem~\ref{teo:locmin2})
we prove that the $W^{2,p}$-local minimality (see Definition~\ref{def:locmindeb})
implies local minimality in the stronger sense of Definition~\ref{def:locmin}.
In particular, by Theorem~\ref{teo:locmin} we deduce that the strict stability of a critical pair $(h,u)$
is a sufficient condition for local minimality (Theorem~\ref{cor:locmin}).
We will also observe, in Theorem~\ref{teo:linelas}, that our methods provide the isolated local minimality in the case of the linear elasticity.

The contradiction argument which leads to the proof of these result
is mainly based on the regularity properties of the solutions to suitable penalization problems,
which will turn out to be \emph{quasi-minimizers} of the anisotropic perimeter, according to the following definition.
For every finite-perimeter set $E$ we denote by $\partial^*E$ its reduced boundary and by $\nu_E$ the generalized outer unit normal.

\begin{definition}\label{def:qmin}
A set of finite perimeter $E\subset\R^N$ is an \emph{$(\omega,R)$-minimizer} for the anisotropic perimeter, with $\omega>0$, $R>0$,
if for every ball $B_r(x)$, $0<r<R$, and for every set of finite perimeter $F$
such that $E\bigtriangleup F\subset\subset B_r(x)$ we have
$$
\int_{\partial^*E}\psi(\nu_E)\,d\hn \leq \int_{\partial^*F}\psi(\nu_F)\,d\hn + \omega |E\bigtriangleup F|.
$$
\end{definition}

In this context, we say that a set $E$ is \emph{periodic} if its characteristic function is one-periodic in the first $N-1$ coordinate directions. The following theorem contains the main regularity property of uniform sequences of quasi-minimizers.

\begin{theorem} \label{teo:qminreg}
Let $E_n$ be a sequence of periodic $(\omega,R)$-minimizers of the anisotropic perimeter such that
$\sup_n\mathcal{H}^{N-1}( \partial^* E_n \cap ([0,1)^{N-1}\times\R) ) <\infty$ and
$\chi_{E_n}\to \chi_{E}$ in $L^1_{\rm loc}(\R^N)$, where $E\subset\R^N$ is a periodic set of class $C^2$.
Then, for $n$ sufficiently large $E_n$ is a set of class $C^{1,\frac12}$
and $\partial E_n\to\partial E$ in $C^{1,\alpha}$ for every $\alpha\in(0,\frac12)$, in the sense that
$$
\partial E_n = \{ z+\vphi_n(z)\nu_E(z) : z\in\partial E \}
$$
with $\vphi_n\to0$ in $C^{1,\alpha}(\partial E)$.
\end{theorem}

The previous result is a consequence of the standard regularity theory for \emph{almost-minimal currents} (see \cite{Alm,Bom,SSim}).
Precisely, it can be deduced from the result stated in \cite[Theorem~15]{FigMag}
by an argument which is well-known to specialists and can be found, for instance, in the proof of \cite[Theorem~8]{FigMag}
(see also \cite[Lemma~3.6]{CicLeo} for the isotropic case).
Notice that the quasi-minimality property considered in \cite{FigMag}, namely
$$
\int_{\partial^*E}\psi(\nu_E)\,d\hn \leq \int_{\partial^*F}\psi(\nu_F)\,d\hn + \omega r \p(E\bigtriangleup F)
$$
whenever $E\bigtriangleup F$ is compactly contained in a ball of radius $r$,
is clearly implied by our definition of quasi-minimality as a consequence of the isoperimetric inequality.

Another preliminary result that we will need in this section is the following lemma,
which can be proved by standard elliptic estimates.

\begin{lemma} \label{lem:ellipticest}
Let $h\in C^2_\#(Q)$, and let $h_n\in C^{1,\alpha}_\#(Q)$ be such that $h_n \to h$ in $C^{1,\alpha}$,
for some $\alpha\in(0,1)$. Assume also that the anisotropic mean curvature $H^\psi_{h_n}$ of $h_n$ is bounded.
Then
\begin{itemize}
  \item [(i)] if $H^{\psi}_{h_n}(\cdot,h_n(\cdot)) \to H^\psi(\cdot,h(\cdot))$ in $L^p(Q)$,
  then $h_n\to h$ in $W^{2,p}(Q)$;
  \item [(ii)] if $\sup_n \|H^{\psi}_{h_n}\|_{L^p(Q)}<\infty$,
  then $\sup_n\|h_n\|_{W^{2,p}(Q)}<\infty$.
\end{itemize}
\end{lemma}

\begin{proof}
The function $h_n$ is a weak solution to the equation
$$
- \int_Q \nabla\psi(-\nabla h_n,1) \cdot (\nabla\eta,0) \,dx = \int_Q H^\psi_{h_n}(x,h_n(x)) \eta(x) \,dx
\qquad\text{for all } \eta\in C^\infty_\#(Q)
$$
with $H^{\psi}_{h_n}(\cdot,h_n(\cdot))\in L^\infty(Q)$,
which implies, by elliptic regularity (see, \textit{e.g.}, \cite[Proposition~7.56]{AFP}), that $h_n\in W^{2,2}_\#(Q)$.
Hence it makes sense to perform the differentiation and rewrite the equation in non-divergence form:
$$
\sum_{i,j=1}^{N-1} \frac{\partial^2\psi}{\partial z_i\partial z_j}(-\nabla h_n(x),1) \frac{\partial^2 h_n}{\partial x_i \partial x_j}(x) = - H^\psi_{h_n}(x,h_n(x)) \qquad\text{a.e. in }Q.
$$
By elliptic regularity results for equations in non-divergence form with continuous coefficients,
we deduce that $h_n\in W^{2,p}_\#(Q)$ for every $p\in[1,\infty)$ (see \cite[Theorem~7.48]{AFP}),
and in turn the conclusion follows from \cite[Theorem~9.11]{GT}
recalling that $h_n\to h$ in $C^{1,\alpha}$.
\end{proof}

We recall that we associated, with a critical pair $(h,u)$, an open set $\Omega'$ containing $\Omega_h$
in terms of which we defined in \eqref{spaziocompet} the class of competitors $X'$.
Our strategy requires now the extension of the functional $F$ to a larger class of admissible pairs:
in particular, we shall consider not just subgraphs of Lipschitz functions, but generic periodic sets with locally finite perimeter.
More precisely, let $\widetilde{X}$ be the set of all pairs $(\Omega,v)$ such that:
\begin{itemize}
  \item $\Omega\subset\Omega'$ is a set of finite perimeter; we will denote by $\Omega^\#$ the periodic extension of $\Omega\cup(Q\times\R^-)$ in the first $N-1$ directions;
  \item $v\in W^{1,\infty}(\Omega'_\#;\R^N)$ is such that $v-u_0\in\mathcal{V}(\Omega')$ and $\det\nabla v>0$ a.e. in $\Omega$.
\end{itemize}
For $(\Omega,v)\in\widetilde{X}$ we define
$$
\widetilde{F}(\Omega,v) := \int_\Omega W(\nabla v)\,dz + \int_{\Gamma_\Omega} \psi(\nu_\Omega)\,d\hn
$$
where $\Gamma_\Omega := \partial^*\Omega^\# \cap \bigl([0,1)^{N-1}\times\R\bigr)$ and $\nu_\Omega$ is the generalized outer unit normal to the reduced boundary of $\Omega^\#$.
We remark that, if $(g,v)\in X'$, then $(\Omega_g,v)\in\widetilde{X}$
and $\widetilde{F}(\Omega_g,v)=F(g,v)$.

We are now ready to state and prove the main result of this section.

\begin{theorem} \label{teo:locmin2}
Let $p\in(1,\infty)$, and assume that a critical pair $(h,u)\in X$ is a $W^{2,p}$-local minimizer,
in the sense of Definition~\ref{def:locmindeb}.
Then $(h,u)$ is a local minimizer for $F$, according to Definition~\ref{def:locmin}.
\end{theorem}

\begin{proof}
We argue by contradiction, assuming the existence of a decreasing sequence $\sigma_n\to 0$
and of a sequence $(g_n,u_n)\in X'$ such that
$$0<\|g_n-h\|_{\infty}\leq\sigma_n, \quad
\|\nabla u_n - \nabla u\|_{L^\infty(\Omega';\M^{N})}\leq\sigma_n, \quad
|\Omega_{g_n}|=|\Omega_h|,
$$
and
\begin{equation} \label{ipotesiassurdo}
F(g_n,u_n) < F(h,u).
\end{equation}
We now split the proof into several steps.

\smallskip
\noindent
{\it Step 1.}
We claim that we can find new sequences $\delta_n\to0$ and $v_n\in C^{\infty}(\overline{\Omega}';\R^N)$ such that
$(g_n,v_n)\in X'$,
$\|g_n-h\|_{{\infty}}\leq\delta_n$,
$\|\nabla v_n - \nabla u\|_{L^\infty(\Omega';\M^{N})} \leq \delta_n$,
and for which we still have
\begin{equation} \label{ipotesiassurdo1}
F(g_n,v_n) < F(h,u).
\end{equation}
Indeed, for every $n$ we can construct an approximating sequence $v_n^k$, $k\in\N$,
in the following way: we let $\rho_{1/k}$ be the standard mollifier in $\R^N$
with support compactly contained in $B_{1/k}$, and we set
$$
v_n^k:=w_n^k\ast\rho_{1/k} + u_0,
\qquad\text{where}\quad
w_n^k(x,y):=
\left\{
  \begin{array}{ll}
    (u_n-u_0)(x,y-1/k) & \hbox{if }y\geq0, \\
    0 & \hbox{if }y<0
  \end{array}
\right.
$$
(where we extended $u_n-u_0$ to 0 in $\R^N_-$).
Then by the properties of the convolution product we have $v_n^k\in C^\infty(\overline{\Omega}';\R^N)$,
$v_n^k-u_0\in\mathcal{V}(\Omega')$, and
$$
\|\nabla v_n^k - \nabla u\|_{L^\infty(\Omega';\M^N)}\leq2\sigma_n
$$
for every $k$ sufficiently large.
Moreover, $F(g_n,v_n^k)\to F(g_n,u_n)$ as $k\to\infty$
by the Lebesgue Dominated Convergence Theorem.
Hence, for every $n$ we can find $k_n$ such that the function $v_n:=v_n^{k_n}$ satisfies the desired properties with $\delta_n=2\sigma_n$.
We set $M_n:= \| \nabla^2 v_n \|_{\infty}$.

\smallskip
\noindent
{\it Step 2.}
Let $(\Omega_n,w_n)\in\widetilde{X}$ be a solution to the penalized problem
\begin{align} \label{penalizzato1}
\min \Bigl\{ J_{\beta}(\Omega,v) : (\Omega,v)\in\widetilde{X}, \;
\Omega_{h-\delta_n}\subset\Omega\subset\Omega_{h+\delta_n}, \;
& v\in W^{2,\infty}(\Omega';\R^N), \nonumber\\
\|\nabla^2v\|_{\infty}\leq M_n, \;
& \|\nabla v - \nabla u\|_{L^\infty(\Omega';\M^{N})}\leq\delta_n \Bigr\},
\end{align}
where
\begin{equation*}
J_{\beta}(\Omega,v):= \widetilde{F}(\Omega,v) + \beta \big| |\Omega|-|\Omega_h| \big|
\end{equation*}
and $\beta$ is a positive constant, to be chosen later.
Observe that problem \eqref{penalizzato1} admits a solution by the direct method of the Calculus of Variations:
indeed, if $(\Omega^k,w^k)$ is a minimizing sequence, then up to subsequences we have that $\Omega^k\to\Omega^0$ in $L^1$
and $w^k\to w^0$ weakly* in $W^{2,\infty}(\Omega')$; the pair $(\Omega^0,w^0)$ satisfies all the constraints
and is a minimizer of \eqref{penalizzato1} by the lower semicontinuity of the functional
(which follows in particular from Reshetnyak's Lower Semicontinuity Theorem, as stated in \cite[Theorem~2.38]{AFP}, for the surface term).

Since $(\Omega_{g_n},v_n)$ is an admissible competitor for \eqref{penalizzato1},
the minimality of $(\Omega_n,w_n)$ and \eqref{ipotesiassurdo1} yield
\begin{equation} \label{contradiction0}
\widetilde{F}(\Omega_n,w_n)\leq J_{\beta}(\Omega_n,w_n) \leq J_{\beta}(\Omega_{g_n},v_n) = F(g_n,v_n) < F(h,u).
\end{equation}

\smallskip
\noindent
{\it Step 3.}
We claim that, for $\beta$ large enough (independently of $n$),
$(\Omega_n,w_n)$ is also a solution to the minimum problem
\begin{align} \label{penalizzato2}
\min \Bigl\{ \widetilde{J}_{\beta}(\Omega,v) :
(\Omega,v)\in\widetilde{X}, \;
v\in W^{2,\infty}(\Omega';\R^N), \;
\|\nabla^2v\|_{\infty}\leq M_n, \;
\|\nabla v - \nabla u\|_{L^\infty(\Omega';\M^{N})}\leq\delta_n \Bigr\},
\end{align}
where
\begin{equation*}
\widetilde{J}_{\beta} (\Omega,v) := J_{\beta}(\Omega,v) + 2 \beta \, |\Omega \bigtriangleup T_n(\Omega)|
\end{equation*}
and $T_n(\Omega):= \bigl( \Omega\cup\Omega_{h-\delta_n} \bigr) \cap \Omega_{h+\delta_n}$.

To prove the claim, consider any competitor $(\Omega,v)$ for problem \eqref{penalizzato2}.
Then we have, since $T_n(\Omega_n)=\Omega_n$,
\begin{align*}
\widetilde{J}_{\beta} (\Omega,v) &- \widetilde{J}_{\beta} (\Omega_n,w_n)
= J_{\beta}(T_n(\Omega),v)  - J_{\beta}(\Omega_n,w_n)
 + 2 \beta \, |\Omega\bigtriangleup T_n(\Omega)| \\
& + \int_{\Omega} W(\nabla v)\,dz - \int_{T_n(\Omega)} W(\nabla v)\,dz
 + \int_{\Gamma_\Omega}\psi(\nu_\Omega)\,d\hn - \int_{\Gamma_{T_n(\Omega)}} \psi(\nu_{T_n(\Omega)})\,d\hn \\
& + \beta \Bigl( \bigl| |\Omega|-|\Omega_h| \bigr| - \bigl| |T_n(\Omega)|-|\Omega_h| \bigr| \Bigr) \\
&\geq \left(2 \beta - W_0 - \beta \right) |\Omega\bigtriangleup T_n(\Omega)|
 + \int_{\Gamma_\Omega}\psi(\nu_\Omega)\,d\hn - \int_{\Gamma_{T_n(\Omega)}} \psi(\nu_{T_n(\Omega)})\,d\hn,
\end{align*}
where in the last inequality we used the fact that $J_{\beta}(T_n(\Omega),v) - J_{\beta}(\Omega_n,w_n)\geq0$
by the minimality of $(\Omega_n,w_n)$,
and $W_0$ is a positive constant depending only on $W$ and $u$.

Now recalling the 1-homogeneity of $\psi$, the Euler's theorem $\psi(\nu)=\nabla\psi(\nu)\cdot\nu$
and the convexity of $\psi$ yield
$$
\psi(\nu_\Omega) \geq \psi(\nu_h) + \nabla\psi(\nu_h)\cdot(\nu_\Omega-\nu_h) = \nabla\psi(\nu_h)\cdot\nu_\Omega
\quad\text{on }\Gamma_{\Omega},
$$
where, for every $z\in\R^N$, we denote by $\nu_h(z)$ the upper unit normal to the graph of $h$ at the point $(\pi(z),h(\pi(z)))$.
Hence, using again Euler's theorem and observing that
$\mathcal{H}^{N-1}$-almost everywhere on $\Gamma_{T_n(\Omega)}\setminus\Gamma_\Omega$
the normal to $\Gamma_{T_n(\Omega)}$ coincides with $\nu_h$ , we obtain
\begin{align} \label{lemmino}
\int_{\Gamma_\Omega} \psi(\nu_{\Omega}) &\,d\hn - \int_{\Gamma_{T_n(\Omega)}} \psi(\nu_{T_n(\Omega)}) \,d\hn \nonumber\\
&\geq \int_{\Gamma_\Omega\setminus\Gamma_{T_n(\Omega)}} \nabla\psi(\nu_h)\cdot\nu_\Omega\,d\hn
 - \int_{\Gamma_{T_n(\Omega)}\setminus\Gamma_{\Omega}} \nabla\psi(\nu_h)\cdot\nu_h\,d\hn \nonumber \\
&\geq - \int_{\Omega\bigtriangleup T_n(\Omega)} \big| \div(\nabla\psi\circ\nu_h) \big| \,dz
\geq -\Lambda_0 |\Omega\bigtriangleup T_n(\Omega)| \,.
\end{align}
Here $\Lambda_0:=\|H^\psi\|_{L^\infty(\Gamma_h)}$, where $H^\psi$ denotes the anisotropic mean curvature of $\Gamma_h$.
Hence we can conclude
$$
\widetilde{J}_{\beta} (\Omega,v) - \widetilde{J}_{\beta} (\Omega_n,w_n)
\geq \left(\beta - W_0 - \Lambda_0 \right) |\Omega\bigtriangleup T_n(\Omega)|,
$$
so that by choosing $\beta>W_0+\Lambda_0$
(notice that this constant depends only on $W$, $\psi$, $h$ and $u$)
we deduce that $(\Omega_n,w_n)$ is a solution to \eqref{penalizzato2}.

\smallskip
\noindent
{\it Step 4.}
We claim that each $\Omega_n$ satisfies the volume constraint
\begin{equation} \label{volume}
|\Omega_{n}| = |\Omega_h|.
\end{equation}
Suppose by contradiction that $|\Omega_h|-|\Omega_{n}|=:d>0$ for some $n$.
We can find $\delta\in(-\delta_n,\delta_n)$ such that $|\Omega_{n}\cup\Omega_{h+\delta}|=|\Omega_h|$.
Define $U := \Omega_n\cup\Omega_{h+\delta}$.
Then, as $|U|=|\Omega_h|$, we have
\begin{align} \label{contradiction1}
J_{\beta}(U,w_n) &- J_{\beta}(\Omega_n,w_n)
= \int_{U} W(\nabla w_n)\,dz - \int_{\Omega_n} W(\nabla w_n)\,dz \nonumber\\
& + \int_{\Gamma_{U}} \psi(\nu_{U})\,d\hn - \int_{\Gamma_{\Omega_n}} \psi(\nu_{\Omega_n})\,d\hn -\beta d \nonumber\\
& \leq (W_0-\beta) \, d + \int_{\Gamma_{U}} \psi(\nu_{U})\,d\hn - \int_{\Gamma_{\Omega_n}} \psi(\nu_{\Omega_n})\,d\hn
\end{align}
where $W_0$ is the same constant as in Step 3.
Now, arguing as in \eqref{lemmino}, we have
$$
\int_{\Gamma_{U}} \psi(\nu_{U})\,d\hn - \int_{\Gamma_{\Omega_n}} \psi(\nu_{\Omega_n})\,d\hn
\leq \Lambda_0 d.
$$
Hence \eqref{contradiction1} implies that
$$
J_{\beta}(U,w_n) - J_{\beta}(\Omega_n,w_n)
\leq \left( W_0+\Lambda_0-\beta \right) d <0
$$
(recall that $\beta>W_0+\Lambda_0$), which is a contradiction with the minimality of $(\Omega_n,w_n)$.

In the case $|\Omega_{n}|>|\Omega_h|$,
we can find $\delta\in(-\delta_n,\delta_n)$ such that $|\Omega_{n}\cap\Omega_{h+\delta}|=|\Omega_h|$.
Then, setting $U := \Omega_n\cap\Omega_{h+\delta}$ and arguing as before,
we still contradict the minimality of $(\Omega_n,w_n)$.

\smallskip
\noindent
{\it Step 5.}
We claim that $\Omega_n^\#$ is an $(\omega,R)$-minimizer for the anisotropic perimeter
(see Definition~\ref{def:qmin}),
with $\omega$ and $R$ independent of $n$.
Indeed, consider any ball $B_r(x)$ and any set $F$ such that $\Omega_n^\#\bigtriangleup F\subset\subset B_r(x)$.
By a translation argument we can assume $B_r(x)\subset Q\times\R$;
moreover, by taking a sufficiently small $R$ we can also assume without loss of generality that $B_r(x)\subset\Omega'$.
Hence, setting $F':=F\cap\Omega'$, we have that $(F',w_n)\in\widetilde{X}$ is an admissible competitor in problem \eqref{penalizzato2}.
By the minimality of $(\Omega_n,w_n)$,
we have $\widetilde{J}_\beta(F',w_n)- \widetilde{J}_\beta(\Omega_n,w_n)\geq0$, which yields
\begin{align*}
\int_{\partial^*\Omega_n\cap B_r(x)}&\psi(\nu_{\Omega_n})\,d\hn \\
&\leq \int_{\partial^*F\cap B_r(x)}\psi(\nu_F)\,d\hn
+\int_{F'}W(\nabla w_n)\,dz - \int_{\Omega_n} W(\nabla w_n)\,dz \\
&\qquad + \beta \big| |F'|-|\Omega_n| \big| + 2\beta |F'\bigtriangleup T_n(F')| \\
&\leq \int_{\partial^*F\cap B_r(x)}\psi(\nu_F)\,d\hn + (W_0+3\beta)|F'\bigtriangleup \Omega_n|,
\end{align*}
where we used the fact that $F' \bigtriangleup T_n(F') \subset F' \bigtriangleup \Omega_n$.
Since $|F'\bigtriangleup \Omega_n|=|F\bigtriangleup\Omega_n^\#|$,
the previous inequality proves the claim with $\omega=W_0+3\beta$.

Hence, by the regularity of quasi-minimizers (see Theorem~\ref{teo:qminreg}), we deduce that
$\Omega_n$ is a set of class $C^{1,\frac12}$ for $n$ large enough,
and that it converges to $\Omega_h$ in $C^{1,\alpha}$ for all $\alpha\in(0,\frac12)$.
In turn, this implies that for $n$ large the set $\Omega_n$
is in fact the subgraph of a function $k_n\in C^{1,\frac12}_\#(Q)$
(that is, $\Omega_n=\Omega_{k_n}$),
and $k_n\to h$ in $C^{1,\alpha}$ for all $\alpha\in(0,\frac12)$.

\smallskip
\noindent
{\it Step 6.}
We claim that $k_n\to h$ in $W^{2,p}$ for every $p\in(1,\infty)$.

Fix $\eta\in C_\#^\infty(Q)$ and set $k_n^\e:=k_n+\e\eta$, for $\e>0$.
By the quasi-minimality property of $\Gamma_{k_n}$ proved in the previous step we have
$$
\int_{\Gamma_{k_n}} \psi(\nu_{k_n})\,d\hn
\leq \int_{\Gamma_{k_n^{\e}}} \psi(\nu_{k_n^{\e}}) \,d\hn
+ \e \int_Q |\eta(x)| \,dx.
$$
Dividing by $\e$ and letting $\e\to0$, we deduce
$$
\int_Q \nabla\psi(-\nabla k_n,1)\cdot(\nabla\eta,0)\,dx \leq \|\eta\|_{L^1(Q)}\,.
$$
Hence, the left-hand side in the previous inequality defines a continuous linear functional on $L_\#^1(Q)$,
that is, denoting by $H^\psi_{k_n}$ the anisotropic mean curvature of $\Gamma_{k_n}$ and recalling \eqref{Hpsi},
$$
-H^\psi_{k_n} (\cdot,k_n(\cdot)) = H_n \qquad\text{on }Q
$$
in the sense of distributions, for some bounded function $H_n$ whose $L^\infty$-norm is bounded by 1.
This uniform bound, combined with the convergence of the functions $k_n$ to $h$ in $C^{1,\alpha}$,
implies by standard elliptic estimates (see Lemma~\ref{lem:ellipticest}) that the functions $k_n$ are equibounded in $W^{2,p}$ for every $p>1$.

We can now write the Euler-Lagrange equations for problem \eqref{penalizzato1}:
since $k_n$ is of class $W^{2,p}$ we have
$$
H^\psi_{k_n}(x,k_n(x)) =
\begin{cases}
-W \bigl( \nabla w_n(x,k_n(x)) \bigr) + \lambda_n & \text{ in } A_n:= \bigl\{ |k_n-h|<\delta_n \bigr\}, \\
-W \bigl( \nabla u(x,h(x) \bigr) + \lambda & \text{ in } \bigl\{ |k_n-h|=\delta_n \bigr\},
\end{cases}
$$
where $\lambda_n$, $\lambda$ are the Lagrange multipliers due to the volume constraint.
To deduce the equation in $A_n$ we considered variations only of the profile $k_n$, compactly supported in $A_n$,
while the equation in the complement of $A_n$ easily follows from the fact that $(h,u)$
satisfies \eqref{critpair}.
Notice that the sequence $\lambda_n$ is bounded, by the uniform bounds on  $H^\psi_{k_n}$ and on $\nabla w_n$.

Now, if $\hn(A_n)\to0$, we immediately have
\begin{equation} \label{convcurvature}
H^\psi_{k_n}(\cdot,k_n(\cdot)) \to H^\psi (\cdot,h(\cdot)) \qquad\text{in } L^{p}(Q) \text{ for all }p>1.
\end{equation}
Otherwise, assuming that $\hn(A_n)\geq c>0$ for all $n$,
integrating the Euler-Lagrange equation in $Q$ we deduce by periodicity that
\begin{align*}
-\int_{A_n} &W(\nabla w_n(x,k_n(x)))\,dx + \lambda_n\hn(A_n)
- \int_{Q\setminus A_n} W(\nabla u(x,h(x)))\,dx + \lambda\hn(Q\setminus A_n) \\
&= \int_Q H^\psi_{k_n}(x,k_n(x))\,dx
= 0 = \int_Q H^\psi(x,h(x))\,dx
= - \int_{Q} W(\nabla u(x,h(x)))\,dx + \lambda\hn(Q).
\end{align*}
Now the uniform convergence of $\nabla w_n$ to $\nabla u$ on $\Gamma_{k_n}$ and the convergence of $k_n$ to $h$ in $C^{1,\alpha}$
yield $(\lambda_n-\lambda)\hn(A_n)\to0$,
and in turn $\lambda_n\to\lambda$ since $\hn(A_n)\geq c>0$.
Hence, using again the Euler-Lagrange equations, we can conclude that \eqref{convcurvature} holds.
In turn, by elliptic regularity (Lemma~\ref{lem:ellipticest}) this implies that $k_n\to h$ in $W^{2,p}$
for every $p>1$, as claimed.

\smallskip
\noindent
{\it Step 7.}
We are now in position to conclude the proof of the theorem.
Since
$$
\|k_n - h\|_{W^{2,p}(Q)}\to0,
\quad
\| \nabla w_n - \nabla u \|_{L^\infty(\Omega_{k_n};\M^N)} \to0,
$$
and, by Step 4, $|\Omega_{k_n}|=|\Omega_h|$,
inequality \eqref{contradiction0} is in contradiction
with the $W^{2,p}$-local minimality of $(h,u)$.
\end{proof}

Combining the previous result with Theorem~\ref{teo:locmin},
we immediately obtain the announced local minimality condition.

\begin{theorem} \label{cor:locmin}
Assume $N=2,3$.
If $(h,u)\in X$ is a strictly stable critical pair, according to Definition~\ref{def:stable2},
then $(h,u)$ is a local minimizer for the functional $F$, in the sense of Definition~\ref{def:locmin}.
\end{theorem}

We conclude this section by observing that Theorem~\ref{cor:locmin} can be extended to the linear elastic case,
where we have the following stronger result.
Given a set $A$ and a constant $M>0$, we denote by ${\rm Lip}_M(A;\R^N)$ the class of Lipschitz functions $v:A\to\R^N$
whose Lipschitz constant is bounded by $M$.

\begin{theorem} \label{teo:linelas}
Assume that the elastic energy density has the form
$$
W(\xi):=\frac12 \, C \, \biggl(\frac{\xi+\xi^T}{2}\biggr) : \biggl(\frac{\xi+\xi^T}{2}\biggr)\,,
\qquad
\xi\in\M^N\,,
$$
for some constant fourth-order tensor $C$ such that
\begin{equation} \label{elastpos}
C\xi:\xi\geq c_0|\xi|^2
\quad\text{for every }\xi\in \M^N_{\rm sym},
\quad c_0>0,
\end{equation}
where $\M^N_{\rm sym}$ denotes the subset of $\M^N$ of the symmetric matrices.
If $N=2,3$ and $(h,u)$ is a strictly stable critical pair, then
$(h,u)$ is an isolated local minimizer for $F$ in the following sense:
for every $M>\|\nabla u\|_\infty$ there exists $\delta=\delta(M)>0$ such that
\begin{equation} \label{locminel}
F(h,u)<F(g,v)
\end{equation}
for every $(g,v)\in X$ with $0<\|g-h\|_\infty<\delta$, $|\Omega_g|=|\Omega_h|$,
and $v\in{\rm Lip}_M(\Omega_g;\R^N)$.

\end{theorem}

\begin{remark}
Notice that, by Korn's inequality, the positive definiteness of the tensor $C$ on the space of symmetric matrices implies that condition \eqref{c0} is automatically satisfied.
We suspect that, as in the two-dimensional case (see \cite{FM}), in the linearized framework the following stronger result should hold: there exists $\delta>0$ such that \eqref{locminel} is satisfied for every $(g,v)\in X$ with $0<\|g-h\|_\infty<\delta$, $|\Omega_g|=|\Omega_h|$, and $v\in{\rm Lip}(\Omega_g;\R^N)$.
In order to prove such a result, we would need a regularity theory for minimizing configurations, which is not yet available in the three-dimensional case.
\end{remark}

\begin{proof}[Proof of Theorem~\ref{teo:linelas}]
We first observe that the conclusion of Theorem~\ref{teo:locmin} holds also in this case.
Indeed, the construction provided by Proposition~\ref{IFT} is now unnecessary, since for every admissible profile $g$ we can consider the unique minimizer $u_g$ of the elastic energy in the corresponding reference configuration $\Omega_g$.
By standard elliptic regularity, the map $g\mapsto u_g$ satisfies the conclusions of Proposition~\ref{IFT},
so that we can repeat the proof of Theorem~\ref{teo:locmin} without changes.
Notice also that the estimate provided by Lemma~\ref{lem:algebra} remains valid in this case, since the fourth order tensor $W_{\xi\xi}$ satisfies the strong ellipticity condition, as a consequence of \eqref{elastpos}.

At this point we can follow the strategy of the proof of Theorem~\ref{teo:locmin2}, where the contradiction hypothesis consists now in assuming the existence of a sequence $(g_n,v_n)\in X$ such that $\delta_n:=\|g_n-v_n\|_\infty\to0$, $|\Omega_{g_n}|=|\Omega_h|$,
$v_n\in{\rm Lip}_M(\Omega_{g_n})$, and $F(g_n,v_n)\leq F(h,u)$

The approximation argument contained in Step~1 of the previous proof is in this case unnecessary,
so that we do not need the strict inequality in \eqref{ipotesiassurdo1}.
Indeed, each function $v_n$ can be extended to $\Omega'$ without increasing the Lipschitz constant,
and we can now consider the penalized minimum problems
\begin{align} \label{penalizzato3}
\min \Bigl\{ J_{\beta}(\Omega,v) : (\Omega,v)\in\widetilde{X}, \;
\Omega_{h-\delta_n}\subset\Omega\subset\Omega_{h+\delta_n}, \;
v\in {\rm Lip}_M(\Omega';\R^N) \Bigr\}
\end{align}
which admits a solution without assuming any \emph{a priori} $W^{2,\infty}$-bound, as we did before.
Replacing \eqref{penalizzato1} by \eqref{penalizzato3}, the proof goes exactly as in the previous case, yielding the $C^{1,\alpha}$-convergence of $k_n$ to $h$ at the end of the fifth step; moreover, $k_n\in W^{2,p}(Q)$, as proved in the first part of Step~6.

Observe now that, denoting by $\tilde{w}_n$ the unique minimizer of the (linear) elastic energy in $\Omega_{k_n}$, by the standard regularity of the elliptic system associated with the first variation of the elastic energy we have that $\nabla\tilde{w}_n\circ\Phi_{k_n}$ converge uniformly to $\nabla u$ in $\overline{\Omega}_h$, so that for $n$ sufficiently large the constraint $\tilde{w}_n\in{\rm Lip}_M(\Omega')$ is satisfied.
Hence we necessarily have $w_n=\tilde{w}_n$: thus $w_n$ is in fact of class $C^{1}$ up to $\Gamma_{k_n}$, and we can conclude as before, by writing the Euler-Lagrange equations for the penalized problems, that $k_n\to h$ in $W^{2,p}(Q)$.

Finally, in the last step of the proof we deduce, by the isolated local minimality of $(h,u)$ proved in Theorem~\ref{teo:locmin},
that $k_n=h$ and $w_n=u$ for all sufficiently large $n$. It follows that $(h,u)$ and, in turn, $(g_n,v_n)$ are solutions to the penalized minimum problem: repeating the same argument for the sequence $(g_n,v_n)$, we conclude that for $n$ sufficiently large $g_n=h$ and $v_n=u$, which is the final contradiction.
\end{proof}

\end{section}


\begin{section}{Stability of the flat configuration} \label{sect:flat}

In this section, as an application of our local minimality criterion, we deal with the issue of the stability of the flat configuration. Given a volume $d>0$, we will assume the existence of an affine critical point for the elastic energy in the domain $\Omega_{d}=Q\times(0,d)$, namely (recall Definition~\ref{def:critpoint}) an affine function $v_0(z)=M[z]$ for some $M\in\M^{N}_+$ solution to the problem
\begin{equation} \label{flatconfig0}
\begin{cases}
\div (W_\xi(\nabla v_0)) = 0 & \text{ in } \Omega_{d}, \\
W_\xi(\nabla v_0)[e_N] = 0 & \text{ on } \Gamma_{d}, \\
v_0-u_0 \in \mathcal{V}(\Omega_{d}),
\end{cases}
\end{equation}
where $u_0 (x,y)=(A [x],0)$ is the boundary Dirichlet datum.
Notice that an affine function automatically satisfies the first condition (as $\nabla v_0$ is constant),
but this is not always the case for the second one, that can be rewritten as
\begin{equation} \label{flatconfig}
\frac{\partial W}{\partial \xi_{iN}}(\nabla v_0)=0 \qquad \text{for every } \, i=1,\ldots,N.
\end{equation}

\begin{definition} \label{def:flatconf}
A pair $(d,v_0)\in X$, with $v_0(z)=M[z]$,
satisfying \eqref{flatconfig0} and condition \eqref{c0}
will be referred to as \textit{flat configuration} with volume $d$.
\end{definition}

We remark that, whenever it exists, $(d,v_0)$ is obviously a critical pair for the functional $F$.

\begin{example}
We now show the existence of an affine critical point for the elastic energy in a flat domain,
for boundary data close to the identity,
under the assumption that the identical deformation is a strict local minimum of the elastic energy.
More precisely, we assume that $W(I)=0$ and that
\begin{equation} \label{esempio:defpos}
\int_{\Omega_d} W_{\xi\xi}(I) \nabla w : \nabla w \,dz \geq k\|w\|_{H^1(\Omega_d;\R^N)}^2
\qquad\text{for every }w\in\Vtilde(\Omega_d),
\end{equation}
for some $k>0$.
Notice that, as $W\geq0$ and $W(I)=0$, necessarily $W_\xi(I)=0$.
We claim that, if $|A-I|<\e_0$ for some $\e_0>0$ sufficiently small,
then there exists an affine solution to \eqref{flatconfig0} corresponding to the boundary datum $u_0(x,y)=(A[x],0)$.

Indeed, given $A\in\M^{N-1}$, we look for a vector $\mathbf{b}=(b_1,\ldots,b_N)$ such that the affine function
$$v_{A,\mathbf{b}}(x,y)=(A[x],0)+y\mathbf{b}$$
satisfies \eqref{flatconfig}.
We define a map $G:(A,\mathbf{b})\mapsto W_\xi(\nabla v_{A,\mathbf{b}})[e_N]$.
As $W_\xi(I)=0$, we have that $G(I,e_N)=0$.
Moreover the matrix $\partial_\mathbf{b}G(I,e_N)$ is positive definite (hence invertible),
since for every vector $w\in\R^N\setminus\{0\}$
\begin{align*}
\partial_\mathbf{b}G(I,e_N)[w,w] = \sum_{i,j=1}^N \frac{\partial^2W}{\partial\xi_{iN}\partial\xi_{jN}}(I)w_iw_j
= W_{\xi\xi}(I) (w\otimes e_N) : (w\otimes e_N) >0\,,
\end{align*}
where the last inequality follows from the fact that the tensor $W_{\xi\xi}(I)$ satisfies the strong ellipticity condition
(by Theorem~\ref{teo:simpsonspector} and \eqref{esempio:defpos}).
Hence the claim follows by applying the Implicit Function Theorem
(notice also that the affine critical point constructed in this way satisfies condition \eqref{c0}, up to taking a smaller $\e_0$ if necessary, by continuity and by \eqref{esempio:defpos}).
\end{example}

When dealing with the flat configuration $(d,v_0)$, it is convenient to identify the space $\Htilde(\Gamma_{d})$ with the space
\begin{align*}
{\widetilde H}^1_\#(Q) &:=\Bigl\{\vphi \in H_{loc}^1(\R^{N-1}) :\, \vphi(x + e_i)=\vphi(x) \text{ for a.e. }x \in \R^{N-1}, \\
&\hspace{1cm} \text{ for every }i=1,\ldots,N-1,\, \int_Q\vphi(x)\,dx=0\Bigr\}\,.
\end{align*}
Notice that condition \eqref{prodscalpos} is always fulfilled (the coefficient $a$ in \eqref{prodottoscalare} vanishes), so that
$$
\| \vphi \|^2_\sim = \int_Q \nabla^2\psi(e_N) [(\nabla\vphi,0), (\nabla\vphi,0)] \,dx \qquad \text{for every }\vphi\in\Htilde(Q)
$$
is an equivalent norm on $\Htilde(Q)$;
in particular, this allows us to discuss the positivity of the second variation at the flat configuration
in terms of the quantity $\lambda_1(d)$ defined by \eqref{lambda1}
(here we make explicit the dependence on the height $d$ of the reference configuration).

We now prove a couple of propositions concerning the stability of the flat configuration.
Precisely, we show that the flat configuration, whenever it exists, is strictly stable if the volume is sufficiently small, while condition \eqref{defpos} is not satisfied if the domain is large enough. In the following, we will always assume to deal with elastic energy densities $W$ which admit a flat configuration.

\begin{proposition}
There exists $d_0>0$ such that for every $d<d_0$
$$
\partial^2F(d,v_0)[\vphi] > 0 \qquad \text{for every }\vphi\in\Htilde(Q)\setmeno\{0\}.
$$
\end{proposition}

\begin{proof}
Denote by $\mu_1(d)$ the value of the minimum in \eqref{mu1} corresponding to the critical pair $(d,v_0)$;
by Theorem~\ref{teo:condequiv} it is sufficient to show that
$$
\lim_{d\to0^+}\mu_1(d)=+\infty.
$$
Assume by contradiction that there exist $C>0$,
a sequence $d_n\to0^+$ and a sequence $v_n\in\Vtilde(\Omega_{d_n})$
such that $\| \Phi_{v_n} \|_\sim=1$ and
$$
\int_{\Omega_{d_n}} W_{\xi\xi}(\nabla v_0)\nabla v_n : \nabla v_n \,dz \leq C.
$$
Then the functions
$$
\tilde{v}_n(x,y) :=
\begin{cases}
0 & \text{if } 0\leq y \leq 1-d_n\\
v_n(x,y-1+d_n) & \text{if } 1-d_n < y \leq 1
\end{cases}
$$
belong to $\Vtilde(\Omega_1)$,
$\| \Phi_{\tilde{v}_n} \|_\sim = \| \Phi_{v_n} \|_\sim=1$
and satisfy
$$
\int_{\Omega_{1}} W_{\xi\xi}(\nabla v_0)\nabla \tilde{v}_n : \nabla \tilde{v}_n \,dz \leq C.
$$
It follows that, up to subsequences, $\tilde{v}_n$ converges weakly to 0 in $\Vtilde(\Omega_1)$.
From the compactness of the map $v\mapsto\Phi_v$
we conclude that $\Phi_{\tilde{v}_n}\to0$ strongly in $\Htilde(Q)$,
a contradiction with the fact that $\| \Phi_{\tilde{v}_n} \|_\sim =1$.
\end{proof}

In order to show a situation where the flat configuration is no longer a local minimizer,
we slightly modify the setting of the problem defining, for $d>0$, $Q_d=(0,d)^{N-1}$ and $\Omega_d=(0,d)^N$;
all the notions considered up to now are extended to this situation in the natural way.

\begin{proposition}
There exists $d_1>0$ such that the quadratic form $\partial^2F(d,v_0)$ is not positive semidefinite for all $d>d_1$.
In particular, for all $d>d_1$ the flat configuration $(d,v_0)$ is not a local minimizer for $F$.
\end{proposition}

\begin{proof}
Consider a nontrivial solution $(v,\vphi)\in\Vtilde(\Omega_1)\times\Htilde(Q)$ of \eqref{autovalori} in $\Omega_1$ with $\lambda=\lambda_1(1)$.
Setting $v_d(z)=v(\frac{z}{d})$, $\vphi_d(x)=d\,\vphi(\frac{x}{d})$,
a direct computation shows that $(v_d,\vphi_d)$ is a nontrivial solution of \eqref{autovalori} in $\Omega_d$
corresponding to $\lambda=d\,\lambda_1(1)$.
Hence $\lambda_1(d)\geq d\,\lambda_1(1)$, and taking $d_1=\frac{1}{\lambda_1(1)}$ we get that
$\lambda_1(d)>1$ for every $d>d_1$.
From this it is easily seen, using \eqref{fquadT}, that the quadratic form $\partial^2F(d,v_0)$ is not positive semidefinite for all $d>d_1$.
The last part of the statement follows from Theorem~\ref{teo:nec}.
\end{proof}

We conclude this section by discussing what happens in the case of \emph{crystalline} anisotropies,
namely if we assume less regularity in the anisotropic surface density
(we refer also to \cite{Bon}, where the two-dimensional case, in the framework of linearized elasticity, is studied in details).
Precisely, we assume here that $\psi_c: \R^N \to [0,+\infty)$ is a Lipschitz, positively 1-homogeneous and convex function,
such that the associated Wulff shape $W_{\psi_c}$ contains a neighborhood of the origin
and its boundary has a flat horizontal facet intersecting the $y$-axis.
We recall (see, \textit{e.g.}, \cite{Fon}) that the Wulff shape associated with a convex function $\psi:S^{N-1}\to(0,+\infty)$
is the convex set $W_\psi:=\{z\in\R^N: z\cdot\nu\leq\psi(\nu) \text{ for every }\nu\in S^{N-1}\}$.

Under these assumptions, we can show that the flat configuration is always a local minimizer for the associated functional $F_c$, whatever the volume $d>0$.

\begin{theorem} \label{teo:anis}
Let $N=2,3$, and let $\psi_c: \R^N \to [0,+\infty)$ be a Lipschitz, positively 1-homogeneous and convex function,
such that $\{|x|\leq a, \, y=b\}\subset\partial W_{\psi_c}$ for some $a,b>0$.
Then for every $d>0$ the flat configuration $(d,v_0)$ is a local minimizer for the associated functional $F_c$,
in the sense of Definition~\ref{def:locmin}.
\end{theorem}

\begin{proof}
Since we always evaluate the function $\psi_c$ at vectors whose last component is nonnegative,
without loss of generality we can assume that the Wulff shape $W_{\psi_c}$ is symmetric with respect to the hyperplane $\{y=0\}$.

From the assumptions on $\psi_c$ it follows that the cylinder $C=\{(x,y): |x|\leq a, |y| \leq b\}$
is contained in $W_{\psi_c}$.
Let $\psi_C(\nu_1,\nu_2) = a|\nu_1| + b|\nu_2|$ be an anisotropy whose Wulff shape is exactly the cylinder $C$.
Observe that
\begin{equation} \label{wulff1}
\psi_C\leq\psi_c, \qquad
\psi_c(0,1)=\psi_C(0,1)=b
\end{equation}
(the first follows from \cite[Proposition~3.5 (iii)]{Fon} and the inclusion $C\subset W_{\psi_c}$,
while the second is a consequence of \cite[Proposition~3.5 (iv)]{Fon}).

We now introduce a family of ``approximating'' functionals:
consider, for $\e>0$, the function $\psi_\e(x,y)=a\sqrt{\e^2y^2+|x|^2} + (b-a\e)|y|$,
and the associated functional $F_\e$.
Note that $\psi_\e$ converges monotonically from below to $\psi_C$ as $\e\to0^+$;
geometrically, the Wulff shapes associated with the functions $\psi_\e$
converge monotonically from the interior to the cylinder $C$.

Consider first the regular functions $\hat{\psi}_\e(x,y)=a\sqrt{\e^2y^2+|x|^2}$
and the associated functionals $\widehat{F}_\e$:
they satisfy all the assumptions of Section~\ref{sect:settings}
(in particular, the uniform convexity condition \eqref{pospsi} follows from the explicit computation of the hessian of $\hat{\psi}_\e$),
and the quadratic form associated to the second variation of $\widehat{F}_\e$ at the flat configuration
turns out to be
$$
\partial^2\widehat{F}_\e(d,v_0)[\vphi]=
-\int_{Q\times(0,d)} W_{\xi\xi}(\nabla v_0) \nabla v_{\vphi} : \nabla v_{\vphi} \, dz + \frac a\e \int_Q |\nabla\vphi|^2\,d\hn.
$$
Since
$$
\int_{Q\times(0,d)} W_{\xi\xi}(\nabla v_0) \nabla v_{\vphi} : \nabla v_{\vphi} \, dz
\leq C \|v_{\vphi}\|^2_{H^1(\Omega_d;\R^2)} \leq C' \|\vphi\|^2_{H^1(Q)}
$$
(where $C,C'$ are positive constants depending only on the boundary Dirichlet datum),
it follows that there exists $\e_0>0$ such that the quadratic form $\partial^2\widehat{F}_{\e_0}(d,v_0)$ is positive definite.
Hence, by Theorem~\ref{cor:locmin}, the flat configuration $(d,v_0)$ is a local minimizer for $\widehat{F}_{\e_0}$ for every volume $d>0$.
The same is true also for $F_{\e_0}$, since the energies $F_{\e_0}$ and $\widehat{F}_{\e_0}$ differ only by a constant value:
$F_{\e_0}=\widehat{F}_{\e_0}+(b-a\e_0)$.

We can now conclude the proof: let $\delta>0$ be such that the flat configuration minimizes the energy $F_{\e_0}$
among all competitors $(g,v)\in X'$ such that $|\Omega_g|=d$,
$0<\|g-d\|_\infty<\delta$,
and $\|\nabla v - \nabla v_0\|_{L^{\infty}(\Omega';\M^N)}<\delta$.
Then for every such $(g,v)$ we have
\begin{align*}
F_c (d,v_0)
&= \int_{Q\times(0,d)} W(\nabla v_0) \,dz + \psi_c(0,1)
 = \int_{Q\times(0,d)} W(\nabla v_0) \,dz + \psi_C(0,1) \\
&= F_C (d,v_0)
 = F_{\e_0}(d,v_0)
 \leq F_{\e_0}(g,v)
 \leq F_C(g,v)
 \leq F_c(g,v),
\end{align*}
where the first inequality follows from the local minimality of the flat configuration for $F_{\e_0}$,
the second one from $\psi_\e\leq\psi_C$
and the last one using $\psi_C\leq\psi_c$.
From the previous chain of inequalities the conclusion follows.
\end{proof}

\begin{remark}
If $W$ is as in Theorem~\ref{teo:linelas} and under the assumptions of Theorem~\ref{teo:anis},
we conclude that for every $d>0$ the flat configuration satisfies the isolated local minimality property
stated in Theorem~\ref{teo:linelas}.
\end{remark}

\end{section}


\begin{section}{Appendix}

\subsection{Fractional Sobolev spaces}\label{sect:appendix}

We collect in this section some auxiliary results concerning fractional Sobolev spaces
which are needed in Section~\ref{sect:locmin}.
The statements are the same as in \cite[Section~8.1]{FM},
rephrased to consider also the case of dimension $N=3$.

Fix a periodic function $h\in C^1_\#(Q)$, $h>0$.
We denote by $c_0$ a positive constant such that $\min_{\overline{Q}} h\geq c_0$.
We recall that the Gagliardo seminorm of a function $\vartheta$ on $\Gamma_h$ is defined as
$$
[\vartheta]_{s,p,\Gamma_h} := \Bigl(
\int_{\Gamma_h}\int_{\Gamma_h} \frac{|\vartheta(z)-\vartheta(w)|^p}{|z-w|^{N-1+sp}}
\,d\hn(z)\,d\hn(w) \Bigr)^{\frac1p}
$$
for $0<s<1$ and $1<p<\infty$, and that $\vartheta\in W^{s,p}(\Gamma_h)$ if
$$
\|\vartheta\|_{W^{s,p}(\Gamma_h)}:=\|\vartheta\|_{L^p(\Gamma_h)} + [\vartheta]_{s,p,\Gamma_h}<\infty.
$$
We denote by $W^{s,p}_\#(\Gamma_h)$ the subspace of functions $\vartheta\in W^{s,p}(\Gamma_h)$
whose periodic extension to $\Gamma_h^\#$ belongs to $W^{s,p}_{\rm loc}(\Gamma_h^\#)$, endowed with the same norm.
The dual spaces of $W^{s,p}(\Gamma_h)$ and of $W^{s,p}_\#(\Gamma_h)$
are denoted by $W^{-s,\frac{p}{p-1}}(\Gamma_h)$ and $W^{-s,\frac{p}{p-1}}_\#({\Gamma_h})$, respectively.
When $p=2$ we switch to the notation $H^s(\Gamma_h)$ for $W^{s,2}(\Gamma_h)$
(and similarly for the other spaces).

\begin{remark} \label{rm:sobolev}
We remark that, if $-1 < t \leq s < 1$ and $p>1$,
the space $W^{s,p}(\Gamma_h)$ is continuously imbedded in $W^{t,p}(\Gamma_h)$.
This follows directly from the definition.
\end{remark}

\begin{theorem}\label{teo:sobolev1}
If $-1\leq t \leq s \leq 1$, $q\geq p$ and $s-\frac{N-1}{p}\geq t-\frac{N-1}{q}$,
then $W^{s,p}(\Gamma_h)$ is continuously imbedded in $W^{t,q}(\Gamma_h)$.
The imbedding constant depends only on $s$, $t$, $p$, $q$ and on the $C^1$-norm of $h$.
\end{theorem}

In particular, it follows that if $N\leq3$ then $H^1(\Gamma_h)$ is continuously imbedded in $L^q(\Gamma_h)$ for every $q\geq1$.
The proof of the theorem follows from \cite[Theorem~1.4.4.1]{Gri} by a change of variables, and taking into account Remark~\ref{rm:sobolev}.
The following theorem, which follows from \cite[Theorem~1.5.1.2]{Gri}, deals with the trace operator on $\Gamma_h$.

\begin{theorem}\label{teo:sobolev2}
There exists a continuous linear operator $T:W^{1,p}(\Omega_h)\to W^{1-\frac1p,p}(\Gamma_h)$
such that $Tu=u|_{\Gamma_h}$ whenever $u$ is continuous on $\overline{\Omega}_h$.
The norm of $T$ is bounded by a constant depending only on $p$, $c_0$, and on the $C^1$-norm of $h$.
\end{theorem}

Denoting by $W^{1,p}_\#(\Omega_h)$ the space of functions $u\in W^{1,p}(\Omega_h)$
whose periodic extension to $\Omega_h^\#$ belongs to $W^{1,p}_{\rm loc}(\Omega_h^\#)$,
we have in particular that, if $u\in W^{1,p}_\#(\Omega_h)$, then $Tu\in W^{1-\frac1p,p}_\#(\Gamma_h)$.
\emph{Viceversa}, we have the following extension theorem.

\begin{theorem}\label{teo:sobolev3}
For every $\vartheta\in W^{1-\frac1p,p}_\#(\Gamma_h)$
there exists $u\in W^{1,p}_\#(\Omega_h)$ such that $Tu=\vartheta$ and
\begin{equation} \label{eq:sobolev}
\|u\|_{W^{1,p}(\Omega_h)}\leq C\|\vartheta\|_{W^{1-\frac1p,p}(\Gamma_h)},
\end{equation}
where $C$ depends only on $p$, $c_0$, and on the $C^1$-norm of $h$.
\end{theorem}

We now state the 3-dimensional version of \cite[Theorem~8.6]{FM}.

\begin{theorem}\label{teo:sobolev4}
Let $N=3$. For every $u\in W^{1,p}_\#(\Omega_h)$ and for $i=1,2$
$$
\Big\| \frac{\partial u}{\partial z_i}\nu_h^3 - \frac{\partial u}{\partial z_3}\nu_h^i \Big\|_{W^{-\frac1p,p}_\#(\Gamma_h)}
\leq C\|\nabla u\|_{L^p(\Omega_h;\R^3)}
$$
where $C$ depends only on $p$, $c_0$, and on the $C^1$-norm of $h$.
\end{theorem}

\begin{proof}
Assume $u\in C^2(\overline{\Omega}_h)$.
Given $\vphi\in W^{\frac1p,\frac{p}{p-1}}_\#(\Gamma_h)$
we consider an extension in $W^{1,\frac{p}{p-1}}_\#(\Omega_h)$ (still denoted by $\vphi$),
according to Theorem~\ref{teo:sobolev3}.
We may also assume, by increasing the constant in \eqref{eq:sobolev}, that $\vphi(x,0)=0$.
Then
\begin{align*}
\int_{\Gamma_h} \Bigl( \frac{\partial u}{\partial z_1}\nu_h^3 &- \frac{\partial u}{\partial z_3}\nu_h^1 \Bigr) \vphi \,d\mathcal{H}^2
= \int_{\Gamma_h} \vphi \Bigl( -\frac{\partial u}{\partial z_3},0,\frac{\partial u}{\partial z_1} \Bigr)\cdot\nu\,d\mathcal{H}^2 \\
&= \int_{\Omega_h} \div\Bigl( -\vphi\frac{\partial u}{\partial z_3},0,\vphi\frac{\partial u}{\partial z_1}\Bigr) \,dz
= \int_{\Omega_h} \nabla u\cdot \Bigl( \frac{\partial\vphi}{\partial z_3},0,-\frac{\partial\vphi}{\partial z_1}\Bigr) \,dz \\
& \leq \|\nabla u\|_{L^p(\Omega_h;\R^3)}\|\nabla\vphi\|_{L^{\frac{p}{p-1}}(\Omega_h;\R^3)}
\leq C \, \|\nabla u\|_{L^p(\Omega_h;\R^3)}\|\vphi\|_{W^{\frac1p,\frac{p}{p-1}}(\Gamma_h)}
\end{align*}
and this shows the claim in the case $i=1$. The case $i=2$ is similar,
and an approximation argument concludes the proof of the theorem.
\end{proof}

We conclude this section with two lemmas which will be used several times in the proof of Theorem~\ref{teo:locmin}.
The proof of the first one follows directly from the definition of the Gagliardo seminorm.

\begin{lemma} \label{lem:sobolev}
Let $p>1$ and let $u$ be a smooth function.
Then:
\begin{itemize}
  \item [(i)] if $a\in C^{0,\alpha}(\Gamma_h)$ with $\alpha>\frac1p$, then $\|ua\|_{W^{-\frac1p,p}(\Gamma_h)}\leq C \|a\|_{C^{0,\alpha}(\Gamma_h)}\|u\|_{W^{-\frac1p,p}(\Gamma_h)}$, for some constant $C$ depending only on $p$, $\alpha$ and on the $C^1$-norm of $h$;
  \item [(ii)] if $\Phi:\Gamma_h\to\Phi(\Gamma_h)$ is a $C^1$-diffeomorphism, then $\|u\circ\Phi^{-1}\|_{W^{-\frac1p,p}(\Phi(\Gamma_h))}\leq C \|u\|_{W^{-\frac1p,p}(\Gamma_h)}$, for some constant $C$ depending only on $p$ and on the $C^1$-norms of $\Phi$ and of $\Phi^{-1}$.
\end{itemize}

\end{lemma}

\begin{lemma}\label{lem:gagliardo}
Let $N\leq3$ and $\alpha>\frac12$.
If $\vphi\in H^1(\Gamma_h)$ and $u\in C^{0,\alpha}(\overline{\Omega}_h;\M^N)$, then
$$
\|\vphi u\|_{H^\frac12(\Gamma_h;\M^N)} \leq C \|\vphi\|_{H^1(\Gamma_h)} \|u\|_{C^{0,\alpha}(\overline{\Omega}_h;\M^N)}
$$
for some constant $C$ depending only on $\alpha$ and on the $C^1$-norm of $h$.
\end{lemma}

\begin{proof}
We can bound the Gagliardo $H^{\frac12}$-seminorm of $\vphi u$ as follows:
choosing $q>2$ such that $(2\alpha-1)q>2N-2$,
adding and subtracting the term $\vphi(z)u(w)$
and using H\"{o}lder inequality, we have
\begin{align*}
\bigl[ \vphi u \bigr]_{\frac12,2,\Gamma_h}^2
&\leq
 2\int_{\Gamma_h}\int_{\Gamma_h} \frac{ |\vphi(z)-\vphi(w)|^2 |u(w)|^2 }{ |z-w|^N } \,d\hn(z)\,d\hn(w) \nonumber\\
&\hspace{.5cm}+
 2\int_{\Gamma_h}\int_{\Gamma_h} \frac{ |\vphi(z)|^2 |u(z)-u(w)|^2 }{ |z-w|^N } \,d\hn(z)\,d\hn(w) \nonumber \\
&\leq
 2 \, \|u\|^2_{\infty} \|\vphi\|^2_{H^{\frac12}(\Gamma_h)}
 + 2 \|u\|^2_{C^{0,\alpha}}
  \int_{\Gamma_h}\int_{\Gamma_h} |\vphi(z)|^2 |z-w|^{2\alpha-N}\,d\hn(z)\,d\hn(w) \nonumber \\
&\leq
 2 \, \|u\|^2_{C^{0,\alpha}(\overline{\Omega}_h;\M^N)}
 \biggl[ \|\vphi\|_{H^{\frac12}(\Gamma_h)}^2 \nonumber \\
&\hspace{.5cm}
 +\|\vphi\|_{L^q(\Gamma_h)}^2\mathcal{H}^{N-1}(\Gamma_h)^{\frac2q}
 \Bigl( \int_{\Gamma_h}\int_{\Gamma_h} |z-w|^{\frac{q(2\alpha-N)}{q-2}} \,d\hn(x)\,d\hn(y) \Bigr)^{\frac{q-2}{q}} \biggr].
\end{align*}
Now the last integral is finite by the choice of $q$,
and the conclusion follows since $H^1(\Gamma_h)$ is continuously imbedded in $L^q(\Gamma_h)$ for every $q$.
\end{proof}

\subsection{Invertibility of the linear system appearing in Lemma~\ref{lem:algebra}}\label{sect:appdeterminante}

The final part of the second step in the proof of Lemma~\ref{lem:algebra}
requires to invert the relations determined by an $18\times18$ linear system
which we can write explicitly as
$$
\xi= M \sigma,
$$
where $\xi$ and $\sigma$ are the column vectors
\begin{align*}
\xi := \bigl(
\vartheta_{111}, \vartheta_{311}, \vartheta_{112}, \vartheta_{212}, \vartheta_{312}, \vartheta_{121},
\vartheta_{321}, \vartheta_{122}, \vartheta_{222}, &\vartheta_{322}, \vartheta_{131}, \vartheta_{331}, \\
&\vartheta_{132}, \vartheta_{232}, \vartheta_{332}, \eta_{13}, \eta_{23}, \eta_{33}
\bigr)^T,
\end{align*}
\begin{align*}
\sigma := \bigl(
\sigma_{111}, \sigma_{121}, \sigma_{131}, \sigma_{221}, \sigma_{231}, \sigma_{331},
\sigma_{112}, \sigma_{122}, \sigma_{132}, &\sigma_{222}, \sigma_{232}, \sigma_{332}, \\
&\sigma_{113}, \sigma_{123}, \sigma_{133}, \sigma_{223}, \sigma_{233}, \sigma_{333} \bigr)^T,
\end{align*}
and $M$ is the matrix
$$
\scriptscriptstyle
\left(
\begin{array}{cccccccccccccccccc}
 \nu_g^3 & 0 & -\nu_g^1 & 0 & 0 & 0 & 0 & 0 & 0 & 0 & 0 & 0 & 0 & 0 & 0 & 0 & 0 & 0 \\
 0 & 0 & \nu_g^3 & 0 & 0 & -\nu_g^1 & 0 & 0 & 0 & 0 & 0 & 0 & 0 & 0 & 0 & 0 & 0 & 0 \\
 0 & \nu_g^3 & -\nu_g^2 & 0 & 0 & 0 & 0 & 0 & 0 & 0 & 0 & 0 & 0 & 0 & 0 & 0 & 0 & 0 \\
 0 & 0 & 0 & \nu_g^3 & -\nu_g^2 & 0 & 0 & 0 & 0 & 0 & 0 & 0 & 0 & 0 & 0 & 0 & 0 & 0 \\
 0 & 0 & 0 & 0 & \nu_g^3 & -\nu_g^2 & 0 & 0 & 0 & 0 & 0 & 0 & 0 & 0 & 0 & 0 & 0 & 0 \\
 0 & 0 & 0 & 0 & 0 & 0 & \nu_g^3 & 0 & -\nu_g^1 & 0 & 0 & 0 & 0 & 0 & 0 & 0 & 0 & 0 \\
 0 & 0 & 0 & 0 & 0 & 0 & 0 & 0 & \nu_g^3 & 0 & 0 & -\nu_g^1 & 0 & 0 & 0 & 0 & 0 & 0 \\
 0 & 0 & 0 & 0 & 0 & 0 & 0 & \nu_g^3 & -\nu_g^2 & 0 & 0 & 0 & 0 & 0 & 0 & 0 & 0 & 0 \\
 0 & 0 & 0 & 0 & 0 & 0 & 0 & 0 & 0 & \nu_g^3 & -\nu_g^2 & 0 & 0 & 0 & 0 & 0 & 0 & 0 \\
 0 & 0 & 0 & 0 & 0 & 0 & 0 & 0 & 0 & 0 & \nu_g^3 & -\nu_g^2 & 0 & 0 & 0 & 0 & 0 & 0 \\
 0 & 0 & 0 & 0 & 0 & 0 & 0 & 0 & 0 & 0 & 0 & 0 & \nu_g^3 & 0 & -\nu_g^1 & 0 & 0 & 0 \\
 0 & 0 & 0 & 0 & 0 & 0 & 0 & 0 & 0 & 0 & 0 & 0 & 0 & 0 & \nu_g^3 & 0 & 0 & -\nu_g^1 \\
 0 & 0 & 0 & 0 & 0 & 0 & 0 & 0 & 0 & 0 & 0 & 0 & 0 & \nu_g^3 & -\nu_g^2 & 0 & 0 & 0 \\
 0 & 0 & 0 & 0 & 0 & 0 & 0 & 0 & 0 & 0 & 0 & 0 & 0 & 0 & 0 & \nu_g^3 & -\nu_g^2 & 0 \\
 0 & 0 & 0 & 0 & 0 & 0 & 0 & 0 & 0 & 0 & 0 & 0 & 0 & 0 & 0 & 0 & \nu_g^3 & -\nu_g^2 \\
 0 & 0 & a_1 & 0 & b_1 & c_1 & 0 & 0 & d_1 & 0 & e_1 & f_1 & 0 & 0 & g_1 & 0 & h_1 & i_1 \\
 0 & 0 & a_2 & 0 & b_2 & c_2 & 0 & 0 & d_2 & 0 & e_2 & f_2 & 0 & 0 & g_2 & 0 & h_2 & i_2 \\
 0 & 0 & a_3 & 0 & b_3 & c_3 & 0 & 0 & d_3 & 0 & e_3 & f_3 & 0 & 0 & g_3 & 0 & h_3 & i_3
\end{array}
\right)
$$
The coefficients in the last three rows of $M$ are defined by
\begin{align*}
\textstyle
a_j := \sum_{k=1}^3 C_{jk11}\nu_g^k, \quad
b_j := \sum_{k=1}^3 C_{jk12}\nu_g^k, \quad
c_j := \sum_{k=1}^3 C_{jk13}\nu_g^k, \\
\textstyle
d_j := \sum_{k=1}^3 C_{jk21}\nu_g^k, \quad
e_j := \sum_{k=1}^3 C_{jk22}\nu_g^k, \quad
f_j := \sum_{k=1}^3 C_{jk23}\nu_g^k, \\
\textstyle
g_j := \sum_{k=1}^3 C_{jk31}\nu_g^k, \quad
h_j := \sum_{k=1}^3 C_{jk32}\nu_g^k, \quad
i_j := \sum_{k=1}^3 C_{jk33}\nu_g^k,
\end{align*}
for $j=1,2,3$,
so that the corresponding equations are exactly the equalities \eqref{matrice2}.
In order to invert the relations determined by the previous system,
we claimed that the determinant of $M$ equals $(\nu_g^3)^{12}\det Q_g$,
where $Q_g$ is the $3\times3$ matrix defined by \eqref{matrice}.

We present here the \emph{Mathematica} code which allows us to check this equality.
We first define the $18\times18$ matrix $M$: here the variables \pmb{n1}, \pmb{n2} and \pmb{n3} stand for the components $\nu_g^1,\nu_g^2,\nu_g^3$ of the normal vector, and the variables \pmb{Cijhk} for the coefficients $C_{ijhk}$ of the tensor.
We then define the matrix $Q_g$ introduced in \eqref{matrice}, whose entries are indicated by \pmb{qij},
and we compute its determinant (multiplied by $(\nu_g^3)^{12}$).
Finally we evaluate the difference between the determinant of $M$ and $(\nu_g^3)^{12}\det Q_g$, which turns out to be zero.

The \emph{Mathematica} code is the following.

\begin{doublespace}
{\scriptsize
\noindent\(\pmb{M=\left(
\begin{array}{cccccccccccccccccc}
 \text{n3} & 0 & -\text{n1} & 0 & 0 & 0 & 0 & 0 & 0 & 0 & 0 & 0 & 0 & 0 & 0 & 0 & 0 & 0 \\
 0 & 0 & \text{n3} & 0 & 0 & -\text{n1} & 0 & 0 & 0 & 0 & 0 & 0 & 0 & 0 & 0 & 0 & 0 & 0 \\
 0 & \text{n3} & -\text{n2} & 0 & 0 & 0 & 0 & 0 & 0 & 0 & 0 & 0 & 0 & 0 & 0 & 0 & 0 & 0 \\
 0 & 0 & 0 & \text{n3} & -\text{n2} & 0 & 0 & 0 & 0 & 0 & 0 & 0 & 0 & 0 & 0 & 0 & 0 & 0 \\
 0 & 0 & 0 & 0 & \text{n3} & -\text{n2} & 0 & 0 & 0 & 0 & 0 & 0 & 0 & 0 & 0 & 0 & 0 & 0 \\
 0 & 0 & 0 & 0 & 0 & 0 & \text{n3} & 0 & -\text{n1} & 0 & 0 & 0 & 0 & 0 & 0 & 0 & 0 & 0 \\
 0 & 0 & 0 & 0 & 0 & 0 & 0 & 0 & \text{n3} & 0 & 0 & -\text{n1} & 0 & 0 & 0 & 0 & 0 & 0 \\
 0 & 0 & 0 & 0 & 0 & 0 & 0 & \text{n3} & -\text{n2} & 0 & 0 & 0 & 0 & 0 & 0 & 0 & 0 & 0 \\
 0 & 0 & 0 & 0 & 0 & 0 & 0 & 0 & 0 & \text{n3} & -\text{n2} & 0 & 0 & 0 & 0 & 0 & 0 & 0 \\
 0 & 0 & 0 & 0 & 0 & 0 & 0 & 0 & 0 & 0 & \text{n3} & -\text{n2} & 0 & 0 & 0 & 0 & 0 & 0 \\
 0 & 0 & 0 & 0 & 0 & 0 & 0 & 0 & 0 & 0 & 0 & 0 & \text{n3} & 0 & -\text{n1} & 0 & 0 & 0 \\
 0 & 0 & 0 & 0 & 0 & 0 & 0 & 0 & 0 & 0 & 0 & 0 & 0 & 0 & \text{n3} & 0 & 0 & -\text{n1} \\
 0 & 0 & 0 & 0 & 0 & 0 & 0 & 0 & 0 & 0 & 0 & 0 & 0 & \text{n3} & -\text{n2} & 0 & 0 & 0 \\
 0 & 0 & 0 & 0 & 0 & 0 & 0 & 0 & 0 & 0 & 0 & 0 & 0 & 0 & 0 & \text{n3} & -\text{n2} & 0 \\
 0 & 0 & 0 & 0 & 0 & 0 & 0 & 0 & 0 & 0 & 0 & 0 & 0 & 0 & 0 & 0 & \text{n3} & -\text{n2} \\
 0 & 0 & \text{a1} & 0 & \text{b1} & \text{c1} & 0 & 0 & \text{d1} & 0 & \text{e1} & \text{f1} & 0 & 0 & \text{g1} & 0 & \text{h1} & \text{i1} \\
 0 & 0 & \text{a2} & 0 & \text{b2} & \text{c2} & 0 & 0 & \text{d2} & 0 & \text{e2} & \text{f2} & 0 & 0 & \text{g2} & 0 & \text{h2} & \text{i2} \\
 0 & 0 & \text{a3} & 0 & \text{b3} & \text{c3} & 0 & 0 & \text{d3} & 0 & \text{e3} & \text{f3} & 0 & 0 & \text{g3} & 0 & \text{h3} & \text{i3}
\end{array}
\right);}\)
}
\end{doublespace}

\begin{doublespace}
\noindent\(\pmb{\text{DM}=\text{Det}[M];}\)
\end{doublespace}

\begin{doublespace}
\noindent\(\pmb{\text{a1}=\text{C1111} \text{n1} + \text{C1211} \text{n2}+\text{C1311} \text{n3};}\\
\pmb{\text{b1}=\text{C1112} \text{n1} + \text{C1212} \text{n2}+\text{C1312} \text{n3};}\\
\pmb{\text{c1}=\text{C1113} \text{n1} + \text{C1213} \text{n2}+\text{C1313} \text{n3};}\\
\pmb{\text{d1}=\text{C1121} \text{n1} + \text{C1221} \text{n2}+\text{C1321} \text{n3};}\\
\pmb{\text{e1}=\text{C1122} \text{n1} + \text{C1222} \text{n2}+\text{C1322} \text{n3};}\\
\pmb{\text{f1}=\text{C1123} \text{n1} + \text{C1223} \text{n2}+\text{C1323} \text{n3};}\\
\pmb{\text{g1}=\text{C1131} \text{n1} + \text{C1231} \text{n2}+\text{C1331} \text{n3};}\\
\pmb{\text{h1}=\text{C1132} \text{n1} + \text{C1232} \text{n2}+\text{C1332} \text{n3};}\\
\pmb{\text{i1}=\text{C1133} \text{n1} + \text{C1233} \text{n2}+\text{C1333} \text{n3};}\\
\pmb{\text{a2}=\text{C2111} \text{n1} + \text{C2211} \text{n2}+\text{C2311} \text{n3};}\\
\pmb{\text{b2}=\text{C2112} \text{n1} + \text{C2212} \text{n2}+\text{C2312} \text{n3};}\\
\pmb{\text{c2}=\text{C2113} \text{n1} + \text{C2213} \text{n2}+\text{C2313} \text{n3};}\\
\pmb{\text{d2}=\text{C2121} \text{n1} + \text{C2221} \text{n2}+\text{C2321} \text{n3};}\\
\pmb{\text{e2}=\text{C2122} \text{n1} + \text{C2222} \text{n2}+\text{C2322} \text{n3};}\\
\pmb{\text{f2}=\text{C2123} \text{n1} + \text{C2223} \text{n2}+\text{C2323} \text{n3};}\\
\pmb{\text{g2}=\text{C2131} \text{n1} + \text{C2231} \text{n2}+\text{C2331} \text{n3};}\\
\pmb{\text{h2}=\text{C2132} \text{n1} + \text{C2232} \text{n2}+\text{C2332} \text{n3};}\\
\pmb{\text{i2}=\text{C2133} \text{n1} + \text{C2233} \text{n2}+\text{C2333} \text{n3};}\\
\pmb{\text{a3}=\text{C3111} \text{n1} + \text{C3211} \text{n2}+\text{C3311} \text{n3};}\\
\pmb{\text{b3}=\text{C3112} \text{n1} + \text{C3212} \text{n2}+\text{C3312} \text{n3};}\\
\pmb{\text{c3}=\text{C3113} \text{n1} + \text{C3213} \text{n2}+\text{C3313} \text{n3};}\\
\pmb{\text{d3}=\text{C3121} \text{n1} + \text{C3221} \text{n2}+\text{C3321} \text{n3};}\\
\pmb{\text{e3}=\text{C3122} \text{n1} + \text{C3222} \text{n2}+\text{C3322} \text{n3};}\\
\pmb{\text{f3}=\text{C3123} \text{n1} + \text{C3223} \text{n2}+\text{C3323} \text{n3};}\\
\pmb{\text{g3}=\text{C3131} \text{n1} + \text{C3231} \text{n2}+\text{C3331} \text{n3};}\\
\pmb{\text{h3}=\text{C3132} \text{n1} + \text{C3232} \text{n2}+\text{C3332} \text{n3};}\\
\pmb{\text{i3}=\text{C3133} \text{n1} + \text{C3233} \text{n2}+\text{C3333} \text{n3};}\)
\end{doublespace}

\begin{doublespace}
\noindent\(\pmb{\text{q11}=\text{C1111} \text{n1} \text{n1} + \text{C1212} \text{n2} \text{n2} + \text{C1313} \text{n3} \text{n3} + (\text{C1112}
+ \text{C1211}) \text{n1} \text{n2} + }\\
\pmb{(\text{C1113}+\text{C1311}) \text{n1} \text{n3} + (\text{C1213}+\text{C1312}) \text{n2} \text{n3};}\\
\pmb{\text{q12}=\text{C1121} \text{n1} \text{n1} + \text{C1222} \text{n2} \text{n2} + \text{C1323} \text{n3} \text{n3} + (\text{C1122} + \text{C1221})
\text{n1} \text{n2} + }\\
\pmb{(\text{C1123}+\text{C1321}) \text{n1} \text{n3} + (\text{C1223}+\text{C1322}) \text{n2} \text{n3};}\\
\pmb{\text{q13}=\text{C1131} \text{n1} \text{n1} + \text{C1232} \text{n2} \text{n2} + \text{C1333} \text{n3} \text{n3} + (\text{C1132} + \text{C1231})
\text{n1} \text{n2} + }\\
\pmb{(\text{C1133}+\text{C1331}) \text{n1} \text{n3} + (\text{C1233}+\text{C1332}) \text{n2} \text{n3};}\\
\pmb{\text{q21}=\text{C2111} \text{n1} \text{n1} + \text{C2212} \text{n2} \text{n2} + \text{C2313} \text{n3} \text{n3} + (\text{C2112} + \text{C2211})
\text{n1} \text{n2} + }\\
\pmb{(\text{C2113}+\text{C2311}) \text{n1} \text{n3} + (\text{C2213}+\text{C2312}) \text{n2} \text{n3};}\\
\pmb{\text{q22}=\text{C2121} \text{n1} \text{n1} + \text{C2222} \text{n2} \text{n2} + \text{C2323} \text{n3} \text{n3} + (\text{C2122} + \text{C2221})
\text{n1} \text{n2} + }\\
\pmb{(\text{C2123}+\text{C2321}) \text{n1} \text{n3} + (\text{C2223}+\text{C2322}) \text{n2} \text{n3};}\\
\pmb{\text{q23}=\text{C2131} \text{n1} \text{n1} + \text{C2232} \text{n2} \text{n2} + \text{C2333} \text{n3} \text{n3} + (\text{C2132} + \text{C2231})
\text{n1} \text{n2} + }\\
\pmb{(\text{C2133}+\text{C2331}) \text{n1} \text{n3} + (\text{C2233}+\text{C2332}) \text{n2} \text{n3};}\\
\pmb{\text{q31}=\text{C3111} \text{n1} \text{n1} + \text{C3212} \text{n2} \text{n2} + \text{C3313} \text{n3} \text{n3} + (\text{C3112} + \text{C3211})
\text{n1} \text{n2} + }\\
\pmb{(\text{C3113}+\text{C3311}) \text{n1} \text{n3} + (\text{C3213}+\text{C3312}) \text{n2} \text{n3};}\\
\pmb{\text{q32}=\text{C3121} \text{n1} \text{n1} + \text{C3222} \text{n2} \text{n2} + \text{C3323} \text{n3} \text{n3} + (\text{C3122} + \text{C3221})
\text{n1} \text{n2} + }\\
\pmb{(\text{C3123}+\text{C3321}) \text{n1} \text{n3} + (\text{C3223}+\text{C3322}) \text{n2} \text{n3};}\\
\pmb{\text{q33}=\text{C3131} \text{n1} \text{n1} + \text{C3232} \text{n2} \text{n2} + \text{C3333} \text{n3} \text{n3} + (\text{C3132} + \text{C3231})
\text{n1} \text{n2} + }\\
\pmb{(\text{C3133}+\text{C3331}) \text{n1} \text{n3} + (\text{C3233}+\text{C3332}) \text{n2} \text{n3};}\)
\end{doublespace}

\begin{doublespace}
\noindent\(\pmb{Q=\left(
\begin{array}{ccc}
 \text{q11} & \text{q12} & \text{q13} \\
 \text{q21} & \text{q22} & \text{q23} \\
 \text{q31} & \text{q32} & \text{q33}
\end{array}
\right);}\)
\end{doublespace}

\begin{doublespace}
\noindent\(\pmb{\text{DQ}=\text{n3}^{12}\text{Det}[Q];}\)
\end{doublespace}

\begin{doublespace}
\noindent\(\pmb{\text{ExpandAll}[\text{DQ}]-\text{ExpandAll}[\text{DM}]}\)
\end{doublespace}

\end{section}


\bigskip
\bigskip
\noindent
\textsc{Acknowledgments.}
Part of this work was carried out while visiting Filippo Cagnetti, whom I thank for the kind ospitality, at the Instituto Superior T\'{e}cnico in Lisbon. I am also very grateful to Massimiliano Morini for multiple valuable discussions.


\bigskip
{
}

\end{document}